\documentclass{amsart}

\usepackage{amsmath,amssymb,amsthm}

\newtheorem{theo}{Theorem}[section]
\newtheorem{lemm}[theo]{Lemma}
\newtheorem{corr}[theo]{Corollary}

\numberwithin{equation}{section}

\theoremstyle{definition}
\newtheorem{defi}[theo]{Definition}

\newtheorem{rema}[theo]{Remark}

\newcommand{\ba}{\mathbf{a}}
\newcommand{\bb}{\mathbf{b}}

\newcommand{\bff}{\mathbf{f}}

\newcommand{\bn}{\mathbf{n}}
\newcommand{\bq}{\mathbf{q}}
\newcommand{\bu}{\mathbf{u}}
\newcommand{\bv}{\mathbf{v}}

\newcommand{\bD}{\mathbf{D}}

\newcommand{\bI}{\mathbf{I}}

\newcommand{\bK}{\mathbf{K}}

\newcommand{\bS}{\mathbf{S}}
\newcommand{\bT}{\mathbf{T}}

\newcommand{\dv}{{\rm div}\,}

\newcommand{\BR}{\mathbb{R}}

\newcommand{\CR}{\mathcal{R}}

\newcommand{\pd}{\partial}

\begin{document}

\title[Compressible-Incompressible Two-Phase Flows]
{Compressible-Incompressible Two-Phase Flows\\
with Phase Transition: Model Problem}

\author{KEIICHI WATANABE}
\address
{Department of Pure and Applied Mathematics,
Graduate School of Fundamental Science and Engineering,
Waseda University, 3-4-1 Ookubo,
Shinjuku-ku, Tokyo, 169-8555, Japan}		
\email{keiichi-watanabe@akane.waseda.jp}
\thanks{}

\begin{abstract}
We study the compressible and incompressible two-phase flows 
separated by a sharp interface with a phase transition and a surface tension.
In particular, we consider the problem in $\mathbb{R}^N$, and
the Navier-Stokes-Korteweg equations is used in the upper domain and
the Navier-Stokes equations is used in the lower domain.
We prove the existence of $\mathcal{R}$-bounded solution 
operator families for a resolvent problem arising from its model problem.
According to Shibata \cite{GS2014}, the regularity
of $\rho_+$ is $W^1_q$ in space, but to solve the kinetic
equation: $\bu_\Gamma\cdot\bn_t = [[\rho\bu]]\cdot\bn_t
/[[\rho]]$ on $\Gamma_t$ we need $W^{2-1/q}_q$ regularity of $\rho_+$
on $\Gamma_t$, which means the regularity loss.
Since the regularity of $\rho_+$ dominated by the Navier-Stokes-Korteweg 
equations is $W^3_q$ in space,
we eliminate the problem by using the Navier-Stokes-Korteweg equations
instead of the compressible Navier-Stokes equations.
\end{abstract}

\maketitle

\section{Introduction}
This paper deals with compressible-incompressible 
two-phase flows separated by a sharp interface. 
In particular, we consider the phase transition at the interface.
Our problem is formulated as follows:  Let $\Omega_{t+}$ and 
$\Omega_{t-}$ be two time dependent domains, and
$\Gamma_t$ be the common boundary of $\Omega_{t+}$ and $\Omega_{t-}$.
We assume that $\Omega_{t+} \cap \Omega_{t-} = \emptyset$
and $\Omega_{t+} \cup \Gamma_t \cup \Omega_{t-} = \BR^N$, where
$\BR^N$ denotes the $N$-dimensional Euclidean space. 
Furthermore, we assume that $\Omega_{t+}$ and $\Omega_{t-}$ 
are occupied by a compressible viscous fluid
and an incompressible viscous fluid, respectively.
For example, $\Omega_{t-}$ is corresponding to an ocean of 
infinite extent without bottom, $\Omega_{t+}$ the atmosphere,
and $\Gamma_t$ the surface of the ocean.
Let $\bn_t$ be the unit outer normal to
$\Gamma_t$ pointed from $\Omega_{t+}$ to $\Omega_{t-}$. 
For any $x_0 \in \Gamma_t$ and function $f$ defined on 
$\Omega_{t+} \cup \Omega_{t-}$, we set
\begin{align*}
[[f]](x_0, t) = \lim_{\substack{x\to x_0 \\ x \in \Omega_{t-}}}f(x, t)
- \lim_{\substack{x\to x_0 \\ x \in \Omega_{t+}}}f(x, t),
\end{align*}
which is the jump quantity of $f$ across $\Gamma_t$.
Let $\dot\Omega_t = \Omega_{t+} \cup \Omega_{t-}$, 
and for any function $f$ defined on $\dot\Omega_t$,
we write $f_\pm = f|_{\Omega_{t\pm}}$. 
In the following, we use the following symbols:
\allowdisplaybreaks{
\begin{alignat*}2
\bullet\enskip& \rho:&\enskip &\dot{\Omega}_t
\to\overline{\mathbb{R}_+}
=[0,\infty)\enskip \text{ is the density}, \\
\bullet\enskip& \mathbf{u}:&\enskip &\dot{\Omega}_t
\to\mathbb{R}^N\enskip\text{ the velocity field}, \\
\bullet\enskip& \bu_\Gamma: &\enskip
& \Gamma_t \to \mathbb{R}^N\enskip\text{the interfacial velocity field}, \\
\bullet\enskip& \pi:&\enskip &\dot{\Omega}_t
\to\mathbb{R}\enskip\text{ the pressure field}, \\
\bullet\enskip& \mathbf{T}:&\enskip&\dot{\Omega}_t
\to\{A\in GL_N(\mathbb{R})\ |{}^\top \!A=A\}
\enskip\text{ the stress tensor field}, \\
\bullet\enskip&\theta:&\enskip&\dot{\Omega}_t
\to(0,\infty) \enskip\text{the thermal field},\\
\bullet\enskip& e:&\enskip&\dot{\Omega}_t
\to\overline{\mathbb{R}_+}
\enskip\text{the internal energy density}, \\
\bullet\enskip& \eta:&\enskip&\dot{\Omega}_t
\to\mathbb{R} \enskip\text{the entropy density}, \\
\bullet\enskip& \psi:&\enskip&\dot{\Omega}_t
\to\mathbb{R} \enskip\text{the Helmholtz free energy function}, \\
\bullet\enskip& \mathbf{q}:&\enskip&\dot{\Omega}_t
\to\mathbb{R}^N \enskip\text{the energy flux}, \\
\bullet\enskip& \mathbf{f}:&\enskip&\dot{\Omega}_t
\to\mathbb{R}^N 
\enskip\text{the external body force per unit mass}, \\
\bullet\enskip& r:&\enskip&\Omega
\to\mathbb{R} \enskip\text{the heat supply}.
\end{alignat*}	
}
Here and in the following, ${}^\top\! M$
denotes the transposed $M$.  
And then, our problem is:
\begin{alignat}2
&\left\{\begin{aligned}\label{10000}
&\pd_t\rho + \dv(\rho\bu) = 0, \\
&\pd_t(\rho\bu) + \dv(\rho\bu\otimes\bu) - \dv\bT = \rho\bff,\\
&\pd_t(\frac{\rho}{2}|\bu|^2 + \rho e) + \dv((\frac{\rho}{2}|\bu|^2
+\rho e)\bu) - \dv(\bT\bu - \mathbf{q})
= \rho\bff\cdot\bu + \rho r, 
\end{aligned}\right.
&\quad &\text{for  $x\in \dot\Omega_t$, $t>0$},	
\intertext{subject to the interface conditions:}
\label{40000}
&\left\{\begin{aligned}
&[[\bu]]= \j[[1/\rho]]\bn_t,\quad
\j[[\mathbf{u}]]-[[\mathbf{T}]]\mathbf{n}_t
=-\sigma H_{\Gamma_t}\mathbf{n}_t,\\
&[[\theta]]=0, \quad \j\theta[[\eta]]-[[d\nabla\theta]]\cdot\mathbf{n}_t=0,\\
&\cfrac{\j^2}{2}\Big[\Big[\cfrac{1}{\rho^2}\Big]\Big]+ [[\psi]]
-\Big[\Big[\cfrac{1}{\rho}(\mathbf{Tn}_t\cdot\mathbf{n}_t)\Big]\Big]=0\\
&V_t=\mathbf{u}_\Gamma\cdot\mathbf{n}_t
=\cfrac{[[\rho\mathbf{u}]]\cdot\mathbf{n}_t}{[[\rho]]},\\
&(\nabla\rho_+)\cdot\mathbf{n}_t=0.	
\end{aligned}\right.
&\quad &\text{for  $x\in\Gamma_{t}$, $t>0$,}
\end{alignat}
where 
$\Omega_{t\pm}$ and $\Gamma_t$ are given by
\begin{align*}
\Omega_{t\pm}=&\{(x',x_N)\in \mathbb{R}\times\mathbb{R}^{N-1}
\mid \pm(x_N-h(x',t))>0,\enskip t\geq 0 \},\\
\Gamma_{t}=&\{x\in\mathbb{R}^N\mid x_N=h(x',t)
\quad\text{ for $x'\in\mathbb{R}^{N-1}$} \},
\end{align*}
respectively, with unknown function $h(x',t)$.
Above, 
$H_{\Gamma_t}$ is the $N-1$ times mean curvature of 
$\Gamma_t$,  
$\sigma$ a positive constant describing 
the coefficient of the surface tension, 
and $V_t$ the velocity of evolution of 
$\Gamma_{t}$ with respect to $\mathbf{n}_t$.
Furthermore, $\pd_t = \pd/\pd t$, $\pd_i = \pd/\pd x_i$, 
for any matrix field $\mathbf{K}$ with $(i,j)^{\rm th}$
component $K_{ij}$,  the quantity ${\rm div}\,\mathbf{K}$
is the $N$-vector with $i^{\rm th}$ component
$\sum_{j=1}^N\partial_jK_{ij}$, and for any $N$ 
vectors or $N$ vector fields, $\ba = {}^\top(a_1, \ldots, a_N)$, 
$\bb = {}^\top(b_1, \ldots, b_N)$,  we set 
$\mathbf{a}\cdot\mathbf{b}=\sum_{j=1}^{N}a_jb_j$,
$\dv\ba = \sum_{j=1}^N\pd_ja_j$, and $\ba\otimes\bb$ denotes
the $N\times N$ matrix with $(i,j)^{\rm th}$ component $a_ib_j$. 
The interface condition \eqref{40000} 
is explained in more detail in Sect. \ref{modeling} below. 
To describe the motion of incompressible viscous flow occupying $\Omega_{t-}$
we use the usual Navier-Stokes equations, and so we set
\begin{align}\label{incompressible-assumption}
\rho_- =\rho_{*-}, \quad
\bT_- =\mu_-\bD(\bu_-)-\pi_-\bI,
\quad \mathbf{q}_- = -d_-\nabla\theta_-,
\end{align}
where $\rho_{*-}$ is a positive constant describing the mass density of 
the reference body $\Omega_{t-}$, $\mu_-$ is 
the viscosity coefficient,  
$\bD(\bu) = (1/2)(\nabla\bu + {}^\top\nabla\bu)$,
is the deformation
tensor with $(i, j)^{\rm th}$ element $D_{ij}(\bu)=\pd_iu_j+ \pd_ju_{i}$
for $\pd_i = \pd/\pd x_i$ and $\bu={}^\top(u_1,\ldots, u_N)$, 
and $d$ is the coefficient of the heat flux. 
In particular, the equation of mass conservation: 
$\pd_t\rho_- + \dv(\rho_-\bu_-) = 0$  leads to 
$\dv \bu_-=0$ in $\Omega_{t-}$. 

On the other hand, to describe the motion of a compressible
viscous fluid occupying 
$\Omega_{t+}$, we adopt the Navier-Stokes-Korteweg
tensor of the following form: $\bT_+ = \bS_+ + \bK_+ - \pi_+\bI$
with 
\begin{align}\label{compressible-assumption}\begin{split}
\bS_+ &= \mu_+\bD(\bu_+) + (\nu_+-\mu_+)\dv\bu_+\bI-\pi_+, \\
\bK_+ & = \Bigl(\frac{\kappa_++\rho_+\kappa'_+}{2}|\nabla\rho_+|^2
+ \kappa_+\rho_+\Delta\rho_+\Bigr)\bI - \kappa_+
\nabla\rho_+\otimes\nabla\rho_+.
\end{split}\end{align}
Here, $\mathbf{K}_+$ is called the Korteweg tensor
(cf. Dunn and Serrin \cite{DS1985} and 
Kotschote \cite{Kotschote2010}). According to Dunn and Serrin
\cite{DS1985}, in view of the second law of 
thermodynamics the energy flux includes not only 
a classical contribution corresponding to 
the Fourier law but also a nonclassical contribution, 
which we now call the interstitial working. In this sense, 
the energy flux $\mathbf{q}_+$ is given by
\begin{align}\label{energy-flux}
\mathbf{q}_+=-d_+\nabla\theta_+ +(\kappa_+\rho_+\,{\rm div}\,\mathbf{u}_+)
\nabla\rho_+
\end{align}
when we use the Navier-Stokes-Korteweg equations,
where $d_+$ is the coefficient of the heat flux.

We assume that $\mu_+=\mu_+(\rho_+,\theta_+)$,
$\nu_+=\nu_+(\rho_+,\theta_+)$, $\kappa_+=\kappa_+(\rho_+,\theta_+)$,
$e_+=e_+(\rho_+,\theta_+)$, $d_+=d_+(\rho_+,\theta_+)$
are positive $C^\infty$ functions with respect to 
$(\rho_+,\theta_+)\in (0,\infty)\times(0,\infty)$,
and $\psi_+(\rho_+,\theta_+)$ and $\eta_+(\rho_+,\theta_+)$ are 
real valued $C^\infty$ functions with respect to 
$(\rho_+,\theta_+)\in(0,\infty)\times(0,\infty)$,
while $\mu_{-}=\mu_{-}(\theta_-)$,
$e_-=e_-(\theta_-)$, $d_-=d_-(\theta_-)$
are positive $C^\infty$ functions with respect to
$\theta_-\in (0,\infty)$, and $\psi_-(\theta_-)$ and
$\eta_-(\theta_-)$ are real valued $C^\infty$ functions
with respect to $\theta_-\in (0,\infty)$.
Moreover, we  assume that $\pd e_+/\pd \theta_+ > 0$, 
$e_-' > 0$, and 
$\pi_+$ is given by
$\pi_+=P_+(\rho_+,\theta_+)$, where $P_+$ is some
$C^\infty$ function with respect to
$(\rho_+,\theta_+)\in (0,\infty)\times(0,\infty)$.	

We now explain why the Navier-Stokes-Korteweg equations
is used in $\Omega_{t+}$ to describe the
motion of the compressible viscous fluid.
Shibata \cite{Shibata2016} used the Navier-Stokes-Fourier
equations for $\Omega_{t+}$, that is $\bT_+ = \bS_+$ and
$\mathbf{q}_+=-d_+\nabla\theta_+$, 
to formulate the compressible-incompressible two-phase flows
separated by a sharp interface with the phase transition.
He proved the existence of $\CR$ bounded solution operators
for the model problem that derives the maximal $L_p\mathchar`-L_q$
regularity of solutions to the linearized equations
automatically with the help of Weis's operator
valued Fourier multiplier theorem \cite{Weis2001}. 
According to Shibata \cite{GS2014}, the regularity
of $\rho_+$ is $W^1_q$ in space, but to solve the kinetic
equation: $\bu_\Gamma\cdot\bn_t = [[\rho\bu]]\cdot\bn_t
/[[\rho]]$ on $\Gamma_t$ we need $W^{2-1/q}_q$ regularity of $\rho_+$
on $\Gamma_t$, which means 
the regularity loss. On the other hand, 
the regularity of $\rho_+$ dominated by the Navier-Stokes-Korteweg 
equations is $W^3_q$ in $\Omega_{t+}$
(cf. Kotschote \cite{Kotschote2008,Kotschote2010} and Saito \cite{Saito}),
which is enough to solve the kinetic equation. In addition, 
quite recently Gorban and Karlin \cite{GK2017} proved that
the Navier-Stokes-Korteweg equations is implied by 
the Boltzmann equation that describes the
statistical behavior of a gas. In this sense, to use
the Navier-Stokes-Korteweg equations to describe the motion of 
compressible viscous fluid flow is meaningful.
Furthermore, we would like to add some comments about 
the Navier-Stokes-Korteweg equations. 
More than one hundred years ago, Korteweg
\cite{Korteweg} derived the Navier-Stokes-Korteweg equations
to describe the two phase problem with diffused
interface like liquid and vapor flows with  phase transition,
which was based on the gradient theory for
the interface developed by van der Waals \cite{vanderWaals}.
In 1985, Dunn and Serrin \cite{DS1985}
studied the Navier-Stokes-Korteweg 
equations with the second law of thermodynamics. 
As equations describing the two-phase flows with diffused interface, 
we also know the Navier-Stokes-Allen-Chan equations and 
the Navier-Stokes-Chan-Hilliard equations (cf. \cite{CH1958}), but 
they can be reduced to the Navier-Stokes-Korteweg equations,
which is quite recently proved by 
Freisth\"uler and Kotchote \cite{FK2017}.
Thus, our formulation \eqref{10000} and \eqref{40000} includes 
the following situation:  The ocean and atmosphere are separated by
a sharp interface and on this interface the phase transition occurs.
In addition, the atmosphere part is two-phase flows with diffused interface
like the mixture of gas and ice. Thus, we totally treat
three phase problem, and liquid and gas-solid  are separated by
a sharp interface with phase transition and gas-solid part has
diffused interface with phase transition. 

Finally, let us mention related results about the
initial-boundary value problem for the Navier-Stokes-Korteweg equations.
In 2008 and 2010, Kotschote \cite{Kotschote2008, Kotschote2010}
proved the existence and uniqueness of local strong 
solutions for an isothermal and non-isothermal model 
of capillary compressible fluids derived by Dunn and Serrin \cite{DS1985}.
Recently, Tsuda \cite{Tsuda2016} studied the existence and stability of
time periodic solution to the Navier-Stokes-Korteweg equations in $\BR^3$, and 
Saito \cite{Saito} proved the existence of 
$\CR$ bounded solution operators for the model problem
of the Navier-Stokes-Korteweg equations with free boundary conditions.

The two phase problem has been studied by
Abels \cite{Abe2007}, Denisova \cite{Den1991,Den1994},
Denisova and Solonnikov \cite{DS1991,DS1995}, Giga and Takahashi \cite{GT1994},
Maryani and Saito \cite{MS2017},
Nouri and Poupaund \cite{NP1995}, 
Pr\"{u}ss {\it et al}. \cite{PSSS2012,PS2010,PS2011},
Shibata and Shimizu \cite{SS2011}, etc..
Although these works dealt with the two phase problem for 
the incompressible-incompressible case, as far as the author knows,
the compressible-incompressible case is few.
The compressible-incompressible case was studied by
Denisova \cite{Den2015}, Denisova and Solonnikov \cite{DS2017},
Kubo, Shibata, and Soga \cite{KSS2014}, and Shibata \cite{Shibata2016}.
In particular, Denisova \cite{Den2015} and Denisova and Solonnikov \cite{DS2017} 
studied the compressible-incompressible case
in the $L_2$ Sobolev-Sobodetskii space.
Denisova \cite{Den2015} proved the energy inequality without surface tension.
Denisova and Solonnikov \cite{DS2017} proved
the global-in-time solvability without surface tension
under the assumption that the data are small.
On the other hand, Kubo, Shibata, and Soga \cite{KSS2014}
and Shibata \cite{Shibata2016}
studied the compressible-incompressible case in
the $L_p$ in time and $L_q$ in space frame work.
Kubo, Shibata, and Soga \cite{KSS2014} proved the existence of 
$\CR$-bounded solution operators to the corresponding generalized 
resolvent problem without surface tension and without phase transition and
Shibata \cite{Shibata2016} prove it with surface tension and phase transition.
However, the work in \cite{KSS2014,Shibata2016}
included the problem about the regularity of density, which we mentioned above.
In this paper, we eliminate this problem by using 
the Navier-Stokes-Korteweg equations 
instead of the compressible Navier-Stokes equations.

Our goal is to prove the local well-posedness and for this purpose,
the key step is to prove the maximal $L_p$-$L_q$ 
regularity of the model problem. Let  
\begin{align*}
\mathbb{R}^N_\pm
=\{x=(x_1,\dots,x_N)\in\mathbb{R}^N|\ \pm x_N>0 \},
\quad
\mathbb{R}^N_0=\{x\in \mathbb{R}^N|\ x_N=0 \},
\quad \dot{\mathbb{R}}^N = \mathbb{R}^N_+ \cup \mathbb{R}^N_-.
\end{align*}
and  for any $x_0\in \mathbb{R}^N_0$, we define
$f|_\pm(x_0)$ by
\begin{align*}
f|_\pm(x_0)
=\lim_{\substack{x\to x_0\\ x\in\mathbb{R}^N_\pm}} f(x),
\end{align*}
Let $\rho_{*\pm}$ and $\theta_{*\pm}$ be the mass density and
the absolute temperature of the reference domain:
$\Omega_{t\pm}|_{t=0}$, all of which are positive constants. 
Transforming $\dot\Omega_t$ and $\Gamma_t$ to $\dot{\mathbb{R}}^N$
and $\mathbb{R}^N_0$, respectively, and linearizing the 
problem at $\rho_{*\pm}$ and $\theta_{*\pm}$, we have following
two model problems.	
One is the following system:
\allowdisplaybreaks{
\begin{alignat}2\label{11000}
\partial_t\rho_++\rho_{*+}{\rm div}\,\mathbf{u}_+
&=f_1 &\text{ in $\mathbb{R}^N_+\times (0,T)$},\\
\rho_{*+}\partial_t\mathbf{u_+}-\mu_{*+}\Delta\mathbf{u_+}
-\nu_{*+}\nabla{\rm div}\,\mathbf{u_+}\nonumber& &\\
-\rho_{*+}\kappa_{*+}\nabla\Delta\rho_+
&=\mathbf{f}_2 &\text{ in $\mathbb{R}^N_+\times (0,T)$},
\nonumber\\
{\rm div}\,\mathbf{u}_-=f_3
&={\rm div}\,\mathbf{f}_4\quad
&\text{ in $\mathbb{R}^N_-\times (0,T)$}, \nonumber\\
\rho_{*-}\partial_t\mathbf{u}_-
-\mu_{*-}\Delta\mathbf{u}_-+\nabla \pi_-
&=\mathbf{f}_5 &\text{ in $\mathbb{R}^N_-\times (0,T)$},
\nonumber
\end{alignat}	
}
subject to the interface condition:
for $x_0\in\mathbb{R}^N_0$ and $t\in(0,T)$
\allowdisplaybreaks{
\begin{align}\label{11001}
\partial_t h-\Big(\frac{\rho_{*-}}{\rho_{*-}-\rho_{*+}}
u_{N-}|_-(x_0)
-\frac{\rho_{*+}}{\rho_{*-}-\rho_{*+}}u_{N+}|_+(x_0) \Big)&=d,\\
\mu_{*-}D_{mN}(\mathbf{u_-})|_-(x_0)
-\mu_{*+}D_{mN}(\mathbf{u}_+)|_+(x_0)&=g_m,\nonumber\\
\{\mu_{*-}D_{NN}(\mathbf{u}_-)-\pi_- \}|_-(x_0)
-\{\mu_{*+}D_{NN}(\mathbf{u}_+)& \nonumber\\
+(\nu_{*+}-\mu_{*+}){\rm div}\,\mathbf{u}_+
+\rho_{*+}\kappa_{*+}\Delta\rho_+ \}|_+(x_0)-\sigma\Delta'h&=f_6, \nonumber\\
\frac{1}{\rho_{*-}}\{\mu_{*-}D_{NN}(\mathbf{u}_-)
-\pi_- \}|_-(x_0)
-\frac{1}{\rho_{*+}}\{\mu_{*+}D_{NN}(\mathbf{u}_+)&\nonumber\\
+(\nu_{*+}-\mu_{*+}){\rm div}\,\mathbf{u}_+
+\rho_{*+}\kappa_{*+}\Delta\rho_+ \}|_+(x_0)& =f_7, \nonumber\\
u_{m-}|_-(x_0)-u_{m+}|_+(x_0)&=h_m, \nonumber\\
\partial_N\rho_+|_+(x_0)&=k, \nonumber
\end{align}	
}	
and initial condition:
\begin{alignat}2\label{11002}
\mathbf{u}_\pm|_{t=0}
&=\mathbf{u}_{0\pm}&\quad &\text{ in $\mathbb{R}^N_\pm$},
\\\rho_+|_{t=0}&=\rho_{*+}+\rho_{0+}&\quad 
&\text{ in $\mathbb{R}^N_+$},\nonumber
\\ h|_{t=0}&=h_0&\quad &\text{ in $\mathbb{R}^N_0$},\nonumber
\end{alignat}
where $m$ ranges from $1$ to $N-1$ and we have set
$\mathbf{u}_\pm=(u_{1\pm},\dots,u_{N\pm})$,
$\mu_{*+}=\mu_+(\rho_{*+},\theta_{*+})$,
$\kappa_{*+}=\kappa_+(\rho_{*+},\theta_{*+})$,
$\nu_{*+}=\nu_+(\rho_{*+},\theta_{*+})$,
$\mu_{*-}=\mu_{-}(\theta_{*-})$,
$\Delta' h=\sum_{j=1}^{N-1}\partial_j^2 h$.

The other is the heat equations:
\begin{alignat}2\label{11003}
&\rho_{*+}\kappa_{v*+}\partial_t\theta_+-d_{*+}\Delta\theta_+
=\tilde{f}_+ \qquad&\text{ in $\mathbb{R}^N_+\times (0,T)$},\\
&\rho_{*-}\kappa_{v*-}\partial_t\theta_--d_{*-}\Delta\theta_-
=\tilde{f}_- \qquad&\text{ in $\mathbb{R}^N_-\times (0,T)$},
\nonumber
\end{alignat}
subject to the interface condition:
for $x_0\in \mathbb{R}^N_0$ and $t\in (0,T)$
\begin{align}\label{11004}
\theta_-|_-(x_0)-\theta_+|_+(x_0)=0,\quad
d_{*-}\partial_N\theta_-|_-(x_0)-d_{*+}\partial_N\theta_+|_+(x_0)
=\widetilde{g},
\end{align}	
and the initial condition:
\begin{align}\label{11005}
\theta_\pm|_{t=0}=\theta_{0\pm}\quad{\rm in}\ \mathbb{R}^N_\pm
\end{align}
where we have set $d_{*+}=d(\rho_{*+},\theta_{*})$,
$\kappa_{v*+}=(\pd e_+/\pd \theta_+)(\rho_{*+},\theta_{*})$, 
$d_{*-}=d_-(\theta_{*})$,
and $\kappa_{v*-}=e'_-(\theta_{*})$.
Here, the right-hand sides of (\ref{11000}), (\ref{11001}),
(\ref{11003}), and (\ref{11004}) are nonlinear terms.

We note that the interface condition (\ref{11001}) can
be rewritten as follows: for $x_0\in\mathbb{R}^N_0$ and $t\in(0,T)$
\begin{align*}
\partial_t h-\Big(\frac{\rho_{*-}}
{\rho_{*-}-\rho_{*+}}u_{N-}|_-(x_0)
-\frac{\rho_{*+}}{\rho_{*-}-\rho_{*+}}u_{N+}|_+(x_0) \Big)&=d 
\nonumber,\\
\mu_{*-}D_{mN}(\mathbf{u_-})|_-(x_0)
-\mu_{*+}D_{mN}(\mathbf{u}_+)|_+(x_0)&=g_m,	\\
\{\mu_{*-}D_{NN}(\mathbf{u}_-)-\pi_- \}|_-(x_0)
&=\sigma_-\Delta'h+g_N, \nonumber\\
\{\mu_{*+}D_{NN}(\mathbf{u}_+)
+(\nu_{*+}-\mu_{*+}){\rm div}\,\mathbf{u}_+
+\rho_{*+}\kappa_{*+}\Delta\rho_+ \}|_+(x_0)
&=\sigma_+\Delta'h+g_{N+1}, \nonumber\\
u_{m-}|_-(x_0)-u_{m+}|_+(x_0)&=h_m, \nonumber\\
\partial_N\rho_+|_+(x_0)&=k, \nonumber
\end{align*}	
with 
\begin{align*}
\sigma_\pm=\frac{\rho_{*\pm}\sigma}{\rho_{*-}-\rho_{*+}},\quad
g_{N}=\frac{\rho_{*-}}{\rho_{*-}-\rho_{*+}}(f_6-\rho_{*+}f_7),\quad
g_{N+1}=\frac{\rho_{*+}}{\rho_{*-}-\rho_{*+}}(f_6-\rho_{*-}f_7).
\end{align*}	

As in Shibata \cite{Shibata2014a, Shibata2016a}, the maximal $L_p$-$L_q$ 
regularity and the generation of $C^0$ analytic semigroup 
follow automatically from 
the existence of $\CR$ bounded solution operator families 
of the corresponding generalized resolvent problem.
Hence, in this paper we concentrate on  the existence of 
$\mathcal{R}$-bounded solution operator families for the 
resolvent problem arising from model problem 
with the interface condition:
\begin{align}\label{201}
\lambda\rho_++\rho_{*+}{\rm div}\,\mathbf{u}_+
&=f_1 \qquad{\rm in}\ \mathbb{R}^N_+,\\
\rho_{*+}\lambda\mathbf{u_+}-\mu_{*+}\Delta\mathbf{u_+}
-\nu_{*+}\nabla{\rm div}\,\mathbf{u_+}
-\rho_{*+}\kappa_{*+}\nabla\Delta\rho_+
&=\mathbf{f}_2 \qquad\text{ in $\mathbb{R}^N_+$}, \nonumber\\
{\rm div}\,\mathbf{u}_-=f_3={\rm div}\,\mathbf{f}_4,\quad 
\rho_{*-}\lambda\mathbf{u}_--\mu_{*-}\Delta\mathbf{u}_-
+\nabla \pi_-&=\mathbf{f}_5 
\qquad\text{ in $\mathbb{R}^N_-$}, \nonumber\\
\lambda H-\Big(\frac{\rho_{*-}}{\rho_{*-}-\rho_{*+}}u_{N-}|_-(x_0)
-\frac{\rho_{*+}}{\rho_{*-}-\rho_{*+}}u_{N+}|_+(x_0) \Big)
&=d \nonumber,\\
\mu_{*-}D_{mN}(\mathbf{u_-})|_-(x_0)
-\mu_{*+}D_{mN}(\mathbf{u}_+)|_+(x_0)&=g_m,	\nonumber\\
\{\mu_{*-}D_{NN}(\mathbf{u}_-)-\pi_- \}|_-(x_0)
&=\sigma_-\Delta'H+g_N, \nonumber\\
\{\mu_{*+}D_{NN}(\mathbf{u}_+)
+(\nu_{*+}-\mu_{*+}){\rm div}\,\mathbf{u}_+
+\rho_{*+}\kappa_{*+}\Delta\rho_+ \}|_+(x_0)
& =\sigma_+\Delta'H+g_{N+1}, \nonumber\\
u_{m-}|_-(x_0)-u_{m+}|_+(x_0)&=h_m, \nonumber\\
\partial_N\rho_+|_+(x_0)&=k, \nonumber
\end{align}
which is corresponding to the time dependent problem 
(\ref{11000}), (\ref{11001}), and (\ref{11002}).
Here, $H(x,t)$ is an extension of $h(x',t)$ such that
$H=h$ on $\BR^N_0$.
In this paper, we do not consider (\ref{11000}), 
(\ref{11001}), and (\ref{11002}) anymore, and so 
we use the same symbols in the right-hand side of (\ref{201})	
as used in (\ref{11000}) and (\ref{11001}) below.

In order to state our main results precisely we introduce
function spaces and some more symbols which will be used
throughout the paper.
For any scalar field $\theta$ we set $\nabla \theta = 
(\pd_1\theta, \ldots, \pd_N\theta)$, and for 
any $N$-vector field $\bu = {}^\top(u_1, \ldots, u_N)$,
$\nabla\bu$ is the $N\times N$ matrix with $(i, j)^{\rm th}$
component $\pd_iu_j$. 
For any domain $D$ in $\mathbb{R}^N$, integer $m$,
and $1\leq q\leq\infty$,
$L_q(D)$ and $W^m_q(D)$ denote
the usual Lebesgue space and Sobolev space of
functions defined on $D$ with norms:
$\|\cdot\|_{L_q(D)}$ and $\|\cdot\|_{W^m_q(D)}$,
respectively.
We set $W^0_q(D)=L_q(D)$.
The $\hat{W}^1_q(D)$ is a homogeneous space defined by
$\hat{W}^1_q(D)=
\{f\in L_{q,{\rm loc}}(D)\ |\ \nabla f\in L_q(D) \} $.
For any Banach space $X$, interval $I$, integer $m$,
and $1\leq p\leq \infty$,
$L_p(I,X)$ and $W^m_p(I,X)$ denote the usual Lebesgue space and
Sobolev space of the $X$-valued functions defined on $I$
with norms:
$\|\cdot\|_{L_p(I,X)}$ and $\|\cdot\|_{W^m_p(I,X)}$,
respectively.
For any Banach space $X$, $X^N$ denote the $N$-product space
of $X$, that is
$X^N=\{f=(f_1,\dots,f_N)\ |\ f_i\in X\ (i=1,\dots,N) \}$.
The norm of $X^N$ is also denoted by $\|\cdot\|_X$ for
simplicity and $\|f\|_X=\sum_{j=1}^{N}\|f_j\|_X$ for
$f=(f_1,\dots,f_N)\in X^N$.
For any two Banach spaces $X$ and $Y$, 
$\mathcal{L}(X,Y)$ denotes the space of all bounded linear
operators from $X$ to $Y$, and $\mathcal{L}(X)$ is
the abbreviation of $\mathcal{L}(X,X)$.
Let $U$ be a subset of $\mathbb{C}$.
Then ${\rm Anal}\,(U,\mathcal{L}(X,Y))$ denotes the set of all
$\mathcal{L}(X,Y)$-valued analytic functions defined on $U$.
Throughout in this paper, the letter $C$ denotes generic
constants and $C_{\alpha,\beta,\gamma,\dots}$
means that the constant depends on the quantities
$\alpha,\beta,\gamma,\dots$.
The values of constants $C$ and 
$C_{\alpha,\beta,\gamma,\dots}$ may change from line to line.	

Before we state the main theorem, 
we first introduce the definition of $\mathcal{R}$-boundedness.
\begin{defi}($\mathcal{R}$-boundedness)
Let $X$ and $Y$ be Banach spaces. 
A set of operators $\mathcal{T}\subset \mathcal{L}(X,Y)$ 
is called $\mathcal{R}$-bounded, 
if there is a constant $0<C<\infty$ and $1\leq p<\infty$ 
such that, for all $T_1,\dots,T_m\in\mathcal{T}$ 
and $x_1,\dots,x_m\in X$ with $m\in\mathbb{N}$, we have
\begin{align*}
\left(\int_{0}^{1}\Big\| \sum_{n=1}^{m}r_n(t)T_nx_n
\Big\|_Y^p dt\right)^{1/p}
\leq C \left(\int_{0}^{1}\Big\| \sum_{n=1}^{m}r_n(t)x_n
\Big\|_X^p dt\right)^{1/p}
\end{align*}
where $r_n(t)={\rm sign}\sin (2^n\pi t)$ are the 
Rademacher functions on $[0,1]$.
The smallest such $C$ is called $\mathcal{R}$-bound of 
$\mathcal{T}$ on $\mathcal{L}(X,Y)$, which is denoted by 
$\mathcal{R}_{\mathcal{L}(X,Y)}(\mathcal{T})$.		
\end{defi}	

Let $\eta_*$ be a constant given by
\begin{align*}
\eta_*=
\Big(\frac{\mu_{*+}+\nu_{*+}}{2\kappa_{*+}} \Big)^2
-\frac{1}{\kappa_{*+}}
\end{align*}
and let $\varepsilon_*\in [0,\pi/2)$ be some angle
that is given precisely in Lemma \ref{lem601} below.
In this paper, we assume $\eta_*\neq 0$ and 
$\kappa_{*+}\neq\mu_{*+}\nu_{*+}$.
We discuss these conditions in more 
detail in Remark \ref{eta.condition} below
(cf. Saito \cite[Remark 3.3]{Saito}).

The following theorem is a main theorem in this paper.
\begin{theo}\label{Th02}
Let $1<q<\infty$ and $\varepsilon_* <\varepsilon<\pi/2$.
Assume that $\rho_{*+}\neq\rho_{*-}$, $\eta_*\neq 0$,
and $\kappa_{*+}\neq \mu_{*+}\nu_{*+}$.
Set	
\begin{align*}
\Sigma_\varepsilon
=&\{\lambda=\gamma+i\tau\in \mathbb{C}\backslash\{0\} 
\mid |{\rm arg}\,\lambda|\leq \pi-\varepsilon \},\\
\Sigma_{\varepsilon,\lambda_0}
=&\{\lambda\in\Sigma_\varepsilon \mid 
|\lambda|>\lambda_0 \} \quad
(\lambda_0\geq 0),\qquad
\partial_\tau=\partial/\partial\tau,\\
X_q=&
\{(f_1,\mathbf{f}_2,f_3,\mathbf{f}_4,
\mathbf{f}_5,\mathbf{g},\mathbf{h},k,d) \mid
f_1\in W^1_q(\mathbb{R}^N_+),\enskip
\mathbf{f}_2 \in L_q(\mathbb{R}^N_+)^N,
f_3\in W^1_q(\mathbb{R}^N_-),\\
&\quad
\mathbf{f}_4 \in L_q(\mathbb{R}^N_-)^N,\enskip
\mathbf{f}_5 \in L_q(\mathbb{R}^N_-)^N,\enskip
\mathbf{g}=(g_{1},\dots,g_{N+1})
\in W^1_q(\mathbb{R}^N)^{N+1},\enskip\\
&\quad
\mathbf{h}=(h_1,\dots,h_{N-1})\in 
W^2_q(\mathbb{R}^N)^{N-1} ,\enskip
k\in W^2_q (\mathbb{R}^N ),\enskip 
d\in W^2_q(\mathbb{R}^N) \},\\
\mathcal{X}_q=&
\{(F_1,\dots,F_{15}) \mid
F_1\in W^1_q(\mathbb{R}^N_+),\enskip
F_2 \in L_q(\mathbb{R}^N_+)^N,\enskip 
F_3\in L_q(\mathbb{R}^N_-),\\
&\quad
F_4,F_5,F_6 \in L_q(\mathbb{R}^N_-)^N,\enskip 
F_7 \in L_q(\mathbb{R}^N)^{N+1},\enskip
F_8 \in L_q(\mathbb{R}^N)^{(N+1)N},\\
&\quad 
F_9 \in L_q(\mathbb{R}^N)^{N-1},\enskip
F_{10} \in L_q(\mathbb{R}^N)^{(N-1)N},\enskip
F_{11 }\in L_q(\mathbb{R})^{(N-1)N^2},\\
&\quad
F_{12},F_{13} \in L_q(\mathbb{R}^N)^N,\enskip
F_{14} \in L_q(\mathbb{R}^N)^{N^2},\enskip
F_{15} \in W^2_q(\mathbb{R}^N) \}.
\end{align*}		
Then, there exist a positive constant $\lambda_0$ 
and operator families
$\mathcal{A}_\pm(\lambda)$, $\mathcal{B}_+(\lambda)$,
$\mathcal{P}_-(\lambda)$, and $\mathcal{H}(\lambda)$
with
\begin{align*}
\mathcal{A}_{\pm}(\lambda)\in&
{\rm Anal}\,(\Sigma_{\varepsilon,\lambda_0},
\mathcal{L}(\mathcal{X}_q,W^2_q(\mathbb{R}^N_\pm)^N),\\
\mathcal{B}_{+}(\lambda)\in&
{\rm Anal}\,(\Sigma_{\varepsilon,\lambda_0},
\mathcal{L}(\mathcal{X}_q,W^3_q(\mathbb{R}^N_+)),\\
\mathcal{P}_{-}(\lambda)\in &
{\rm Anal}\,(\Sigma_{\varepsilon,\lambda_0},\mathcal{L}
(\mathcal{X}_q,\hat{W}^1_q(\mathbb{R}^N_-))),\\
\mathcal{H}(\lambda)\in&
{\rm Anal}\,(\Sigma_{\varepsilon,\lambda_0},
\mathcal{L}(\mathcal{X}_q,W^3_q(\mathbb{R}^N)),
\end{align*} 
such that for any $\lambda\in \Sigma_{\varepsilon,\lambda_0}$ 
and 
$\mathbf{F}= (f_1,\mathbf{f}_2,f_3,\mathbf{f}_4,
\mathbf{f}_5,\mathbf{g},\mathbf{h},k,d)\in X_q$,
$\mathbf{u}_\pm=\mathcal{A}_\pm(\lambda) \widetilde{\mathbf{F}}$,
$\rho_+=\mathcal{B}_+(\lambda) \widetilde{\mathbf{F}}$,
$\pi_-=\mathcal{P}_-(\lambda) \widetilde{\mathbf{F}}$,
and $H=\mathcal{H}(\lambda)\widetilde{\mathbf{F}}$ 
are unique solutions of problem (\ref{201}).
Furthermore, for $s=0,1$, we have
\begin{align*}
\mathcal{R}_{\mathcal{L}
(\mathcal{X}_{q},L_q(\mathbb{R}^N)^{N^3+N^2+N})}
(\{(\tau\partial_\tau)^s(G_\lambda^1 \mathcal{A}_\pm(\lambda))
\mid \lambda \in \Sigma_{\varepsilon,\lambda_0}\})
\leq& c_0,\\
\mathcal{R}_{\mathcal{L}
(\mathcal{X}_{q},L_q(\mathbb{R}^N)^{N^3+N^2}
\times W^1_q(\mathbb{R}^N))}
(\{(\tau\partial_\tau)^s(G_\lambda^2 \mathcal{B}_+(\lambda))
\mid \lambda \in \Sigma_{\varepsilon,\lambda_0}\})
\leq& c_0,\\
\mathcal{R}_{\mathcal{L}(\mathcal{X}_{q},L_q(\mathbb{R}^N)^{N})}
(\{(\tau\partial_\tau)^s(\nabla \mathcal{P}_-(\lambda))
\mid \lambda  \in \Sigma_{\varepsilon,\lambda_0}\})
\leq& c_0,\\
\mathcal{R}_{\mathcal{L}
(\mathcal{X}_{q},W^2_q(\mathbb{R}^N)^{N+1})}
(\{(\tau\partial_\tau)^s(G_\lambda^3 \mathcal{H} (\lambda))
\mid \lambda  \in \Sigma_{\varepsilon,\lambda_0}\})
\leq& c_0,
\end{align*}
with some positive constant $c_0$. 
Here, $G_\lambda^1 \mathcal{A}_\pm(\lambda)
=(\lambda\mathcal{A}_\pm(\lambda), 
\lambda^{1/2}\nabla\mathcal{A}_\pm(\lambda),
\nabla^2\mathcal{A}_\pm(\lambda))$,  
$G_\lambda^2 \mathcal{B}_-(\lambda)
=(\lambda\mathcal{B}_-(\lambda),
\lambda^{1/2}\nabla^2\mathcal{B}_-(\lambda),
\nabla^3\mathcal{B}_-(\lambda))$, 
$G_\lambda^3 \mathcal{H} (\lambda)
=(\lambda\mathcal{H}(\lambda),\nabla\mathcal{H})$,
and 
$$\widetilde{\mathbf{F}}
=(f_1,\mathbf{f}_2, \lambda^{1/2}f_3, \nabla f_3,
\lambda\mathbf{f}_4, \mathbf{f}_5, 
\lambda^{1/2}\mathbf{g}, \nabla\mathbf{g},
\lambda \mathbf{h},\lambda^{1/2}\nabla\mathbf{h},
\nabla^2\mathbf{h}, \lambda k, \lambda^{1/2}\nabla k,
\nabla^2 k, d).$$			
\end{theo}

\begin{rema}
(1) The uniqueness of solutions of problem
(\ref{201}) follows from the existence of solutions 
for a dual problem in a similar way to Shibata and Shimizu
\cite[Sect. 3]{SS2012}, so that we omit its proof.

(2) It is easy to show the existence of $\mathcal{R}$-bounded 
solution operator families for the resolvent problem 
arising from (\ref{11003}), (\ref{11004}), and (\ref{11005}). 
In fact, when we employ the similar argumentation to that
in the proof of Theorem \ref{Th02} given in the sequel.
Hence, we do not consider problem (\ref{11003}), (\ref{11004}), 
and (\ref{11005}) in this paper.

(3) We can show the maximal $L_p\mathchar`-L_q$ regularity
theorem for (\ref{11000}), (\ref{11001}), (\ref{11002}),
(\ref{11003}), (\ref{11004}), and (\ref{11005}) 
due to the same theory as in Shibata \cite{Shibata2016}
with the help of the $\mathcal{R}$-bounded solution operator
and the operator valued Fourier multiplier theorem
of Weis \cite{Weis2001}. 			
\end{rema}

This paper is organized as follows.
In Sect. \ref{modeling}, according to the argument due to 
Pr\"{u}ss {\it et al}. \cite{PSSS2012}
(cf. Pr\"{u}ss and Simonett \cite{PS2016}
and Shibata \cite{Shibata2014})
we explain the interface condition \eqref{40000} 
in more detail from the point of conservation of mass,
conservation of momentum, conservation of energy and
increment of entropy and 
we show the complete model. 
In Sect. \ref{Sect.Result}, we introduce some results 
of half spaces.
From Sect.. \ref{Sect.solution} to Sect. \ref{Sect.Lopatinski},
we consider the problem without the surface tension.
In Sect. \ref{Sect.solution}, by the partial Fourier transform,
we have ordinary differential equations with respect to $x_N$.
Then, we solve them and apply the inverse partial 
Fourier transform to its solution in order 
to obtain exact solution formulas to the resolvent problem.
In Sect. \ref{Sect.multipliers}, we introduce some technical
lemmas and give some estimates for the multipliers appearing
in the solution formula.
In Sect. \ref{Sect.Lopatinski}, we analyze the Lopatinski
determinant appearing in the solution formula.
Finally in Sect. \ref{Sect.surface}, 
we prove the main theorem for the 
$\mathcal{R}$-bounded solution operator families.	


\section{Derivation of interface conditions}\label{modeling}

In this section, assuming that the equation \eqref{10000} holds in the bulk
$\dot\Omega_t$, we derive interface conditions \eqref{40000} 
under which  balance of mass, balance of momentum,
balance of energy, and entropy production hold. 
We follow the argument due to Pr\"{u}ss {\it et al}. in \cite{PSSS2012}.
Our model is, however, different from Pr\"{u}ss {\it et al}.
\cite{PSSS2012}, and so we give a detailed explanation. 

For our purpose we may assume that integration appearing below 
is finite. In this sense, our argument below is rather formal from
the integrability point of view \footnote{In order to make the discussion 
in this section rigorous, 
it is enough to assume that the domain is bounded and the outer
boundary conditions are imposed like Pr\"{u}ss {\it et al}. \cite{PSSS2012}}.
In addition, we assume that there exists 
a smooth diffeomorphism $\phi_t:\BR^N \to \BR^N$ such that 
$$\dot\Omega_{t\pm}=\{x = \phi_t(y) \mid y \in \Omega_{0\pm}\},
\quad \Gamma_t = \{x = \phi_t(y) \mid y \in \Gamma_0\}.
$$
for $t > 0$. Let $\bv = (\pd_t\phi_t)(\phi_t^{-1}(x), t)$ and then,
it follows from the Reynolds transport theorem that
\begin{align}\label{2.1}
\frac{d}{dt}\int_{\dot\Omega}f\,dx  = \int_{\dot\Omega_t}
(\pd_t f + \dv(f\bv))\,dx.
\end{align}
In particular, $\bu_\Gamma\cdot \bn_t = \bv\cdot\bn_t$ on $\Gamma_t$. 
We assume that  
\begin{align}
&\pd_t\rho + \dv(\rho\bu) = 0, \label{2.2}\\
&\pd_t(\rho\bu) + \dv(\rho\bu\otimes\bu) - \dv\bT = \rho\bff,
\label{2.3}\\
&\pd_t\Big(\frac{\rho}{2}|\bu|^2 + \rho e\Big)
+ \dv\Big(\Big(\frac{\rho}{2}|\bu|^2 + \rho e\Big)\bu\Big)
-\dv(\bT\bu - \bq) = \rho\bff\cdot\bu + \rho r
\label{2.4}
\end{align}
hold in the bulk $\dot\Omega_t$. And then, we look for the interface
conditions under which the following formulas hold:
\begin{alignat}2
&\frac{d}{dt}\int_{\dot\Omega} \rho\,dx = 0
&\quad&\text{(Balance of Mass)},
\label{2.5} \\
&\frac{d}{dt}\int_{\dot\Omega} \rho\bu\,dx = \int_{\dot\Omega} \rho\bff\,dx
&\quad&\text{(Balance of Momentum)}, 
\label{2.6} \\
&\frac{d}{dt}\Bigl\{\int_{\dot\Omega_t}
\Big(\frac{\rho}{2}|\bu|^2 + \rho e\Big)\,dx + \sigma|\Gamma_t|\Bigr\}
= \int_{\dot\Omega_t}(\rho\bff\cdot + \rho r)\,dx
&\quad
&\text{(Balance of Energy)},
\label{2.7}\\
&\frac{d}{dt}\int_{\dot\Omega} \rho\eta\,dx \geq 
\int_{\dot\Omega_t}\rho\theta^{-1}r\,dx
&\quad&\text{(Entropy Production)}\label{2.8}.
\end{alignat}
where $|\Gamma_t|$ is the Hausdorff measure of $\Gamma_t$.
We know that
\begin{align}\label{2.9}
\frac{d}{dt}|\Gamma_t| 
= -\int_{\Gamma_t}H_{\Gamma_t}\bu_\Gamma\cdot\bn_t\,d\tau.
\end{align}
Here and in the sequel, $d\tau$ denotes the surface element of 
$\Gamma_t$. 
\vskip0.5pc\noindent
{\bf Balance of Mass:}\\
By \eqref{2.1}, \eqref{2.2}, and the divergence theorem of Gauss we have 
\begin{align*}
\frac{d}{dt}\int_{\dot\Omega}\rho\,dx 
= \int_{\dot\Omega_t}(\pd_t\rho + \dv(\rho\bv))\,dx
=\int_{\dot\Omega_t}\dv(\rho(\bv-\bu))\,dx
=\int_{\Gamma_t}[[\rho(\bu_\Gamma -\bu)]]\cdot\bn_t\,d\tau.
\end{align*}
Thus, to obtain \eqref{2.5} it is sufficient to assume that
$$[[\rho(\bu_\Gamma -\bu)]]\cdot\bn_t = 0 \quad\text{on $\Gamma_t$},
$$
and so we define $\j$ be 
\begin{align}\label{2.10}
\j = \rho_+(\bu_+-\bu_\Gamma)\cdot\bn_t = \rho_-(\bu_--\bu_\Gamma)\cdot\bn_t,
\end{align}
which is called the phase flux, more precisely, the interfacial mass flux. 
Since $\bu_\Gamma\cdot\bn_t= \bu_+\cdot\bn_t - \j/\rho_+
= \bu_-\cdot\bn_t - \j/\rho_-$, we have
\begin{align}\label{2.11}
\j = \frac{[[\bu]]\cdot\bn_t}{[[1/\rho]]}.
\end{align}
A phase transition takes place if $\j \not=0$. 

Furthermore, by \eqref{2.10}, we have 
\begin{align}\label{2.11*}
V_{\Gamma_t} = \bu_\Gamma\cdot\bn_t
= \Bigl(\frac{\rho_+}{\rho_+-\rho_-}\bu_+
-\frac{\rho_-}{\rho_+-\rho_-}\bu_-\Bigr)\cdot\bn_t
= \frac{[[\rho\bu]]\cdot\bn_t}{[[\rho]]}.
\end{align}
which is the kinetic condition in the case that $\j\not=0$. 

On the other hand, if $\j=0$, then we have
$\bu_\pm\cdot\bn_t = \bu_\Gamma\cdot\bn_t$ on $\Gamma_t$.
Thus, we have a usual kinetic condition:
$V_{\Gamma_t} = \bu\cdot\bn_t$. If $\Gamma_t$ is 
defined by $F(x, t) = 0$ locally, then 
$F(\phi_t(y), t) = 0$ for $y \in \Gamma_0$. Thus, we have
$$ 0 = \frac{d}{dt}F(\phi_t(y),t) = \pd_t F + (\nabla F)\cdot\bu_\Gamma.$$
Since $\bn_t$ is parallel to $\nabla F$, we have 
$(\nabla F)\cdot\bu_\Gamma = (\nabla F)\cdot\bu_\pm$ on
$\Gamma_t$, and so we have
$$\pd_t F + \bu\cdot\nabla F= 0\quad\text{on $\Gamma_t$}.$$
This is a different representation formula of kinetic 
condition when the phase transition does not take place. 
\vskip0.5pc\noindent
{\bf Balance of Momentum:} \\
We will prove that  it follows from the balance of momentum that
\begin{align}\label{2.12}\left\{\begin{aligned}
\rho(\pd\bu + \bu\cdot\nabla\bu) -\dv\bT = \rho\bff
\qquad&\text{in $\dot\Omega_t$}, \\
\j[[\bu]]-[[\bT\bn_t]] = -\sigma H_{\Gamma_t}\bn_t
\qquad&\text{on $\Gamma_t$}.
\end{aligned}\right.\end{align}
In fact, we write  \eqref{2.3} componentwise as 
$$\pd_t(\rho u_i) + \dv(\rho u_i\bu) = \sum_{j=1}^N\pd_jT_{ij}
= \rho f_i\quad\text{in $\dot\Omega$}.
$$
And then, by \eqref{2.1} we have
\begin{align*}
\frac{d}{dt}\int_{\dot\Omega} \rho u_i\,dx
& = \int_{\dot\Omega} (\pd_t(\rho u_i) + \dv(\rho u_i\bv)\,dx\\
&= \int_{\dot\Omega_t}\Big(\dv(\rho u_i(\bv-\bu))+ \sum_{j=1}^N\pd_jT_{ij}
+ \rho f_i\Big)\,dx \\
& = \int_{\Gamma_t}([[\rho u_i(\bu_\Gamma-\bu)]]\cdot\bn_t
+(\text{$i^{\rm th}$ component of}\, [[\bT\bn_t]]))\,d\tau
+ \int_{\Omega_t}\rho f_i\,dx.
\end{align*}
By \eqref{2.10}, we have $[[\rho u_i(\bu_\Gamma-\bu)]]\cdot\bn_t
= -\j[[u_i]]$ on $\Gamma_t$, and so 
in order that \eqref{2.6} holds it is sufficient to 
assume that
$$\int_{\Gamma_t}(-\j[[\bu]] + [[\bT\bn_t]])\,d\tau=0,
$$
which leads to
$$
-\j[[\bu]]+[[\bT\bn_t]]=\dv_{\Gamma_t} \bT_{\Gamma_t}
\quad\text{on $\Gamma_t$},
$$
where $\bT_{\Gamma_t}$ denotes surface stress and $\dv_{\Gamma_t}$ denotes the
surface divergence. When we consider  surface tension on
$\Gamma_t$, we have
$$\dv_{\Gamma_t} \bT_{\Gamma_t} = \sigma\dv_{\Gamma_t}\bn_t 
= \sigma \Delta_{\Gamma_t}x = \sigma H_{\Gamma_t} \bn_t
$$
where $x$ is a position vector of $\Gamma_t$.
Thus, we have the interface condition: 
$$\j[[\bu]] - [[\bT\bn_t]] = -\sigma H_{\Gamma_t}\bn_t.
\quad\text{on $\Gamma_t$}.
$$
Moreover, by \eqref{2.2} we have
\begin{align*}
\pd_t(\rho\bu) + \dv(\rho\bu\otimes\bu)
= (\pd_t\rho + \dv(\rho\bu))\bu
+ \rho(\pd_t\bu + \bu\cdot\nabla\bu)
= \rho(\pd_t\bu + \bu\cdot\nabla\bu),
\end{align*}
and so we can rewrite \eqref{2.3} as 
$$\rho(\pd\bu + \bu\cdot\nabla\bu) -\dv\bT = \rho\bff
\quad\text{in $\dot\Omega_t$}.$$
\vskip0.5pc\noindent
{\bf Balance of Energy:}\\
We will prove that it follows from the balance of energy that 
\begin{align}\label{2.13}\left\{\begin{aligned}
\rho(\pd_te + \bu\cdot\nabla e) + \dv \bq
-\bT:\nabla\bu = \rho r
\qquad&\text{in $\dot\Omega_t$},\\
\frac{\j}{2}[[|\bu-\bu_\Gamma|^2]] + \j[[e]]
- [[\bT(\bu-\bu_\Gamma)]]\cdot\bn_t
+ [[\bq]]\cdot\bn_t = 0 \qquad&\text{on $\Gamma_t$}.. 
\end{aligned}\right.\end{align}
In fact, by \eqref{2.1}, \eqref{2.9} and \eqref{2.4}, we have
\begin{align*}
&\frac{d}{dt}\Bigl\{\int_{\dot\Omega_t}(\frac{\rho}{2}|\bu|^2
+\rho e)\,dx + \sigma|\Gamma_t|\Bigr\}\\
& = \int_{\dot\Omega_t}\Bigl\{\pd_t\Big(\frac{\rho}{2}|\bu|^2+\rho e\Big)
+ \dv\Big(\Big(\frac{\rho}{2}|\bu|^2 +\rho e\Big)\bv\Big)\Bigl\}\,dx
- \sigma\int_{\Gamma_t} H_{\Gamma_t}\bu_\Gamma\cdot\bn_t\,d\tau
\\
& = \int_{\dot\Omega}\Bigl\{\dv\Big(\Big(\frac{\rho}{2}|\bu|^2 +\rho e\Big)
(\bv-\bu)\Big)+\dv(\bT\bu-\bq) + \rho\bff\cdot\bu + \rho r\Bigl\}\,dx 
- \sigma\int_{\Gamma_t} H_{\Gamma_t}\bu_\Gamma\cdot\bn_t\,d\tau \\
& = \int_{\Gamma_t}\Bigl\{\Big[\Big[\Big(\frac{\rho}{2}|\bu|^2 + \rho e\Big)
(\bu_\Gamma - \bu)\Big]\Big]\cdot\bn_t + [[\bT\bu - \bq]]\cdot\bn_t
- \sigma H_{\Gamma_t}\bu_\Gamma\cdot\bn_t\Bigl\}\,d\tau
+ \int_{\dot\Omega} (\rho\bff\cdot\bu + \rho r)\,dx.
\end{align*}
If we assume that 
$$\Big[\Big[\Big(\frac{\rho}{2}|\bu|^2 + \rho e\Big)(\bu_\Gamma - \bu)\Big]\Big]
\cdot\bn_t + [[\bT\bu - \bq]]\cdot\bn_t
- \sigma H_{\Gamma_t}\bu_\Gamma\cdot\bn_t = 0,
$$
then we have the balance of energy \eqref{2.7}. By \eqref{2.10}
\begin{align*}
\Big[\Big[\Big(\frac{\rho}{2}|\bu|^2 + \rho e\Big)(\bu_\Gamma- \bu)\Big]\Big]
\cdot\bn_t
&= -\frac{\j}{2}[[|\bu|^2]] - \j[[e]] 
= -\frac{\j}{2}[[|\bu-\bu_\Gamma|^2]]-\j[[\bu]]\cdot\bu_\Gamma
-\j[[e]].
\end{align*}
Noting that $\bT$ is a symmetric matrix, by \eqref{2.12} we have
\begin{align*}
[[\bT\bu-\bq]]\cdot\bn_t - \sigma H_\Gamma\bu_\Gamma\cdot\bn_t
&=
[[\bT(\bu-\bu_\Gamma)]]\cdot\bn_t +
[[\bT\bn_t]]\cdot\bu_\Gamma
-[[\bq]]\cdot\bn_t - \sigma(H_{\Gamma_t}\bn_t)\cdot\bu_\Gamma\\
& = [[\bT(\bu-\bu_\Gamma)]]\cdot\bn_t - [[\bq]]\cdot\bn_t
+ \j[[\bu]]\cdot\bu_\Gamma.
\end{align*}
Putting these formulas together gives
the following interface condition:
$$
\frac{\j}{2}[[|\bu-\bu_\Gamma|^2]] + \j[[e]]
- [[\bT(\bu-\bu_\Gamma)]]\cdot\bn_t
+ [[\bq]]\cdot\bn_t = 0 \quad\text{on $\Gamma_t$}. 
$$
By \eqref{2.2} and \eqref{2.3}, we rewrite \eqref{2.4} as
\begin{align*}
&\pd_t\Big(\frac{\rho}{2}|\bu|^2 + \rho e\Big)
+ \dv\Big(\Big(\frac{\rho}{2}|\bu|^2
+ \rho e\Big)\bu\Big)-\dv(\bT\bu - \bq)\\
&= \Big(\frac12|\bu|^2+e\Big)(\pd_t\rho + \dv(\rho \bu))
+ \rho(\bu\cdot\bu_t + e_t)
+ \rho\bu\cdot\nabla\Big(\frac12|\bu|^2 + e\Big)
-(\dv\bT)\cdot\bu - \bT:\nabla \bu + \dv\bq\\
& = \rho(\pd_t e + \bu\cdot\nabla e) + \rho\bff\cdot\bu - 
\bT:\nabla\bu + \dv\bq,
\end{align*}
where we have set $\bT:\nabla\bu = \sum_{i,j=1}^NT_{ij}\pd_iu_j$. 
Putting this and \eqref{2.4} together gives
$$\rho(\pd_t e+ \bu\cdot\nabla e) + \dv \bq
-\bT:\nabla\bu = \rho r
\quad\text{in $\dot\Omega_t$}.
$$
\vskip0.5pc\noindent
{\bf Entropy Production:} \\
We now introduce the fundamental thermodynamic relations 
which read
\begin{align}\label{entropy-assump:1}
\frac{\pd e}{\pd \eta} = \theta, \quad
e = \psi + \theta\eta, \quad \eta = -\frac{\pd \psi}{\pd \theta},
\quad \kappa_v = \frac{\pd e}{\pd \theta},
\quad \ell = \theta\Big[\Big[\frac{\pd \psi}{\pd \theta}\Big]\Big] 
= -\theta[[\eta]].
\end{align}
The quantities $\kappa_v$ and $\ell$ are called heat capacity and latent heat,
respectively.  We assume that
\begin{align}\label{entropy-assump:2}
\kappa_v = \frac{\pd e}{\pd \theta} > 0, \quad \theta > 0.
\end{align}

As constitutive laws in the phases, for the compressible viscous fluid part,
$\Omega_{t+}$, we employ the Korteweg's law for the stress tensor 
and the Dunn-Serrin law for the energy flux, while for the
incompressible viscous fluid part, $\Omega_{t-}$, we employ
the Newton's law for the stress tensor and Fourier's law for the 
energy flux.  Namely, we assume that 
\vskip0.5pc\noindent
{\bf Constitutive Law in the Phases:}
\begin{alignat*} 2
\bT_+  &= \bS_+ +\bK_+ - \pi_+\bI, 
\quad \bS_+ = \mu_+\bD(\bu_+) +(\nu_+-\mu_+)\dv\bu_+ - \pi_+\bI,&&\\
\bK_+ &= (\alpha_0(\rho_+)|\nabla\rho_+|^2 + \alpha_1(\rho_+)\Delta\rho_+)\bI
+ \beta(\rho_+)\nabla\rho_+\otimes\nabla\rho_+, \\
\bq_+ &= -d_+\nabla\theta_+ + (\kappa_+(\rho_+) \rho_+ \dv \bu_+)\nabla \rho_+
&\quad &\text{in $\Omega_{t+}$}, \\
\bT_- & = \bS_- - \pi_-\bI,
\quad \bS_- = \mu_-\bD(\bu_-) - \pi_-\bI,
\quad \bq = -d_-\nabla\theta_- &\quad&\text{in $\Omega_{t-}$}.
\end{alignat*}
Here, $\bD(\bu) = (1/2)(\nabla \bu + {}^\top\nabla\bu)$. 
To ensure non-negative entropy production in the bulk $\Omega_{t \pm}$,
we assume that
\begin{align}\label{const-assump:1}
\alpha_0 &= \frac{\kappa_+ + \rho_+\kappa'_+}{2}|\nabla\rho_+|^2,
\quad \alpha_1 = \kappa_+\rho_+, \quad \beta= - \kappa_+,\\
\mu_\pm &> 0, \quad \nu_+ > \frac{N-1}{N}\mu_+,
\quad d_\pm > 0.\nonumber
\end{align}
Assuming 
$\bq_+ = -d_+\nabla\theta_+ + (\kappa_+(\rho_+) \rho_+ \dv \bu_+)\nabla \rho_+$, 
Dunn and Serrin \cite{DS1985} proved that $\bK_+$ should 
equal   
$(1/2)(\kappa_++\rho_+\kappa_+')|\nabla\rho_+|^2 + \kappa_+\rho_+\Delta
\rho_+)\bI - \kappa_+\nabla\rho_+\otimes\nabla\rho_+$ to ensure non-negative
entropy production. But, in the following, assuming that 
$\bK_+ = (\alpha_0(\rho_+)|\nabla\rho_+|^2 + \alpha_1(\rho_+)\Delta\rho_+)\bI
+ \beta(\rho_+)\nabla\rho\otimes\rho$, we prove that
$\bq_+ = -d_+\nabla\theta_+ + (\kappa_+(\rho_+) \rho_+ \dv \bu_+)\nabla \rho_+$,
and $\alpha_0$, $\alpha_1$ and $\beta$ should be given as in
\eqref{compressible-assumption} to ensure non-negative
entropy production.  Namely, our argument below is just opposite direction to
Dunn and Serrin \cite{DS1985}.

We also assume the following.
\vskip0.5pc\noindent
{\bf Constitutive Law on the Interface  $\Gamma_t$:} 
\begin{align}\label{const:1}
[[\theta]]=0, \quad
[[\bu-(\bu\cdot\bn_t)\bn_t]]=0,
\quad(\nabla\rho_+)\cdot\bn_t=0
\quad\text{on $\Gamma_t$}.
\end{align}
Hence in our model, the temperature and the tangential part of
velocity field are continuous across the interface,
and the interstitial working does not take place in the normal direction
of boundary $\Gamma_{t}$.
In addition, the third boundary condition in (\ref{const:1})
ensures the Fourier law and the generalized Gibbs-Thomson law,
which we will explain below.

We first consider 
$$\frac{d}{dt}\int_{\Omega_{t+}}\rho_+\eta_+\,dx.$$
We assume that $e_+ = e_+(\eta_+, \rho_+, \nabla_+\rho_+)$. 
Let $z_j$ be a variable corresponding to $\pd_j\rho_+$. 
We then have
\begin{align*}
\pd_t e_++ \bu_+\cdot\nabla e_+ 
=& \frac{\pd e_+}{\pd \eta_+}(\pd_t\eta_+ + \bu_+\cdot\nabla\eta_+)
+ \frac{\pd e_+}{\pd \rho_+}(\pd_t\rho_++\bu_+\cdot\nabla\rho_+)\\
&+ \sum_{j=1}^N\frac{\pd e_+}{\pd z_j}(\pd_t\pd_j\rho_+ + \bu_+\cdot
\nabla\pd_j\rho_+).
\end{align*}
By \eqref{2.2}, we have 
\begin{align}\label{const:2}
\pd_t\rho_+ + \bu_+\cdot\nabla\rho_+ = -\rho_+\dv\bu_+.
\end{align}
Differentiating \eqref{const:2}  by $x_j$, we have
\begin{align}\label{const:3}
\pd_t\pd_j\rho_+ + \bu_+\cdot\nabla\pd_j\rho_+
=-(\pd_j\bu_+)\cdot(\nabla \rho_+) -\pd_j(\rho_+\dv\bu_+).
\end{align}
Inserting \eqref{const:2} and \eqref{const:3}
and using the assumption:
$\pd e_+/\pd\eta_+ = \theta_+$,  we have
\begin{align*} \pd_te_+ + \bu_+\cdot\nabla e_+
=& \theta_+(\pd_t\eta_+ + \bu_+\cdot\nabla\eta_+)
-\frac{\pd e_+}{\pd \rho_+}\rho_+\dv\bu_+ \\
&-\sum_{j=1}^N\frac{\pd e_+}{\pd z_j}((\pd_j\bu_+)\cdot(\nabla\rho_+)
+ \pd_j(\rho_+\dv\bu_+)).
\end{align*}
On the other hand, by \eqref{2.13}, we have
$$\rho_+(\pd_te_+ +\bu_+\cdot\nabla e_+)
= -\dv \bq_+ + \bT_+:\nabla\bu_+ +\rho_+ r_+,
$$
and so 
\begin{align}\label{entropy:0}
\rho_+\theta_+(\pd_t\eta_+ + \bu_+\cdot\nabla\eta_+)
=& -\dv\bq_+ + \bT_+:\nabla\bu_+ +
\frac{\pd e_+}{\pd \rho_+}\rho_+^2\dv\bu_+\\
&+ \sum_{j=1}^N\frac{\pd e_+}{\pd z_j}\rho_+
((\pd_j\bu_+)\cdot(\nabla\rho_+) + \pd_j(\rho_+dv\bu_+))
+\rho_+ r_+.\nonumber
\end{align}
Using \eqref{2.2}, we have $\pd_t(\rho\eta) + \dv(\rho\eta\bu)
= \rho(\pd_t \eta + \bu\cdot\nabla\eta)$,
and so by \eqref{2.1} we have
\begin{align}\label{entropy:1}
\frac{d}{dt}\int_{\dot\Omega} \rho\eta\,dx
&= \int_{\dot\Omega}(\pd_t(\rho\eta) + \dv(\rho\eta\bv))\,ds \\
&= \int_{\dot\Omega} \rho(\pd_t\eta + \bu\cdot\nabla\eta)\,dx
+ \int_{\dot\Omega} \dv(\rho\eta(\bv-\bu)\,dx.\nonumber
\end{align}
From \eqref{entropy:0} we have
\begin{align}\label{entropy:3}
&\rho_+(\pd_t\eta_+ + \bu_+\cdot\nabla\eta_+) -\theta_+^{-1}\rho_+ r_+\\
&=\theta_+^{-1}\Big\{-\dv\bq_+ + \bT_+:\nabla\bu_+ +
\frac{\pd e_+}{\pd \rho_+}\rho_+^2\dv\bu_+
+ \sum_{j=1}^N\frac{\pd e_+}{\pd z_j}\rho_+
((\pd_j\bu_+)\cdot(\nabla\rho_+) + \pd_j(\rho_+\dv\bu_+))\Big\}.\nonumber
\end{align}

We now assume that 
\begin{align*}\bS_+ =& \mu_+\bD(\bu_+) + (\nu_+-\mu_+)\dv\bu_+\bI-\pi_+\bI,\\
\bK_+ =& (\alpha_0(\rho_+)|\nabla\rho_+|^2 + \alpha_1(\rho_+)\Delta\rho_+)\bI
+ \beta(\rho_+)\nabla\rho_+\otimes\rho_+.
\end{align*}
Since $\bD(\bu) = (D_{ij}(\bu))$
is symmetric, setting $\bu={}^\top(u_1, \ldots, u_N)$, we have 
$$\bD(\bu):\nabla\bu = \sum_{i,j=1}^N D_{ij}(\bu)\pd_iu_j 
= \frac12\sum_{i,j=1}^N D_{ij}(\bu)
(\pd_iu_j + \pd_ju_i) = |\bD(\bu)|^2.$$
In the same manner, we have
\begin{align*}
\nabla\rho_+\otimes\nabla\rho_+:\nabla\bu_+
=& \sum_{i,j=1}^N(\pd_i\rho_+)(\pd_j\rho_+)(\pd_iu_{j+})\\
=&\frac12 \sum_{i,j=1}^N(\pd_i\rho_+)(\pd_j\rho_+)(\pd_iu_{j+}+\pd_ju_{i+}) \\
=& \sum_{i,j=1}^N
D_{ij}(\bu_+)(\pd_i\rho_+)(\pd_j\rho_+).
\end{align*}
We then have
\begin{align*}
&\bT_+:\nabla \bu_+ \\
&= \sum_{i,j=1}^N\{\mu_+ D_{ij}(\bu_+) 
+((\nu_+-\mu_+)\dv\bu_+ - \pi_+ + \alpha_0|\nabla\rho_+|^2
+\alpha_1\Delta\rho_+)\delta_{ij} +\beta(\pd_i\rho_+)(\pd_j\rho_+)\}\pd_iu_{j+}\\
& = \mu_+|\bD(\bu_+)|^2 + (\nu_+-\mu_+)(\dv\bu_+)^2 
+(-\pi_+ +\alpha_0|\nabla\rho_+|^2 + \alpha_1\Delta\rho_+)\dv\bu_+\\
&\quad + \beta\sum_{i,j=1}^ND_{ij}(\bu_+)(\pd_i\rho_+)(\pd_j\rho_+).
\end{align*}
Using chain rule and the fact that 
$(\nabla \alpha_1)\cdot(\nabla\rho_+)= \alpha'_1|\nabla\rho_+|^2$,  we have
\begin{align*}
\theta_+^{-1}\alpha_1(\Delta\rho_+)\dv\bu_+ =&
\dv(\theta_+^{-1}\alpha_1(\nabla\rho_+)\dv\bu_+) 
- (\nabla(\theta_+^{-1}\alpha_1\dv\bu_+))\cdot\nabla\rho_+ \\
=&\dv(\theta_+^{-1}\alpha_1(\nabla\rho_+)\dv\bu_+ 
+\theta_+^{-2}\alpha_1(\nabla\rho_+)\dv\bu_+)\cdot(\nabla\theta_+)\\
&-\theta_+^{-1}\alpha'_1|\nabla\rho_+|^2\dv\bu_+ 
-\theta_+^{-1}\alpha_1(\nabla\dv\bu_+)\cdot\nabla\rho_+, \\
\theta_+^{-1}\dv\bq_+ =& \dv(\theta_+^{-1}\bq_+) 
+\theta_+^{-2}\bq_+\cdot(\nabla\theta_+).
\end{align*}
Inserting the formulas above into \eqref{entropy:3}, we have
\begin{align}
&\rho_+(\pd_t\eta_+ + \bu_+\cdot\eta_+) 
-\theta_+^{-1}\rho_+ r_+ \label{entropy:4}\\
&= -\dv(\theta_+^{-1}\bq_+) - \theta_+^{-2}\bq_+\cdot(\nabla\theta_+)
+\dv(\theta_+^{-1}\alpha_1(\nabla\rho_+)\dv\bu_+) 
+\theta_+^{-2}(\alpha_1(\nabla\rho_+)\dv\bu_+)\cdot(\nabla\theta_+)\nonumber\\
&\quad+\theta_+^{-1}\Big\{\mu_+|\bD(\bu_+)|^2 + (\nu_+-\mu_+)(\dv\bu_+)^2 
+(-\pi_+ + (\alpha_0-\alpha_1')|\nabla\rho_+|^2)\dv\bu_+\nonumber\\
&\quad+ \beta\sum_{i,j=1}^ND_{ij}(\bu_+)(\pd_i\rho_+)(\pd_j\rho_+)
+ \frac{\pd e_+}{\pd \rho_+}\rho_+^2\dv\bu_+\nonumber\\
&\quad+ \sum_{j=1}^N\frac{\pd e_+}{\pd z_j}\rho_+
((\pd_j\bu_+)\cdot(\nabla\rho_+) + (\pd_j\rho_+)\dv\bu_+
+ \rho_+(\pd_j\dv\bu_+))
-\alpha_1(\nabla\dv\bu_+)\cdot\nabla\rho_+\Big\}.\nonumber
\end{align}
We now assume that
\begin{align}\label{entropy-assump:3}
\frac{\pd e_+}{\pd z_j}= \frac{\alpha_1}{\rho_+^{2}}\pd_j\rho_+,
\end{align}
to obtain
$$\sum_{j=1}^N\frac{\pd e_+}{\pd z_j}\rho_+^2\pd_j\dv\bu_+
= \alpha_1(\nabla\dv\bu_+)\cdot(\nabla\rho_+).
$$
In particular, we have
\begin{align}\label{energy:1}
e_+= \frac{\alpha_1}{2\rho_+^{2}}|\nabla\rho_+|^2 + e'_+(\eta_+, \rho_+).
\end{align}
with some function $e'_+(\eta_+, \rho_+)$. 
Thus, we have
$$\frac{\pd e_+}{\pd \rho_+} = \frac{\alpha_1'}{2\rho_+^{2}}|\nabla\rho_+|^2 
- \frac{\alpha_1}{\rho_+^{3}}|\nabla\rho_+|^2 + \frac{\pd e'_+}{\pd \rho_+},$$
and so we have 
\begin{align*}
&\frac{\pd e_+}{\pd \rho_+}\rho_+^2\dv\bu_+
+ \sum_{j=1}^N\frac{\pd e_+}{\pd z_j}\rho_+
((\pd_j\bu_+)\cdot(\nabla\rho_+) + (\pd_j\rho_+)\dv\bu_+
+ \rho_+(\pd_j\dv\bu_+))
-\alpha_1(\nabla\dv\bu_+)\cdot\nabla\rho_+\\
& = \frac12\alpha_1'|\nabla\rho_+|^2\dv\bu_+ 
- \alpha_1\rho_+^{-1}|\nabla\rho_+|^2\dv\bu_+
+ \frac{\pd e'_+}{\pd\rho_+}\rho^2_+\dv\bu_+ \\
&\quad + \sum_{i,j=1}^N\frac{\alpha_1}{\rho_+}
D_{ij}(\bu_+)(\pd_i\rho_+)(\pd_j\rho_+)
+ \frac{\alpha_1}{\rho_+}|\nabla\rho_+|^2\dv\bu_+,
\end{align*}
which, combined with \eqref{entropy:4}, leads to 
\begin{align*}
&\rho_+(\pd_t\eta_+ + \bu_+\cdot\eta_+)  -\theta_+^{-1}\rho_+ r_+\\
&= -\dv(\theta_+^{-1}\bq_+) - \theta_+^{-2}\bq_+\cdot\nabla\theta_+
+\dv(\theta_+^{-1}\alpha_1(\nabla\rho_+)\dv\bu_+) +
\theta_+^{-2}(\alpha_1(\nabla\rho_+)\dv\bu_+)\cdot(\nabla\theta_+)
\nonumber \\
&+\theta_+^{-1}\Bigl\{\mu_+|\bD(\bu_+)|^2 + (\nu_+-\mu_+)(\dv\bu_+)^2\\ 
&+\Big(-\pi_+ + \Big(\alpha_0-\frac{1}{2}\alpha_1'\Big)|\nabla\rho_+|^2  + 
\frac{\pd e'_+}{\pd\rho_+}\rho^2_+\Big)\dv\bu_+
+ \Big(\beta+\frac{\alpha_1}{\rho_+}\Big)
\sum_{i,j=1}^ND_{ij}(\bu_+)(\pd_i\rho_+)(\pd_j\rho_+)
\Bigl\}.
\end{align*}
We now assume that 
$\bq_+= -d_+\nabla\theta_+ + \alpha_1\dv\bu_+(\nabla\rho_+)$, 
that is we employ the Dunn-Serrin law for the energy flux. 
Furthermore,  we assume that 
\begin{align}\label{entropy:5}
-\pi_++ \Big(\alpha_0-\frac{1}{2}\alpha_1'\Big) 
+ \frac{\pd e'_+}{\pd \rho_+}\rho_+^2=0,
\quad \beta + \frac{\alpha_1}{\rho_+} = 0.
\end{align}
We then have
\begin{align}\label{entropy:6}
\rho_+(\pd_t\eta_+ + \bu_+\cdot\nabla\eta_+)  
=& \dv(\theta_+^{-1}d_+\nabla\theta_+) + \theta_+^{-2}d_+|\nabla\theta_+|^2\\
&+ \theta_+^{-1}(\mu_+|\bD(\bu_+)|^2 + (\nu_+-\mu_+)(\dv\bu_+)^2)
+\theta_+^{-1}\rho_+ r_+\nonumber
\quad\text{in $\Omega_{t+}$}.
\end{align}
If we choose 
\begin{align}\label{const:4}
\beta = -\kappa, \quad \alpha_1 = \rho_+\kappa, \quad
\alpha_0 = \frac12\alpha_1' = \frac12(\kappa + \rho_+\kappa'),
\quad \frac{\pd e'_+}{\pd \rho_+}\rho_+^2 = \pi_+,
\end{align}
then, the formulas in \eqref{entropy:5} hold. 
In particular,  
we have the Korteweg tensor given in \eqref{compressible-assumption}.

Since $\bT_- = \bS_-$, assuming that $e_- = e_-(\rho_-, \eta_-)$,
$\bq_- = -d_-\nabla\theta_-$ and $\dfrac{\pd e_-}{\pd \eta_-} 
= \theta_-$, and $\dfrac{\pd e_-}{\pd \rho_-}\rho_-^2 = \pi_-$, we also have
\begin{align}\label{entropy:7}
\rho_-(\pd_t\eta_- + \bu_-\cdot\nabla\eta_-)  
=& \dv(\theta_-^{-1}d_-\nabla\theta_-) + \theta_-^{-2}d_-|\nabla\theta_-|^2\\
&+ \theta_-^{-1}(\mu_-|\bD(\bu_-)|^2 + (\nu_--\mu_-)(\dv\bu_-)^2) 
+ \theta_-^{-1}\rho_- r_- \quad\text{in $\Omega_{t-}$}.\nonumber
\end{align}
Putting \eqref{entropy:1}, \eqref{entropy:6}, and \eqref{entropy:7}
together 
and assuming that $[[\theta]]=0$, by the divergence theorem of
Gauss we have
\begin{align}\label{entropy:8}
\frac{d}{dt}\int_{\dot\Omega_t} \rho\eta\,dx
&= \int_{\dot\Omega_t}(\theta^{-2}d|\nabla\theta|^2 
+ \mu|\bD(\bu)|^2 + (\nu-\mu)(\dv\bu)^2)\,dx \\
&\quad+ \int_{\Gamma_t}(\theta^{-1}[[d\nabla\theta]]\cdot\bn_t -
[[\rho\eta(\bu-\bu_\Gamma)]]\cdot\bn_t)\,d\tau
+ \int_{\dot\Omega_t} \theta^{-1}\rho r\,dx.\nonumber
\end{align}Since 
$(\dv\bu)^2 \leq N|\bD(\bu)|^2$, if we assume that
\begin{align}\label{entropy:8*}
\mu_{\pm} > 0, \quad \nu_+ \geq \frac{N-1}{N}\mu_+, 
\end{align}
then we have 
$$\int_{\dot\Omega_t}(\theta^{-2}d|\nabla\theta|^2 
+ \mu|\bD(\bu)|^2 + (\nu-\mu)(\dv\bu)^2)\,dx \geq 0
$$
because $\dv \bu_-=0$ in $\Omega_{t-}$. 
Since $[[\rho\eta(\bu-\bu_\Gamma)]]\cdot\bn = \j[[\eta]]$ as follows from 
\eqref{2.10}, 
to obtain Entropy Production:
$$\frac{d}{dt}\int_{\dot\Omega_t}\rho\eta \,dx \geq \int_{\dot\Omega_t}
\rho\theta^{-1} r\,dx,$$
it is sufficient to assume that 
\begin{align}\label{stefan:1}
\j\theta[[\eta]]- [[d\nabla\theta]]\cdot\bn_t = 0
\quad\text{on $\Gamma_t$},
\end{align}
which is called the Stefan law.

We now derive the generalized Gibbs-Thomson law:
\begin{align}\label{entropy:9}
[[\psi]] + \j^2\Bigl[\Bigl[\frac{1}{2\rho^2}\Bigr]\Bigr]
-\Bigl[\Bigl[\frac{1}{\rho}\bn_t\cdot(\bT\bn_t)\Bigr]\Bigr]=0
\quad\text{on $\Gamma_t$}.
\end{align}
provided that $\j\not=0$.  Let $\tau_i$ ($i=1, \ldots, N-1$) be
the tangent vectors of $\Gamma_t$.  We then write 
$$\bu-\bu_\Gamma =<\bu-\bu_\Gamma, \bn_t>\bn_t
+ \sum_{i-1}^{N-1}<\bu-\bu_\Gamma, \tau_i>\tau_i.
$$
Using the orthogonality of $\{\tau_1, \ldots, \tau_{N-1}, \bn_t\}$, 
we have 
$$|\bu-\bu_\Gamma|^2 =
|(\bu-\bu_\Gamma)\cdot\bn_t|^2 
+\sum_{i=1}^{N-1}|(\bu-\bu_\Gamma)\cdot\tau_i|^2..
$$
Since $[[\bu - <\bu, \bn_t>\bn_t]]=0$ as follows from \eqref{const:1},
we have 
\begin{align}\label{entropy:10}
[[<\bu-\bu_\Gamma, \tau_i>]]= 0, 
\end{align}
and so by \eqref{2.10} we have 
\begin{align*}
\frac12[[|\bu- \bu_\Gamma|^2]]&= \frac12[[|(\bu-\bu_\Gamma)\cdot\bn_t|^2]]
= \frac12(|(\bu_+-\bu_\Gamma)\cdot\bn_t|^2-|(\bu_--\bu_\Gamma)\cdot\bn_t|^2)
\\
&= \frac12\Bigl(\Bigl(\frac{\j}{\rho_+}\Bigr)^2
-\Bigl(\frac{\j}{\rho_-}\Bigr)^2\Bigr)
= \frac{\j^2}{2}\Bigl[\Bigl[\frac{1}{\rho^2}\Bigr]\Bigr].
\end{align*}
Inserting this formula into the second formula in \eqref{2.13},
using the relation: $\psi = e-\theta\eta$, and recalling
the formulas: 
$\bq_+ = -d_+\nabla\theta_+ +(\kappa_+\rho_+\dv\bu_+)\nabla\rho_+$
and $\bq_- = -d_-\nabla\theta_-$, by \eqref{const:1} and \eqref{stefan:1} 
we have
\begin{align}
0 &= \frac{\j^3}{2}\Bigl[\Bigl[\frac{1}{\rho^2}\Bigr]\Bigr] 
+ \j[[\psi]]+\j\theta[[\eta]]
-[[(\bu-\bu_\Gamma)\cdot\bT\bn_t)]]-[[d\nabla\theta]]\cdot\bn_t
+(\kappa_+\rho_+\dv\bu_+)(\nabla\rho_+)\cdot\bn_+ \label{entropy:11}\\
& = \frac{\j^3}{2}\Bigl[\Bigl[\frac{1}{\rho^2}\Bigr]\Bigr] 
+\j[[\psi]]-[[(\bu-\bu_\Gamma)\cdot(\bT\bn_t)]].
\nonumber
\end{align}
We write the last term as
\begin{align*}
&[[(\bu-\bu_\Gamma)\cdot(\bT\bn_t)]]
= [[<\bu-\bu_\Gamma, \bn_t><\bn_t, \bT\bn_t)]]
+ \sum_{i=1}^{N-1}[[<\bu-\bu_\Gamma, \tau_i><\tau_i, \bT\bn_t>]]\\
&=<\bu_+-\bu_\Gamma, \bn_t><\bn_t, \bT_+\bn_t> 
-<\bu_--\bu_\Gamma, \bn_t><\bn_t, \bT_-\bn_t>\\
&+ \sum_{i-1}^{N-1}
\{<\bu_+-\bu_\Gamma, \tau_i><\tau_i, \bT_+\bn_+>
-<\bu_--\bu_\Gamma, \tau_i><\tau_i, \bT_-\bn_t>\}.
\end{align*}
By \eqref{2.10}, we have
\begin{align*}
& <\bu_+-\bu_\Gamma, \bn_t><\bn_t, \bT_+\bn_t> 
-<\bu_--\bu_\Gamma, \bn_t><\bn_t, \bT_-\bn_t> \\
&\quad= \j\Bigl(\frac{<\bn_t, \bT_+\bn_t>}{\rho_+}
-\frac{<\bn_t, \bT_-\bn_t>}{\rho_-}\Bigr)
=\j\Bigl[\Bigl[\frac{1}{\rho}\bn_t\cdot(\bT\bn_t)\Bigr]\Bigr].
\end{align*}
On the other hand, by \eqref{const:1}, \eqref{2.12},
and \eqref{entropy:11}, we have
\begin{align*}
&<\bu_+-\bu_\Gamma, \tau_i><\tau_i, \bT_+\bn_+>
-<\bu_--\bu_\Gamma. \tau_i><\tau_i, \bT_-\bn_t> \\
&\quad = <\bu_+-\bu_\Gamma, \tau_i><\tau_i, [[\bT\bn_t]]>
=<\bu_+-\bu_\Gamma, \tau_i><\tau_i, \j[[\bu]]> = 0.
\end{align*}
Summing up, we have obtained
$$[[(\bu-\bu_\Gamma)\cdot(\bT\bn_t)]]=\j\Bigl[\Big[\frac{1}{\rho}
\bn_t\cdot(\bT\bn_t)\Bigr]\Bigr],
$$
which, combined with \eqref{entropy:10}, leads to 
\eqref{entropy:9}. 
Notice that if $\j=0$, we do not have \eqref{entropy:9}. 

Finally, using $\theta$ we rewrite the first equation in \eqref{2.13}.
Since
$$\dv(\theta^{-1}d\nabla\theta) + \theta^{-2}d|\nabla\theta|^2
= \theta^{-1}\dv(d\nabla\theta),
$$
by \eqref{entropy:6}
and \eqref{entropy:7}, we have
\begin{align}\label{entropy:12}
\rho\theta(\pd_t\eta + \bu\cdot\nabla\eta)
= \dv(d\nabla\theta) + \mu|\bD(\bu)|^2
+ (\nu-\mu)(\dv\bu)^2 + \rho r
\quad\text{in $\dot\Omega_t$}.
\end{align}
By \eqref{entropy-assump:1}, 
\begin{align*}
\kappa_v = \frac{\pd e}{\pd \theta} 
= \frac{\pd}{\pd \theta}(\psi + \theta\eta)
= \frac{\pd\psi}{\pd\theta} + \eta + \theta\frac{\pd\eta}{\pd\theta}
=\theta\frac{\pd \eta}{\pd \theta}, \quad 
\frac{\pd \eta}{\pd \rho} = -\frac{\pd^2\psi}{\pd\rho\pd\theta},
\end{align*}
and so, by \eqref{2.2}  we have
\begin{align*}
\rho\theta(\pd_t\eta + \bu\cdot\nabla\eta)
&= \rho\theta\frac{\pd \eta}{\pd \theta}(\pd_t\theta + \bu\cdot\nabla\theta)
+ \rho\theta\frac{\pd \eta}{\pd \rho}(\pd_t\rho + \bu\cdot\nabla\bu)\\
&= \rho \kappa_v (\pd_t\theta + \bu\cdot\nabla\theta)
+\theta\rho^2\frac{\pd^2\psi}{\pd\rho\pd\theta}\dv\bu,
\end{align*}
which, combined with \eqref{entropy:12} leads to 
$$\rho\kappa_v(\pd_t\theta + \bu\cdot\nabla\theta)
= \dv(d\nabla\theta) + \mu|\bD(\bu)|^2
+ (\nu-\mu)(\dv\bu)^2 -\theta\rho^2\frac{\pd^2\psi}{\pd\rho\pd\theta}\dv\bu 
+ \rho r \quad\text{in $\dot\Omega_t$}.
$$
Summing up, we have obtained the following complete model. \vskip0.5pc
{\bf The Complete Model:}\\
In the bulk:
\begin{alignat*}2
&\left\{\begin{aligned}
&\pd_t\rho_+ + \dv(\rho_+\bu_+) = 0\\
&\rho_+(\pd_t\bu_+ +\bu_+\cdot\nabla\bu_+)
- \dv(\bS_++\bT_+-\pi_+\bI) = \rho_+\bff_+\\
&\rho_+\kappa_{v+}(\pd_t\theta_+ + \bu_+\cdot\nabla\theta_+)
-\dv(d_+\nabla\theta_+)-\mu_+|\bD(\bu_+)|^2 \\
&\quad
- (\nu_+-\mu_+)(\dv\bu_+)^2
+ \theta_+\rho_+^2\frac{\pd^2\psi_+}{\pd\rho_+\pd\theta_+}\dv\bu_+
= \rho_+r_+
\end{aligned}\right.
&\qquad &\text{in $\Omega_{t+}$},\\
&\left\{\begin{aligned}
&\dv \bu_- =0 \\
&\rho_-(\pd_t\bu_- + \bu_-\nabla\bu_-) - 
\dv(S_- - \pi_-\bI) = \rho_-\bff_- \\
&\rho_-\kappa_{v-}(\pd_t\theta_- + \bu_-\cdot\nabla\theta_-)
-\dv(d_-\nabla\theta_-) - \mu_-|\bD(\bu_-)|^2
=\rho_-r_-
\end{aligned}\right.
&\qquad&\text{in $\Omega_{t-}$}. 
\end{alignat*}
On the interface $\Gamma_t$: 
$$
\left\{\begin{aligned}
&[[\bu]] = \j\Bigl[\Bigl[\frac{1}{\rho}\Bigr]\Bigr]\bn_t,  
\quad \j[[\mathbf{u}]]-[[\mathbf{T}]]\mathbf{n}_t
=-\sigma H_{\Gamma_t}\mathbf{n}_t,\\
&[[\theta]]=0, \quad
\j\theta[[\eta]]-[[d\nabla\theta]]\cdot\mathbf{n}_t=0,\\
&[[\psi]] +  \cfrac{\j^2}{2}\Big[\Big[\cfrac{1}{\rho^2}\Big]\Big] 
-\Big[\Big[\cfrac{1}{\rho}(\mathbf{Tn}_t\cdot\mathbf{n}_t)\Big]\Big]=0\\
&V_t=\mathbf{u}_\Gamma\cdot\mathbf{n}_t
=\cfrac{[[\rho\mathbf{u}]]\cdot\mathbf{n}_t}{[[\rho]]},
\quad (\nabla\rho_+)\cdot\mathbf{n}_t=0.	
\end{aligned}\right.
$$



\section{Results of half spaces} 
\label{Sect.Result}	
	
In this section, we introduce some results of half spaces.
To prove Theorem \ref{Th02}, 
we consider the following systems:	
\begin{align}
\label{202}
&\left\{\begin{aligned}
\enskip\lambda\rho_++\rho_{*+}{\rm div}\,\mathbf{u}_+&=f_1 
&\text{ in $\mathbb{R}^N_+$},\\
\enskip\rho_{*+}\lambda\mathbf{u_+}
-\mu_{*+}\Delta\mathbf{u_+}
-\nu_{*+}\nabla{\rm div}\,\mathbf{u_+}
-\rho_{*+}\kappa_{*+}\Delta\nabla\rho_+
&=\mathbf{f}_2
&\text{ in $\mathbb{R}^N_+$},\\
\enskip\mu_{*+}D_{mN}(\mathbf{u}_+)|_+&=g_{m},\\
\enskip\{\mu_{*+}D_{NN}(\mathbf{u}_+)
+(\nu_{*+}-\mu_{*+}){\rm div}\,\mathbf{u}_+&\\
+\rho_{*+}\kappa_{*+}\Delta\rho_+ \}|_+ &=g_{N+1}, \\
\enskip\partial_N\rho_+|_+&=k
\end{aligned}\right.
\intertext{and}
\label{203}
&\left\{\begin{aligned}
{\rm div}\,\mathbf{u}_-=f_3={\rm div}\,\mathbf{f}_4,&
&\enskip\rho_{*-}\lambda\mathbf{u}_-
-\mu_{*-}\Delta\mathbf{u}_-+\nabla \pi_-&=\mathbf{f}_5 
&\text{ in $\mathbb{R}^N_-$},\\
\mu_{*-}D_{mN}(\mathbf{D(u_-)})|_-=0,&&\enskip 
\{\mu_{*-}D_{NN}(\mathbf{u}_-)-\pi_- \}|_-&=g_{N}.
\end{aligned}\right.
\end{align}
The existence of $\mathcal{R}$-bounded solution operators of 
(\ref{202}) and (\ref{203}) are proved by Saito \cite{Saito} 
and Shibata \cite{Shibata2014a}, respectively.	
In fact, we know the following two lemmas.
\begin{lemm}\label{Th03}{\rm (\cite{Saito})}
Let $1<q<\infty$, $\varepsilon_*<\varepsilon<\pi/2$. 
Assume that $\rho_{*+}\neq\rho_{*-}$, $\eta_*\neq 0$,
and $\kappa_{*+}\neq \mu_{*+}\nu_{*+}$. Set	
\begin{align*}
Y_{q+}=
&\{(f_1, \mathbf{f}_2, \widetilde{\mathbf{g}}, k) \mid
f_1\in W^1_q(\mathbb{R}^N_+),\enskip
\mathbf{f}_2 \in L_q(\mathbb{R}^N_+)^N,\\
&\quad
\widetilde{\mathbf{g}}=(g_{1},\dots,g_{(N-1)},g_{N+1})
\in W^1_q(\mathbb{R}^N_+)^N ,\enskip
k\in W^2_q (\mathbb{R}^N_+) \},\\
\mathcal{Y}_{q+}=
&\{(F_1,F_2, \widetilde{F}_{7+},
\widetilde{F}_{8+} , F_{12}, F_{13}, F_{14}) \mid
F_1\in W^1_q(\mathbb{R}^N_+),\enskip
F_2 \in L_q(\mathbb{R}^N_+)^N,\\
&\quad
\widetilde{F}_{7+} \in L_q(\mathbb{R}^N_+)^N,\enskip
\widetilde{F}_{8+} \in L_q(\mathbb{R}^N_+)^{N^2},\enskip
F_{12} \in L_q(\mathbb{R}^N_+)^N,\\
&\quad
F_{13} \in L_q(\mathbb{R}^N_+)^N,\enskip 
F_{14} \in L_q(\mathbb{R}^N_+)^{N^2} \}.	
\end{align*}
Then, there exists a positive constant $\lambda_0$ 
and operator families $\mathcal{A}^1_{+}(\lambda)$
and $\mathcal{B}^1_{+}(\lambda)$ with	
\begin{align*}
\mathcal{A}^1_{+}(\lambda)\in&
{\rm Anal}\,(\Sigma_{\varepsilon,\lambda_0},
\mathcal{L}(\mathcal{Y}_{q+},W^2_q(\mathbb{R}^N_+)^N),\\
\mathcal{B}^1_{+}(\lambda)\in&
{\rm Anal}\,(\Sigma_{\varepsilon,\lambda_0},
\mathcal{L}(\mathcal{Y}_{q+},W^3_q(\mathbb{R}^N_+)),
\end{align*}
such that for any $\lambda\in \Sigma_{\varepsilon,\lambda_0}$
and  $\mathbf{F}_+=(f_1,\mathbf{f}_2,
\widetilde{\mathbf{g}},k)\in Y_{q+}$,
$\mathbf{u}_+=\mathcal{A}^1_+(\lambda)\widetilde{\mathbf{F}}_+$ 
and
$\rho_+=\mathcal{B}^1_+(\lambda) \widetilde{\mathbf{F}}_+$
are unique solutions of problem (\ref{202}).
Furthermore, for $s=0,1$, we have	
\begin{align*}
\mathcal{R}_{\mathcal{L}
(\mathcal{Y}_{q+},L_q(\mathbb{R}^N)^{N^3+N^2+N})}
(\{(\tau\partial_\tau)^s(G_\lambda^1 \mathcal{A}^1_+(\lambda))
\mid \lambda \in \Sigma_{\varepsilon,\lambda_0}\})
\leq& c_0,\\
\mathcal{R}_{\mathcal{L}
(\mathcal{Y}_{q+},L_q(\mathbb{R}^N)^{N^3+N^2}
\times W^1_q(\mathbb{R}^N))}
(\{(\tau\partial_\tau)^s(G_\lambda^2 \mathcal{B}^1_+(\lambda))
\mid \lambda \in \Sigma_{\varepsilon,\lambda_0}\})
\leq& c_0,
\end{align*}	
with some positive constant $c_0$.
Here, 
$\widetilde{\mathbf{F}}_+=
(f_1,\mathbf{f}_2,\lambda^{1/2}\widetilde{\mathbf{g}},
\nabla\widetilde{\mathbf{g}},
\lambda k, \lambda^{1/2}\nabla k ,\nabla^2 k)$.				
\end{lemm}	
\begin{lemm}\label{Th04}{\rm (\cite{Shibata2014a})}
Let $1<q<\infty$, $0<\varepsilon<\pi/2$. Set
\begin{align*}
Y_{q-}=
&\{(f_3,\mathbf{f}_4,\mathbf{f}_5,g_N) \mid
f_3\in W^1_q(\mathbb{R}^N_-),\enskip
\mathbf{f}_4,\enskip
\mathbf{f}_5 \in L_q(\mathbb{R}^N_-)^N,\enskip
g_{N}
\in W^1_q(\mathbb{R}^N_-) \}
\\
\mathcal{Y}_{q-}=
&\{(F_3, F_4, F_5, F_6, \widetilde{F}_{7-},
\widetilde{F}_{8-}) \mid
F_3 \in L_q(\mathbb{R}^N_-),\enskip
F_4,F_5,F_6 \in L_q(\mathbb{R}^N_-)^N,\\
&\quad
\widetilde{F}_{7-} \in L_q(\mathbb{R}^N_-),\enskip
\widetilde{F}_{8-} \in L_q(\mathbb{R}^N_-)^N \}.
\end{align*}
Then, there exists 
operator families 
$\mathcal{A}^1_{-}(\lambda)$ and 
$\mathcal{P}^1_{-}(\lambda)$ with
\begin{align*}
\mathcal{A}^1_{-}(\lambda)\in &
{\rm Anal}\,(\Sigma_\varepsilon,\mathcal{L}
(\mathcal{Y}_{q-},W^2_q(\mathbb{R}^N_-)^N)),
\\
\mathcal{P}^1_{-}(\lambda)\in &
{\rm Anal}\,(\Sigma_\varepsilon,\mathcal{L}
(\mathcal{Y}_{q-},\hat{W}^1_q(\mathbb{R}^N_-)))
\end{align*}
such that for any $\lambda\in \Sigma_\varepsilon$ and 
$\mathbf{F}_-=(f_3,\mathbf{f}_4,\mathbf{f}_5, g_N)\in 
Y_{q-}$, 
$\mathbf{u}_-
=\mathcal{A}^1_{-}(\lambda)\widetilde{\mathbf{F}}_-$
and $\pi_-=\mathcal{P}^1_{-}(\lambda)\widetilde{\mathbf{F}}_-$
are unique solutions of problem (\ref{203}).
Furthermore, for $s=0,1$, we have
\begin{align*}
\mathcal{R}_{\mathcal{L}
(\mathcal{Y}_{q-},L_q(\mathbb{R}^N)^{N^3+N^2+N})}
(\{(\tau\partial_\tau)^s(G_\lambda^1 \mathcal{A}^1_-(\lambda))
\mid \lambda \in \Sigma_{\varepsilon,\lambda_0}\})
\leq& c_0,\\
\mathcal{R}_{\mathcal{L}(\mathcal{Y}_{q-},L_q(\mathbb{R}^N)^{N})}
(\{(\tau\partial_\tau)^s(\nabla \mathcal{P}^1_-(\lambda))
\mid \lambda  \in \Sigma_{\varepsilon,\lambda_0}\})
\leq& c_0,
\end{align*}
with some positive constant $c_0$.
Here, $\widetilde{\mathbf{F}}_-
=(\lambda^{1/2}f_3, \nabla f_3, \lambda\mathbf{f}_4,
\mathbf{f}_5, \lambda^{1/2} g_N, \nabla g_N)$.				
\end{lemm}	
	
Thus, it is sufficient to consider the problem (\ref{201}) with 	
$f_1=0$, $\mathbf{f}_2=0$, $f_3=0$, $\mathbf{f}_5=0$,
and $g_k=0$ ($k=1,\dots,N+1$). Finally, we consider one more
auxiliary problem:	
\begin{align}\label{204}
\lambda\rho_++\rho_{*+}{\rm div}\,\mathbf{u}_+&=0 
&\text{ in $\mathbb{R}^N_+$},\\
\rho_{*+}\lambda\mathbf{u_+}
-\mu_{*+}\Delta\mathbf{u_+}
-\nu_{*+}\nabla{\rm div}\,\mathbf{u_+}
-\rho_{*+}\kappa_{*+}\Delta\nabla\rho_+&=0 
&\text{ in $\mathbb{R}^N_+$}, \nonumber\\
{\rm div}\,\mathbf{u}_-=0,\quad 
\rho_{*-}\lambda\mathbf{u}_--\mu_{*-}\Delta\mathbf{u}_-
+\nabla \pi_-&=0 
&\text{ in $\mathbb{R}^N_-$}, \nonumber\\
\mu_{*-}D_{mN}(\mathbf{D(u_-)})|_-
-\mu_{*+}D_{mN}(\mathbf{u}_+)|+&=0,	\nonumber\\
\{\mu_{*-}D_{NN}(\mathbf{u}_-)-\pi_- \}|_-&
=0, \nonumber\\
\{\mu_{*+}D_{NN}(\mathbf{u}_+)
+(\nu_{*+}-\mu_{*+}){\rm div}\,\mathbf{u}_+
+\rho_{*+}\kappa_{*+}\Delta\rho_+ \}|_+&
=0, \nonumber\\
u_{m-}|_--u_{m+}|_+=h_m,\ \partial_N\rho_+|_+&=0. \nonumber
\end{align}	
From Sect. \ref{Sect.solution} to Sect. \ref{Sect.Lopatinski},
we prove the following theorem.	
\begin{theo}\label{Th05}
Let $1<q<\infty$ and $\varepsilon_*<\varepsilon<\pi/2$. 
Assume that $\rho_{*+}\neq\rho_{*-}$, $\eta_*\neq 0$,
and $\kappa_{*+}\neq \mu_{*+}\nu_{*+}$.
Set
\begin{align*}
Z_q=&
\{\mathbf{h}=(h_1,\dots,h_{N-1})\mid 
\mathbf{h}\in W^2_q(\mathbb{R}^N)^{N-1}\},\\
\mathcal{Z}_q=&
\{(F_9, F_{10}, F_{11}) \mid
F_9 \in L_q(\mathbb{R}^N)^{N-1},\enskip 
F_{10} \in L_q(\mathbb{R}^N)^{(N-1)N},\enskip
F_{11} \in L_q(\mathbb{R}^N)^{(N-1)N^2}\}.
\end{align*}
Then, there exist a positive constant $\lambda_0$ 
and operator families
$\mathcal{A}_\pm(\lambda)$, $\mathcal{B}_+(\lambda)$,
and $\mathcal{P}_-(\lambda)$ with
\begin{align*}
\mathcal{A}^2_{\pm}(\lambda)\in&
{\rm Anal}\,(\Sigma_{\varepsilon,\lambda_0},
\mathcal{L}(\mathcal{Z}_q,W^2_q(\mathbb{R}^N_\pm)^N),\\
\mathcal{B}^2_{+}(\lambda)\in&
{\rm Anal}\,(\Sigma_{\varepsilon,\lambda_0},
\mathcal{L}(\mathcal{Z}_q,W^3_q(\mathbb{R}^N_+)),\\
\mathcal{P}^2_{-}(\lambda)\in &
{\rm Anal}\,(\Sigma_{\varepsilon,\lambda_0},\mathcal{L}
(\mathcal{Z}_q,\hat{W}^1_q(\mathbb{R}^N_-))),
\end{align*} 	
such that for any $\lambda\in \Sigma_{\varepsilon,\lambda_0}$ 
and 
$\mathbf{h}\in Z_q$,
$\mathbf{u}_\pm=\mathcal{A}^2_\pm(\lambda) 
\widetilde{\mathbf{H}}$,
$\rho_+=\mathcal{B}^2_+(\lambda) \widetilde{\mathbf{H}}$, and
$\pi_-=\mathcal{P}^2_-(\lambda) \widetilde{\mathbf{H}}$
are solutions of problem (\ref{204}).
Furthermore, for $s=0,1$, we have
\begin{align}\label{operator.estimate1}
\mathcal{R}_{\mathcal{L}
(\mathcal{Z}_{q},L_q(\mathbb{R}^N)^{N^3+N^2+N})}
(\{(\tau\partial_\tau)^s(G_\lambda^1 \mathcal{A}^2_\pm(\lambda))
\mid \lambda \in \Sigma_{\varepsilon,\lambda_0}\})
\leq& c_0,\\
\mathcal{R}_{\mathcal{L}
(\mathcal{Z}_{q},L_q(\mathbb{R}^N)^{N^3+N^2}
\times W^1_q(\mathbb{R}^N))}
(\{(\tau\partial_\tau)^s(G_\lambda^2 \mathcal{B}^2_+(\lambda))
\mid \lambda \in \Sigma_{\varepsilon,\lambda_0}\})
\leq& c_0,
\nonumber\\
\mathcal{R}_{\mathcal{L}(\mathcal{Z}_{q},L_q(\mathbb{R}^N)^{N})}
(\{(\tau\partial_\tau)^s(\nabla \mathcal{P}^2_-(\lambda))
\mid \lambda  \in \Sigma_{\varepsilon,\lambda_0}\})
\leq& c_0,\nonumber
\end{align}
with some positive constant $c_0$. 
Here, $\widetilde{\mathbf{H}}
=(\lambda \mathbf{h},\lambda^{1/2}\nabla\mathbf{h},
\nabla^2\mathbf{h})$.		
\end{theo}	




\section{Solution formulas without surface tension} 
\label{Sect.solution}
In this section, we consider the following equations:
\begin{align} \label{401}
\lambda\rho_++\rho_{*+}{\rm div}\,\mathbf{u}_+&=0 
&\text{ in $\mathbb{R}^N_+$},\\
\rho_{*+}\lambda\mathbf{u_+}-\mu_{*+}\Delta\mathbf{u_+}
-\nu_{*+}\nabla{\rm div}\,\mathbf{u_+}
-\rho_{*+}\kappa_{*+}\Delta\nabla\rho_+&=0
&\text{ in $\mathbb{R}^N_+$}, \nonumber\\
{\rm div}\,\mathbf{u}_-=0,\quad 
\rho_{*-}\lambda\mathbf{u}_--\mu_{*-}\Delta\mathbf{u}_-
+\nabla \pi_-&=0 
&\text{ in $\mathbb{R}^N_-$}, \nonumber\\
\mu_{*-}D_{mN}(\mathbf{D(u_-)})|_-
-\mu_{*+}D_{mN}(\mathbf{u}_+)|+&=0,	\nonumber\\
\{\mu_{*-}D_{NN}(\mathbf{u}_-)-\pi_- \}|_-&
=\sigma_-\Delta'H, \nonumber\\
\{\mu_{*+}D_{NN}(\mathbf{u}_+)+(\nu_+{*+}-\mu_{*+})
{\rm div}\,\mathbf{u}_+
+\rho_{*+}\kappa_{*+}\Delta\rho_+ \}|_+&
=\sigma_+\Delta'H, \nonumber\\
u_{m-}|_--u_{m+}|_+=h_m,\quad \partial_N\rho_+|_+&=0, \nonumber
\end{align}	
where we have added $\sigma_\pm\Delta' H$ with 
$\sigma_\pm=\rho_{*\pm}\sigma/(\rho_{*-}-\rho_{*+})$
to (\ref{204}) for the latter use.
Let $\widehat{v}=\mathcal{F}_{x'}[v](\xi',x_N)$ denote
the partial Fourier transform with respect to the tangential
variable $x'=(x_1,\dots,x_{N-1})$ with 
$\xi'=(\xi_1,\dots,\xi_{N-1})$ defined by
$\mathcal{F}_{x'}[v](\xi',x_N)
=\int_{\mathbb{R}^{N-1}}e^{-i\xi'\cdot\xi'}v(x',x_N)dx'$.
Applying the partial Fourier transforms to (\ref{401}) yields 
ordinary differential equations with respect to $x_N\neq 0$:	
\begin{align}
\label{402}	
\lambda\widehat{\rho}_++\rho_{*+}
\widehat{{\rm div}\,\mathbf{u}_+}&=0
&\text{ for $x_N>0$},\\
\label{403}	
\rho_{*+}\lambda\widehat{u}_{j+}-\mu_{*+}
(\partial_N^2-|\xi'|^2)\widehat{u}_{j+}
-\nu_{*+}i\xi_j\widehat{{\rm div}\,\mathbf{u_+}}&\\
\qquad-i\xi_j\{\rho_{*+}\kappa_{*+}(\partial_N^2-|\xi'|^2)\}
\widehat{\rho}_+&=0&\text{ for $x_N>0$},\nonumber\\
\label{404}	
\rho_{*+}\lambda\widehat{u}_{N+}
-\mu_{*+}(\partial_N^2-|\xi'|^2)\widehat{u}_{N+}
-\nu_{*+}\partial_N\widehat{{\rm div}\,\mathbf{u_+}}&\\
-\partial_N\{\rho_{*+}\kappa_{*+}(\partial_N^2-|\xi'|^2)\}
\widehat{\rho}_+&=0
&\text{ for $x_N>0$},\nonumber\\
\label{405}	
\widehat{{\rm div}\,\mathbf{u}_-}&=0&\text{ for $x_N<0$},\\
\label{406}	
\rho_{*-}\lambda\widehat{u}_{j-}-\mu_{*-}(\partial_N^2-|\xi'|^2)
\widehat{u}_{j-}-i\xi_j\widehat{\pi}_-&=0
&\text{ for $x_N<0$},\\
\label{407}	
\rho_{*-}\lambda\widehat{u}_{N-}
-\mu_{*-}(\partial_N^2-|\xi'|^2)\widehat{u}_{N-}
-\partial_N\widehat{\pi}_-&=0&\text{ for $x_N<0$},
\end{align}	
subject to the interface condition:
\begin{align}
\label{408}	
\mu_{*-}(\partial_N\widehat{u}_{m-}+i\xi_m\widehat{u}_{N-})|_--
\mu_{*+}(\partial_N\widehat{u}_{m+}+i\xi_m\widehat{u}_{N+})|_+
&=0,\\
\label{409}	
\{2\mu_{*-}\partial_N\widehat{u}_{N-}-\widehat{\pi}_- \}|_-
&=\sigma_-|\xi'|^2\widehat{H}(0),\\
\label{410}	
\{2\mu_{*+}\partial_N\widehat{u}_{N+}
+(\nu_{*+}-\mu_{*+})\widehat{{\rm div}\,\mathbf{u_+}}
-\rho_{*+}\kappa_{*+}(\partial_N^2-|\xi'|^2)
\widehat{\rho}_+ \}|_+
&=\sigma_+|\xi'|^2\widehat{H}(0),\\
\label{411}	
\widehat{u}_{m-}|_--\widehat{u}_{m+}|_+&=\widehat{h}_m(0),\\
\label{412}	
\partial_N\widehat{\rho}_+&=0.
\end{align}	
Here and in the sequel, $j$ runs from $1$ to $N-1$.
According to Saito \cite{Saito}, 
from (\ref{402}), (\ref{403}), and (\ref{404}), we obtain
\begin{align}\label{413}
(\partial_N^2-B_+^2)P_\lambda(\partial_N)\widehat{u}_{J+}
=0\qquad(J=1,\dots,N-1)
\end{align}	
with
\begin{align*}
B_+&=\sqrt{|\xi'|^2+\rho_{*+}\mu_{*+}^{-1}\lambda}\qquad
({\rm Re}\,B_+>0),\\
P_\lambda(t)&
=\rho_{*+}\lambda^2-\lambda(\mu_{*+}+\nu_{*+})(t^2-|\xi'|^2)
+(t^2-|\xi'|^2)\{\rho_{*+}\kappa_{*+}(t^2-|\xi'|^2) \}.
\nonumber
\end{align*}
The roots of $P_\lambda(t)=0$ are 
$t=\pm\sqrt{|\xi'|^2+s_i\lambda}\ (i=1,2)$
and $s_i\ (i=1,2)$ are the root of the following equation:
\begin{align}\label{eq:P}
z^2-\left(\frac{\mu_{*+}+\nu_{*+}}{\kappa_{*+}}
\right)z 
+\frac{1}{\kappa_{*+}}=0.
\end{align}	
Here, $t_i$ are defined by $t_i=\sqrt{|\xi'|^2+s_i\lambda}$, 
whose detail will be discussed in Sect. \ref{Sect.multipliers}.
As seen in Sect. \ref{Sect.multipliers},
we have three roots $B_+$, $t_1$, and $t_2$ with
positive real parts different from each other.	
	
On the other hand, according to Shibata \cite{Shibata2016}, 
from (\ref{405}), (\ref{406}), and (\ref{407}), we obtain
\begin{align}
\label{415}
(\partial_N^2-A^2)(\partial^2_N-B_-^2)\widehat{u}_{J-}&=0,\\
\label{416}
(\partial_N^2-A^2)\widehat{\pi}_-=0
\end{align}
with
\begin{align*}
A=|\xi'|,\ B_-=\sqrt{|\xi'|^2+\rho_{*-}\mu_{*-}^{-1}\lambda}
\qquad ({\rm Re}\,B_->0).
\end{align*}
In view of (\ref{413}), (\ref{415}), and (\ref{416}), 
we look for solutions $\widehat{u}_{J\pm}$ and 
$\widehat{\pi}_-$ of the forms:
\begin{align}
\label{417}
\widehat{u}_{J+}(x_N)&=\alpha_{J+}e^{-B_+x_N}
+\beta_{J+}(e^{-t_1x_N}-e^{-B_+x_N})\\
&\qquad+\gamma_{J+}(e^{-t_2x_N}-e^{-B_+x_N}),\nonumber\\
\label{418}
\widehat{u}_{J-}(x_N)&=\alpha_{J-}e^{B_-x_N}
+\beta_{J-}(e^{B_-x_N}-e^{Ax_N}),\\
\label{419}
\widehat{\pi}_-(x_N)&=\gamma_-e^{Ax_N}.
\end{align}	
Using $A$ and $B_\pm$, we rewrite 
(\ref{403}), (\ref{404}), (\ref{406}), and (\ref{407}) 
as follows:
\begin{align}
\label{420}
\mu_{*+}\lambda(\partial_N^2-B_+^2)\widehat{u}_{j+}
+i\xi_j\{\nu_{*+}\lambda
-\rho_{*+}\kappa_{*+}(\partial^2_N-A^2) \}
\widehat{{\rm div}\,\mathbf{u_+}}&=0,\\
\label{421}
\mu_{*+}\lambda(\partial_N^2-B_+^2)\widehat{u}_{N+}
+\partial_N\{\nu_{*+}\lambda
-\rho_{*+}\kappa_{*+}(\partial^2_N-A^2) \}
\widehat{{\rm div}\,\mathbf{u_+}}&=0,\\
\label{422}
\mu_{*-}(\partial^2_N-B^2_-)
\widehat{u}_{j-}-i\xi_j\widehat{\pi}_-&=0,\\
\label{423}
\mu_{*-}(\partial_N^2-B^2_-)
\widehat{u}_{N-}-\partial_N\widehat{\pi}_-&=0.
\end{align}	
	
To state our solution formulas of equations: 
(\ref{402}) - (\ref{412}), we introduce 
some classes of multipliers.
\begin{defi}
Let $0<\varepsilon<\pi/2,\ \lambda_0\geq 0$, 
and let $s$ be a real number.
Set
\begin{align*}
\widetilde{\Sigma}_{\varepsilon,\lambda_0}
=\{(\lambda,\xi') \mid 
\lambda=\gamma+i\tau\in
\Sigma_{\varepsilon,\lambda_0},
\xi'=(\xi_1,\dots,\xi_{N-1})\in \mathbb{R}^{N-1}
\backslash\{0 \}\}.		
\end{align*}
Let $m(\lambda,\xi')$ be a function defined on 
$\widetilde{\Sigma}_{\varepsilon,\lambda_0} $ 
which is infinitely times differentiable with respect to 
$\tau$ and $\xi'$ when 
$(\lambda,\xi')\in \widetilde{\Sigma}_{\varepsilon,\lambda_0} $.
	
(1) If there exists a real number $s$ such that for 
any multi-index 
$\alpha'=(\alpha_1,\dots,\allowbreak\alpha_{N-1})
\in\mathbb{N}_0^{N-1}$
and $(\lambda,\xi')\in 
\widetilde{\Sigma}_{\varepsilon,\lambda_0}$ 
there hold the estimates:
\begin{align}
|\partial_{\xi'}^{\alpha'}m(\lambda,\xi')|
\leq C_{\alpha'}(|\lambda|^{1/2}+A)^{s-|\kappa'|},\\
\Big|\partial_{\xi'}^{\alpha'}
\Big(\tau \partial_\tau m(\lambda,\xi') \Big)\Big| 
\leq C_{\alpha'}(|\lambda|^{1/2}+A)^{s-|\kappa'|} \nonumber
\end{align}
for some constant $C_{\alpha'}$ depending on 
$s$, $\alpha'$, $\varepsilon$, $\mu_{*\pm}$, $\nu_{*+}$,
$\kappa_{*+}$, and $\rho_{*\pm}$.
Then, $m(\lambda,\xi')$ is called a multiplier of 
order $s$ with type 1.	
	
(2) If there exists a real number $s$ such that for 
any multi-index 
$\alpha'=(\alpha_1,\dots,\allowbreak\alpha_{N-1})
\in\mathbb{N}_0^{N-1}$
and $(\lambda,\xi')\in 
\widetilde{\Sigma}_{\varepsilon,\lambda_0}$ 
there hold the estimates:
\begin{align}
|\partial_{\xi'}^{\alpha'}m(\lambda,\xi')|
\leq C_{\alpha'}(|\lambda|^{1/2}+A)^s A^{-|\kappa'|},\\
\Big|\partial_{\xi'}^{\alpha'}
\Big(\tau\partial_\tau m(\lambda,\xi') \Big)\Big| 
\leq C_{\alpha'}(|\lambda|^{1/2}+A)^s A^{-|\kappa'|} 
\nonumber
\end{align}
for some constant $C_{\alpha'}$ depending on 
$s$, $\alpha'$, $\varepsilon$, $\mu_{*\pm}$, $\nu_{*+}$,
$\kappa_{*+}$, and $\rho_{*\pm}$.
Then, $m(\lambda,\xi')$ is called a multiplier of 
order $s$ with type 2.		
	
In what follows, we denote the set of multipliers defined 
on $\widetilde{\Sigma}_{\varepsilon,\lambda_0} $ of 
order $s$ with type $l\ (l=1,2)$ by 
$\mathbb{M}_{s,l,\varepsilon,\lambda_0}$.	
\end{defi}	
	
Obviously, $\mathbb{M}_{s,l,\varepsilon,\lambda_0}$ 
are the vector spaces on $\mathbb{C}$.
Furthermore, by the fact 
$|\lambda^{1/2}+A|^{-|\alpha'|}\leq A^{-|\alpha'|}$ 
and the Leibniz rule, we have the following lemma immediately.

\begin{lemm}\label{lem:multiplier}
Let $s_1,s_2\in \mathbb{R}$, $0<\varepsilon<\pi/2$,
$\lambda\in\Sigma_{\varepsilon,\lambda_0}$.	

(1) Given 
$m_i\in \mathbb{M}_{s_i,1,\varepsilon,\lambda_0}\ (i=1,2)$, 
we have 
$m_1m_2\in\mathbb{M}_{s_1+s_2,1,\varepsilon,\lambda_0}$.

(2) Given 
$l_i\in \mathbb{M}_{s_i,i,\varepsilon,\lambda_0}\ (i=1,2)$, 
we have 
$l_1l_2\in\mathbb{M}_{s_1+s_2,2,\varepsilon,\lambda_0}$.

(3) Given 
$n_i\in \mathbb{M}_{s_i,2,\varepsilon,\lambda_0}\ (i=1,2)$, 
we have 
$n_1n_2\in\mathbb{M}_{s_1+s_2,2,\varepsilon,\lambda_0}$.	
\end{lemm}	
	
\begin{rema}\label{rema.multi}
We easily see that 
$i\xi_j\in\mathbb{M}_{1,2,\varepsilon,0}$ ($j=1,\dots,N-1$),
and $A\in\mathbb{M}_{1,2,\varepsilon,0}$.
Especially, 
$i\xi_j/A\in\mathbb{M}_{0,2,\varepsilon,0}\ (j=1,\dots,N-1)$.
In addition, 
$\mathbb{M}_{s,1,\varepsilon,\lambda}
\subset \mathbb{M}_{s,2,\varepsilon,\lambda}$ 
for any $s\in\mathbb{R}$.	
\end{rema}	
	
Then, we arrive at the following  solution formulas for 
equations (\ref{402}) - (\ref{412}):
\begin{align}\label{4.428}
\widehat{u}_{J+}=&\sum_{i=1}^{4}\widehat{u}_{Ji}^+,\qquad 
\widehat{u}_{J-}=\sum_{i=1}^{3}\widehat{u}_{Ji}^-,
\\
\widehat{\rho}_+=&A M_{0+}(x_N)
\Big\{\sum_{m=1}^{N-1} P_{m,1}^+ 
\widehat{h}_m(0)+AP_{N,1}^+ \widehat{H}(0) \Big\}\nonumber\\
&+Ae^{-t_1 x_N}\Big\{\sum_{m=1}^{N-1} P_{m,2}^+ 
\widehat{h}_m(0)+AP_{N,2}^+ \widehat{H}(0) \Big\},\nonumber\\
\widehat{\pi}_-=&e^{Ax_N}\Big\{\sum_{m=1}^{N-1} P_m^- 
\widehat{h}_m(0)+AP_N^-\widehat{H}(0) \Big\},\nonumber\\
\widehat{u}_{J1}^+=&AM_{1+}(x_N)
\Big\{\sum_{m=1}^{N-1} Q_{Jm,1}^+ \widehat{h}_m(0)
+AQ_{JN,1}^+\widehat{H}(0) \Big\},
\nonumber\\
\widehat{u}_{J2}^+=&AM_{2+}(x_N)
\Big\{\sum_{m=1}^{N-1} Q_{Jm,2}^+ \widehat{h}_m(0) +
+AQ_{JN,2}^+\widehat{H}(0) \Big\},
\nonumber\\
\widehat{u}_{J3}^+=&Ae^{-B_+x_N}
\Big\{\sum_{m=1}^{N-1} R_{Jm}^+ \widehat{h}_m(0) 
+AR_{JN}^+\widehat{H}(0) \Big\},
\nonumber\\	
\widehat{u}_{j4}^+=&e^{-B_+x_N}
S^+_{j}\widehat{h}_j,\nonumber\\
\widehat{u}_{N4}^+=&0,\nonumber\\
\widehat{u}_{J1}^-=&AM_-(x_N)
\Big\{\sum_{m=1}^{N-1} Q_{Jm}^- \widehat{h}_m(0)  
+AQ_{JN}^-\widehat{H}(0) \Big\},
\nonumber\\	
\widehat{u}_{J2}^-=&Ae^{B_- x_N}
\Big\{\sum_{m=1}^{N-1} R_{Jm}^- \widehat{h}_m(0)
+AR_{JN}^-\widehat{H}(0) \Big\},
\nonumber\\	
\widehat{u}_{j3}^-=&e^{B_-x_N}
S^-_{j}\widehat{h}_j,\nonumber\\
\widehat{u}_{N3}^-=&0,\nonumber
\end{align}	
with
\begin{alignat}3\label{4.429}
P_{m,1}^+\in&\enskip \mathbb{M}_{-1,2,\varepsilon,0},
&\quad
P_{N,1}^+\in&\enskip\mathbb{M}_{-1,2,\varepsilon,0},
&\quad
P_{m,2}^+\in&\enskip \mathbb{M}_{-1,2,\varepsilon,0},
\\
P_{N,2}^+\in&\enskip\mathbb{M}_{-1,2,\varepsilon,0},
&\quad
P_{m}^-\in&\enskip \mathbb{M}_{1,2,\varepsilon,0},
&\quad
P_{N}^-\in&\enskip\mathbb{M}_{1,2,\varepsilon,0},
\nonumber\\
Q_{Jm,1}^+\in&\enskip \mathbb{M}_{0,2,\varepsilon,0},
&\quad
Q_{JN,1}^+\in&\enskip\mathbb{M}_{0,2,\varepsilon,0},
&\quad
Q_{Jm,2}^+\in&\enskip
\mathbb{M}_{0,2,\varepsilon,0},
\nonumber\\
Q_{JN,2}^+\in&\enskip\mathbb{M}_{0,2,\varepsilon,0},
&\quad
R_{Jm}^+\in&\enskip
\mathbb{M}_{-1,2,\varepsilon,0},
&\quad
R_{JN}^+\in&\enskip\mathbb{M}_{-1,2,\varepsilon,0},
\nonumber\\
S_{j}^+\in&\enskip\mathbb{M}_{0,1,\varepsilon,0},
&\quad
Q_{Jm}^-\in&\enskip\mathbb{M}_{0,2,\varepsilon,0},
&\quad
Q_{JN}^-\in&\enskip\mathbb{M}_{0,2,\varepsilon,0},
\nonumber\\
R_{Jm}^-\in&\enskip\mathbb{M}_{-1,2,\varepsilon,0},
&\quad
R_{JN}^-\in&\enskip\mathbb{M}_{-1,2,\varepsilon,0},
&\quad
S_{j}^-\in&\enskip\mathbb{M}_{0,1,\varepsilon,0}.
\nonumber
\end{alignat}	
Here and in the following, $J$ runs from $1$ through $N$.
Recall that $j$ and $m$ run from $1$ through $N-1$, respectively.
Furthermore, we define $M_{0+}(x_N)$, $M_{1+}(x_N)$,
$M_{2+}(x_N)$, and $M_{-}(x_N)$ as follows:	
\begin{align*}
M_{0+}(x_N)=&\frac{e^{-t_2x_N}-e^{-t_1x_N}}{t_2-t_1},\\
M_{i+}(x_N)=&\frac{e^{-t_ix_N}-e^{-B_+x_N}}{t_i-B_+}
\enskip(i=1,2),\\
M_-(x_N)=&\frac{e^{B_-x_N}-e^{Ax_N}}{B_--A}.
\end{align*}	
	
From now on, we prove (\ref{4.428}).
On the other hand, we prove (\ref{4.429}) 
in Sect.\ref{Sect.multipliers}.
By (\ref{419}), we obtain	
\begin{align*}
\widehat{{\rm div}\,\mathbf{u}_+}
=&(i\xi'\cdot\alpha'_+-i\xi'\cdot\beta'_+-i\xi'\cdot\gamma'_+
-B_+\alpha_{N+}+B_+\beta_{N+}+B_+\gamma_{N+})e^{-B_+x_N}
\nonumber\\
&+(i\xi'\cdot\beta'_+-t_1\beta_{N+})e^{-t_1x_N}
+(i\xi'\cdot\beta'_+-t_2\beta_{N+})e^{-t_2x_N}
\nonumber
\end{align*}
with
\begin{align*}
i\xi'\cdot v'=\sum_{j=1}^{N-1}i\xi_j v_j\quad
\text{ for $\mathbf{v}=(v_1,\dots,v_{N-1},v_N)$}.
\end{align*}
Then, from (\ref{420}) and (\ref{421}), we have
\begin{align}	
\label{424}
i\xi'\cdot\alpha'_+-i\xi'\cdot\beta'_+
-i\xi'\cdot\gamma'_+-B_+\alpha_{N+}+B_+\beta_{N+}
+B_+\gamma_{N+}&=0,\\
\mu_{*+}\lambda\beta_{j+}(t_1^2-B_+^2)
+i\xi_j (i\xi'\cdot\beta'_+-t_1\beta_{N+})
\{\nu_{*+}\lambda-\rho_{*+}\kappa_{*+}(t_1^2-A^2) \}&=0,
\nonumber\\
\mu_{*+}\lambda\gamma_{j+}(t_1^2-B_+^2)
+i\xi_j (i\xi'\cdot\gamma'_+-t_2\gamma_{N+})
\{\nu_{*+}\lambda-\rho_{*+}\kappa_{*+}(t_2^2-A^2) \}&=0,
\nonumber\\
\mu_{*+}\lambda\beta_{N+}(t_1^2-B_+^2)
-t_1 (i\xi'\cdot\beta'_+-t_1\beta_{N+})
\{\nu_{*+}\lambda-\rho_{*+}\kappa_{*+}(t_1^2-A^2) \}&=0,
\nonumber\\
\mu_{*+}\lambda\gamma_{N+}(t_2^2-B_+^2)
-t_2 (i\xi'\cdot\gamma'_+-t_2\gamma_{N+})
\{\nu_{*+}\lambda-\rho_{*+}\kappa_{*+}(t_2^2-A^2) \}&=0,
\nonumber
\end{align}
which furnishes that
\begin{align}
\label{425}
&\widehat{{\rm div}\,\mathbf{u}_+}
=(i\xi'\cdot\beta'_+-t_1\beta_{N+})e^{-t_1x_N}
+(i\xi\cdot\gamma'_+-t_2\gamma_{N+})e^{-t_2x_N},\\
\label{426}
&\beta_{j+}=-\cfrac{i\xi_j}{t_1}\beta_{N+},\quad 
\gamma_{j+}=\cfrac{i\xi_j}{t_2}\gamma_{N+}.
\end{align}
By (\ref{405}), (\ref{422}), and (\ref{423}), we have
\begin{align}
\label{427}
&i\xi'\cdot\alpha'_-+i\xi'\cdot\beta_-
+B_-\alpha_{N+}+B_-\beta_{N-}=0,\\
\label{428}
&\mu_{*-}(B^2_--A^2)\beta_{j-}-i\xi_j\gamma_-=0,\\
\label{429}
&\mu_{*-}(B^2_--A^2)\beta_{N-}-A\gamma_-=0.
\end{align}	
Combining (\ref{427}) and (\ref{428}), 
we deduce $i\xi'\cdot\beta'_-+A\beta_{N-}=0$. 
Hence, by (\ref{427}), (\ref{428}), and (\ref{429}), 
we observe	
\begin{align}
\label{430}
&i\xi\cdot\beta'_-
=\frac{A}{B_--A}(i\xi'\cdot\alpha'_-+B_-\alpha_{N-}),\\
\label{431}
&\beta_{N-}
=\frac{1}{A-B_-}(i\xi'\cdot\alpha'_-+B_-\alpha_{N-}),\\
\label{432}
&\gamma_-
=-\frac{\mu_{*-}(A+B_-)}{A}(i\xi'\cdot\alpha'_-+B_-\alpha_{N-}).
\end{align}	
	
Next, we consider the interface condition.
From (\ref{402}) and (\ref{410}), we have
\begin{align}\label{433}
\lambda\sigma_+A^2\widehat{H}(0)
=&2\mu_{*+}\lambda\{-B_+\alpha_{N+}-(t_1-B_+)\beta_{N+}
-(t_2-B_+)\gamma_{N+} \}\\
&+\{\lambda(\nu_{*+}-\mu_{*+})
-\rho_{*+}\kappa_{*+}(t_1^2-A^2) \}
(i\xi'\cdot\beta'_+-t_1\beta_{N+})\nonumber\\
&+\{\lambda(\nu_{*+}-\mu_{*+})
-\rho_{*+}\kappa_{*+}(t_2^2-A^2) \}
(i\xi'\cdot\gamma'_+-t_2\gamma_{N+}).\nonumber
\end{align}
Substituting (\ref{424}) and (\ref{426}) into (\ref{433}) to 
obtain
\begin{align}\label{434}
\alpha_{N+}
=&-\frac{\sigma_+A^2\widehat{H}(0)}{2\mu_{*+}B_+}
+\frac{1}{2t_1B_+}(2t_1B_+-B_+^2-A^2)\beta_{N+}
+\frac{1}{2t_2B_+}(2t_2B_+-B_+^2-A^2)\gamma_{N+}
\end{align}
because $\lambda\mu_{*+}\neq0$.	
In addition, by (\ref{424}) and (\ref{434}), it follows that
\begin{align} \label{435}
i\xi'\cdot\alpha'_+
=&-\frac{\sigma_+A^2\widehat{H}(0)}{2\mu_{*+}}
+\frac{A^2-B^2_+}{2t_1}\beta_{N+}
+\frac{A^2-B^2_+}{2t_2}\gamma_{N+}.
\end{align}
Together with (\ref{411}) and (\ref{435}), this shows	
\begin{align}\label{436}
i\xi'\cdot\alpha'_-
=&-\frac{\sigma_+A^2\widehat{H}(0)}{2\mu_{*+}}
+i\xi'\cdot\widehat{h'}(0)
+\frac{A^2-B^2_+}{2t_1}\beta_{N+}
+\frac{A^2-B^2_+}{2t_2}\gamma_{N+}.
\end{align}	
By (\ref{409}), we have
\begin{align*}
2\mu_{*-}\{B_-\alpha_{N-}+(B_--A)\beta_{N-}\}-\gamma_-
=\sigma_-A^2\widehat{H}(0).
\end{align*}
Combining with (\ref{431}) and (\ref{432}), this yields
\begin{align}\label{438}
\alpha_{N-}
=\frac{\sigma_-A^3\widehat{H}(0)}{\mu_{*-}(A+B_-)B_-}
+\frac{A-B_-}{(A+B_-)B_-}i\xi'\cdot\alpha'_-.
\end{align}
From (\ref{408}), we have
\begin{align}\label{439}
0
=&\mu_{*-}(B_-\alpha_{m-}+(B_--A)\beta_{m-}+i\xi_m\alpha_{N-})\\
&-\mu_{*+}(-B_+\alpha_{m+}-(t_1-B_+)\beta_{m+}
-(t_2-B_+)\gamma_{m+}+i\xi_m\alpha_{N+}),\nonumber
\end{align}	
which implies
\begin{align}\label{440}
0
=&\mu_{*-}(B_-i\xi'\cdot\alpha'_{-}
+(B_--A)i\xi'\cdot\beta'_{-}-A^2\alpha_{N-})\\
&-\mu_{*+}(-B_+i\xi'\cdot\alpha'_{+}
-(t_1-B_+)i\xi'\cdot\beta'_{+}
-(t_2-B_+)i\xi'\cdot\gamma'_{+}-A^2\alpha_{N+}).\nonumber
\end{align}
Substituting (\ref{426}), (\ref{430}), (\ref{434}), 
(\ref{435}), (\ref{436}), and (\ref{438}) into 
(\ref{440}), we obtain
\begin{align}\label{441}
&t_2(Dt_1-E)\beta_{N+}+t_1(Dt_2-E)\gamma_{N+}
=2t_1t_2\{A^2F\widehat{H}(0)
-Gi\xi'\cdot\widehat{h'}(0)\}
\end{align}
with	
\begin{align}\label{442}
D=&4\mu_{*+}A^2B_+B_-(A+B_-),\\
E=&\mu_{*+}B_-(A+B_-)(A^2+B_+^2)^2\nonumber\\
&+\mu_{*-}B_+(A^2-B_+^2)(A^3-3AB_--AB_-^2-B_+^3)\nonumber,\\
F=&
\frac{\sigma_+B_+}{2\mu_{*+}}
\{\mu_{*+}(A+B_-)(A^2+B_+B_-)\nonumber\\
&+\mu_{*-}(B_-^3+B_-^2A+3B_-A^2-A^3) \}
-\sigma_-A^2B_+(B_--A),\nonumber\\
G=&
\mu_{*+}B_+^2B_-(A+B_-)
+\mu_{*-}B_+(B_+^3+B_-^3+2B_-^2A+3B_-A^2-A^3)
\nonumber.
\end{align}	
By (\ref{402}) and (\ref{412}), we have
\begin{align}\label{443.1}
\rho_{*+}(t_1^2-A^2)\beta_{N+}+\rho_{*+}(t_2^2-A^2)\gamma_{N+}
=0.
\end{align}
Consequently, by (\ref{441}) and (\ref{443.1}), we have
\begin{align}\label{443}
L
\left(\begin{array}{cccc}
\beta_{N+}\\\gamma_{N+}	 
\end{array} \right)
=
\left(\begin{array}{cccc}
2t_1t_2\{A^2F\widehat{H}(0)
-Gi\xi'\cdot\widehat{h'}(0)\} \\
0
\end{array} \right),
\end{align}
where	
\begin{align*}
L=
\left(\begin{array}{cccc}
t_2(Dt_1-E) & t_1(Dt_2-E)  \\
\rho_{*+}(t_1^2-A^2) & \rho_{*+}(t^2_2-A^2) 
\end{array} \right).
\end{align*}	
By direct calculations, we have
\begin{align}\label{444}
\det L&
=\rho_{*+}(t_1-t_2)\{E(t^2_1+t_1t_2+t_2^2-A^2)
-Dt_1t_2(t_1+t_2)\}.
\end{align}
According to Saito \cite{Saito},
we have following formula:
\begin{align*}
&(A^2+B_+^2)^2(t^2_1+t_1t_2+t_2^2-A^2)-4A^2B_+t_1t_2(t_1+t_2)\\
=&
\frac{\lambda\mathfrak{m}_1(\lambda,\xi')}
{t_1(t_1+B_+)}
=
\frac{\lambda\mathfrak{m}_2(\lambda,\xi')}
{t_2(t_2+B_+)}
=:\lambda \mathfrak{n}(\lambda,\xi')\nonumber
\end{align*}
with
\begin{align*}
\mathfrak{m}_i
=&
\rho_{*+}^{2}\mu_{*+}^{-2}\lambda t_i(t_i+B_+)
(t_1^2+t_1t_2+t_2^2-A^2)\\
&+4A^2B_+\{s_it_iB_+(t_i+B_+)-(s_i-\rho_{*+}\mu_{*+}^{-1})
t_1t_2(t_1+t_2) \} \quad (i=1,2)
\end{align*}	
Then, we rewrite (\ref{444}) as follows:
\begin{align}\label{det:L}
\det L
=&\rho_{*+}\lambda(t_1-t_2)
\Big( \mu_{*+}B_-(A+B_-)\mathfrak{n}(\lambda,\xi')\\
&-\rho_{*+}\mu_{*+}\mu_{*-}B_+(A^3-3A^2B_--AB_-^2-B_+^3)
(t^2_1+t_1t_2+t_2^2-A^2) \Big)\nonumber\\
=&:\rho_{*+}\lambda(t_1-t_2)\mathfrak{l}(\lambda,\xi').
\nonumber
\end{align}
If $\det L\neq 0$, the inverse of $L$ exists and we see
\begin{align}\label{445}
\left(\begin{array}{cccc}
\beta_{N+}\\\gamma_{N+}	 
\end{array} \right)
=L^{-1}
\left(\begin{array}{cccc}
2t_1t_2\{A^2F\widehat{H}(0)
-Gi\xi'\cdot\widehat{h'}(0)\} \\
0
\end{array} \right)
\end{align}	
with
\begin{align}\label{446}
L^{-1}
=
\frac{1}{\det L}
\left(\begin{array}{cccc}
\rho_{*+}(t^2_2-A^2) & -t_1(Dt_2+E)  \\
-\rho_{*+}(t^2_1-A^2) & t_2(Dt_1+E)
\end{array} \right)
=:
\frac{1}{\det L}
\left(\begin{array}{cccc}
L_{11} & L_{12}  \\
L_{21} & L_{22}
\end{array} \right).
\end{align}	
In this section, we assume $\det L\neq 0$ and continue to obtain 
the solution formula.
We shall prove $\det L\neq 0$ when 
$(\lambda,\xi')\in \widetilde{\Sigma}_{\varepsilon,\lambda_0}$ 
in Sect. \ref{Sect.Lopatinski}.
By (\ref{445}) and (\ref{446}), we obtain	
\begin{align*}
\beta_{N+}
=&-\frac{2t_1t_2GL_{11}}{\det L}i\xi'\cdot\widehat{h'}(0)
+\frac{2t_1t_2A^2FL_{11}}{\det L}\widehat{H}(0) ,\nonumber\\
\gamma_{N+}
=&-\frac{2t_1t_2GL_{21}}{\det L}i\xi'\cdot\widehat{h'}(0)
+\frac{2t_1t_2A^2FL_{21}}{\det L}\widehat{H}(0) \nonumber.
\end{align*}
{\allowdisplaybreaks
Setting	
\begin{align}\label{448}
&Q^+_{Nm,1}
=-\frac{2i\xi_mt_1t_2GL_{11}(t_1-B_+)}{A\det L},
\quad Q_{NN,1}^+=\frac{2t_1t_2FL_{11}(t_1-B_+)}{\det L},
\\
&Q^+_{Nm,2}
=-\frac{2i\xi_mt_1t_2GL_{21}(t_2-B_+)}{A\det L},
\quad Q_{NN,2}^+=\frac{2t_1t_2FL_{21}(t_2-B_+)}{\det L},
\nonumber
\end{align}
with $m=1,\dots,N-1$,
we have	
\begin{align}\label{449}
\beta_{N+}=
&\frac{A}{t_1-B_+}\Big\{\sum_{m=1}^{N-1} 
Q_{Nm,1}^+ \widehat{h}_m(0) +AQ_{NN,1}^+\widehat{H}(0)
\Big\},\\
\gamma_{N+}=
&\frac{A}{t_2-B_+}\Big\{\sum_{m=1}^{N-1} 
Q_{Nm,2}^+\widehat{h}_m(0) +AQ_{NN,2}^+\widehat{H}(0)
\Big\}.\nonumber
\end{align}
From (\ref{426}), we have
\begin{align*}
\beta_{j+}=
&\frac{A}{t_1-B_+}
\Big\{\sum_{m=1}^{N-1}Q_{jm,1}^+ \widehat{h}_m(0) 
+AQ_{jN,1}^+\widehat{H}(0)
\Big\},\\
\gamma_{j+}=
&\frac{A}{t_2-B_+}
\Big\{\sum_{m=1}^{N-1}Q_{jm,2}^+ \widehat{h}_m(0)
+AQ_{jN,2}^+\widehat{H}(0)\Big\}.\nonumber
\end{align*}
with	
\begin{align*}
&Q^+_{jm,1}
=-\frac{2\xi_m \xi_j t_2GL_{11}(t_1-B_+)}{A\det L},
\quad Q_{jN,1}^+=-\frac{2i\xi_jt_2FL_{11}(t_1-B_+)}
{\det L},
\nonumber\\
&Q^+_{jm,2}
=-\frac{2\xi_m \xi_j t_1GL_{21}(t_2-B_+)}{A\det L},
\quad Q_{jN,2}^+=-\frac{2i\xi_jt_2FL_{21}(t_2-B_+)}
{\det L}.
\nonumber
\end{align*}
Furthermore, combined with (\ref{425}), (\ref{426}), 
and (\ref{449}),
we have	
\begin{align*}
\widehat{\rho}_+=&AM_{0+}(x_N) 
\Big\{\sum_{m=1}^{N-1} P_{m,1}^+ 
\widehat{h}_m(0)+AP_{N,1}^+ \widehat{H}(0) \Big\}
+Ae^{-t_1 x_N}\Big\{\sum_{m=1}^{N-1} P_{m,2}^+ 
\widehat{h}_m(0)+AP_{N,2}^+ \widehat{H}(0) \Big\}
\end{align*}
with
\begin{alignat*}2
P_{m,1}^+&=-\frac{2\rho_{*+}s_1s_2i\xi_mt_1G}
{A\mathfrak{l}(\lambda,\xi')},
&\quad
P_{N,1}^+&=\frac{2\rho_{*+}s_1s_2t_1F}
{\mathfrak{l}(\lambda,\xi')},
\\
P_{m,2}^+&=\frac{2\rho_{*+}s_1s_2i\xi_m(t_2-t_1)G}
{A\mathfrak{l}(\lambda,\xi')},
&\quad
P_{N,2}^+&=\frac{2\rho_{*+}s_1s_2(t_1-t_2)F}
{\mathfrak{l}(\lambda,\xi')}.
\end{alignat*}	
By (\ref{434}), we have
\begin{align*}
\alpha_{N+}
=A\Big\{\sum_{m=1}^{N-1} R_{Nm}^+ \widehat{h}_m(0) 
+AR_{NN}^+\widehat{H}(0) \Big\}
\end{align*}
with
\begin{align}\label{452}
R_{Nm}^+
=&
-\frac{i\xi_mG\{s_1(2t_2B_+-B_+^2-A^2)t_1
+s_2(2t_1B_+-B_+^2-A^2)t_2 \}}
{(t_1-t_2)AB_+\mathfrak{l}(\lambda,\xi')},\\
R^+_{NN}=&
\frac{F\{s_1(2t_2B_+-B_+^2-A^2)t_2+s_2(2t_1B_+-B_+^2-A^2)t_1\}}
{(t_1-t_2)B_+\mathfrak{l}(\lambda,\xi')}
-\frac{\sigma_+}{2\mu_{*+}B_+}.\nonumber
\end{align}	
Substituting (\ref{449}) into (\ref{436}) to obtain
\begin{align}\label{453}
i\xi'\cdot\alpha'_-
=&\sum_{m=1}^{N-1}
\Big(\frac{A^2(A^2-B_+^2)}{2t_1(t_1-B_+)} Q^+_{Nm,1}
+\frac{A^2(A^2-B_+^2)}{2t_2(t_2-B_+)} Q^+_{Nm,2}+i\xi_m
\Big)\widehat{h_m}(0)\\
&+\Big(\frac{A^3(A^2-B_+^2)}{2t_1(t_1-B_+)}Q^+_{NN,1}
+\frac{A^3(A^2-B_+^2)}{2t_2(t_2-B_+)}Q^+_{NN,2}
-\frac{\sigma_+A^2}{2\mu_{*+}}
\Big)\widehat{H}(0).
\nonumber
\end{align}
Combining (\ref{438}) and (\ref{453}), we obtain
\begin{align}\label{454}
\alpha_{N-}
=A\Big\{\sum_{m=1}^{N-1} R_{Nm}^- \widehat{h}_m(0)
+AR_{NN}^-\widehat{H}(0) \Big\}
\end{align}
with
\begin{align}\label{455}
R^-_{Nm}
=&\frac{A-B_-}{(A+B_-)B_-}
\Big(\frac{A(A^2-B_+^2)}{2t_1(t_1-B_+)} Q^+_{Nm,1}
+\frac{A(A^2-B_+^2)}{2t_2(t_2-B_+)} Q^+_{Nm,2}
+\frac{i\xi_m}{A}\Big),\\
R^-_{NN}
=&\frac{A-B_-}{(A+B_-)B_-}
\Big(\frac{A(A^2-B_+^2)}{2t_1(t_1-B_+)}Q^+_{NN,1}
+\frac{A(A^2-B_+^2)}{2t_2(t_2-B_+)}Q^+_{NN,2}
-\frac{\sigma_+}{2\mu_{*+}}
\Big)
\nonumber\\
&+\frac{\sigma_-A}{\mu_{*-}(A+B_-)B_-}.\nonumber
\end{align}	 
 Substituting (\ref{453}) and (\ref{454}) 
into (\ref{431}) and (\ref{432}), we have
\begin{align}\label{456}
\beta_{N-}=
&\frac{A}{B_--A}
\Big\{\sum_{m=1}^{N-1} Q_{Nm}^- \widehat{h}_m(0)
+AQ_{NN}^-\widehat{H}(0) \Big\},\\
\gamma_-=
&\sum_{m=1}^{N-1} 
P_m^- \widehat{h}_m(0)+AP_N^-\widehat{H}(0), \nonumber
\end{align}
respectively. Here we set
\begin{align*}
Q^-_{Nm}
=&-\frac{1}{A+B_-}
\Big(\frac{A^2(A^2-B_+^2)}{2t_1(t_1-B_+)} Q^+_{Nm,1}
+\frac{A^2(A^2-B_+^2)}{2t_2(t_2-B_+)} Q^+_{Nm,2}+i\xi_m
\Big),
\nonumber\\
Q^-_{NN}
=&-\frac{1}{A+B_-}
\Big(\frac{A^2(A^2-B_+^2)}{2t_1(t_1-B_+)}Q^+_{NN,1}
+\frac{A^2(A^2-B_+^2)}{2t_2(t_2-B_+)}Q^+_{NN,2}
-\frac{\sigma_+A}{2\mu_{*+}}\Big)
-\frac{\sigma_-A}{\mu_{*-}(A+B_-)},\nonumber\\
P_m^-
=&- \mu_{*-}(A+B_-) Q^-_{Nm},\nonumber\\
P_N^-
=&- \mu_{*-}(A+B_-) Q^-_{NN},
\nonumber
\end{align*}	
for short.
From (\ref{428}) and (\ref{456}), we have
\begin{align}\label{457}
\beta_{j-}=
&\frac{A}{B_--A}\Big\{\sum_{m=1}^{N-1} Q_{jm}^- \widehat{h}_m(0)  
+AQ_{jN}^-\widehat{H}(0) \Big\}
\end{align}
with
\begin{align*}
&Q_{jm}^-
=-\frac{i\xi_j}{A} Q^-_{Nm} ,
\quad Q_{jN}^-=-\frac{i\xi_j}{A} Q_{NN}^- .
\end{align*}
Accordingly, by (\ref{411}), (\ref{439}), and (\ref{457}), 
we obtain	 
\begin{align*}
\alpha_{j+}
=&A\Big\{\sum_{m=1}^{N-1} R_{jm}^+ \widehat{h}_m(0)
+AR_{jN}^+\widehat{H}(0) \Big\}
+S^+_{j}\widehat{h}_j,\quad\\
\alpha_{j-}
=&A\Big\{\sum_{m=1}^{N-1} R_{jm}^- \widehat{h}_m(0) 
+AR_{jN}^-\widehat{H}(0) \Big\}
+S^-_{j}\widehat{h}_j,\nonumber
\end{align*}
with
\begin{align*}
R^+_{jm}
=&-(\mu_{*+}B_++\mu_{*-}B_-)^{-1}
\Big(\mu_{*+}(AQ^+_{jm,1}+AQ^+_{jm,2}-i\xi_j R^+_{Nm})
+\mu_{*-}(Q^-_{jm}+i\xi_jR_{Nm}^-) \Big),
\nonumber\\
R^+_{jN}
=&-(\mu_{*+}B_++\mu_{*-}B_-)^{-1}
\Big(\mu_{*+}(AQ^+_{jN,1}+AQ^+_{jN,2}-i\xi_j R^+_{NN})
+\mu_{*-}(Q^-_{jN}+i\xi_jR_{NN}^-) \Big),
\nonumber\\
S_{j}^+
=&-\mu_{*-}B_-(\mu_{*+}B_++\mu_{*-}B_-)^{-1},\quad
R^-_{jm}
=R^+_{jm},\quad
R^-_{jN}
=R^+_{jN},\nonumber\\
S_{j}^-
=&\mu_{*+}B_+(\mu_{*+}B_++\mu_{*-}B_-)^{-1}.
\nonumber
\end{align*}	
This completes the proof of (\ref{4.428}).
	 
To prove Theorem \ref{Th05}, we consider problem (\ref{204}),
namely, problem (\ref{401}) with $H=0$.
First of all, we define our solution operators 
$\mathcal{A}_{Ji}^{3+}(\lambda)$
$(i=1,2,3,4 )$, $\mathcal{A}_{Ji}^{3-}(\lambda)$
$(i=1,2,3)$, $\mathcal{B}_+^3(\lambda)$,
and $\mathcal{P}_-^3(\lambda)$ of problem (\ref{4.428})
such that
\begin{alignat*}2
u^+_{Ji}&=\mathcal{A}_{Ji}^{3+}(\lambda) \mathbf{h} 
&\quad
&\text{ on $\mathbb{R}^N_+\enskip (i=1,2,3,4 )$},\\
u^-_{Ji}&=\mathcal{A}_{Ji}^{3-}(\lambda) \mathbf{h} 
&\quad
&\text{ on $\mathbb{R}^N_-\enskip (i=1,2,3)$}\nonumber,\\
\rho_+&=\mathcal{B}^3_+(\lambda) \mathbf{h} 
&\quad
&\text{ on $\mathbb{R}^N_+$}\nonumber,\\
\pi_-&=\mathcal{P}_-^3(\lambda) \mathbf{h} 
&\quad
&\text{ on $\mathbb{R}^N_-$}\nonumber
\end{alignat*}
with
\begin{align*}
\mathcal{A}_{J1}^{3+}(\lambda) \mathbf{h}
=&\sum_{m=1}^{N-1}\mathcal{F}^{-1}_{\xi'}\Big[
Q _{Jm,1}^+ A M_{1+}(x_N)\widehat{h}_m(0)
\Big](x'),\nonumber\\
\mathcal{A}_{J2}^{3+}(\lambda) \mathbf{h}
=&\sum_{m=1}^{N-1}\mathcal{F}^{-1}_{\xi'}\Big[
Q _{Jm,2}^+ A M_{2+}(x_N)\widehat{h}_m(0)
\Big](x'),\nonumber\\
\mathcal{A}_{J3}^{3+}(\lambda) \mathbf{h}
=&\sum_{m=1}^{N-1}\mathcal{F}^{-1}_{\xi'}
\Big[R ^+_{Jm} Ae^{-B_+x_N}\widehat{h}_m(0)
\Big](x'),\nonumber\\
\mathcal{A}_{j4}^{3+}(\lambda) \mathbf{h}
=&\mathcal{F}^{-1}_{\xi'}
\Big[S^+_{j}e^{-B_+x_N}\widehat{h}_j(0)
\Big](x'),\nonumber\\
\mathcal{A}_{N4}^{3+}(\lambda) \mathbf{h}
=&0,\nonumber\\
\mathcal{A}_{J1}^{3-}(\lambda) \mathbf{h}
=&\sum_{m=1}^{N-1}\mathcal{F}^{-1}_{\xi'}\Big[
Q_{Jm}^- A M_{-}(x_N)\widehat{h}_m(0)
\Big](x'),\nonumber\\
\mathcal{A}_{J2}^{3-}(\lambda) \mathbf{h}
=&\sum_{m=1}^{N-1}\mathcal{F}^{-1}_{\xi'}
\Big[R ^-_{Jm} Ae^{B_-x_N}\widehat{h}_m(0)
\Big](x'),\nonumber\\
\mathcal{A}_{j3}^{3-}(\lambda) \mathbf{h}
=&\mathcal{F}^{-1}_{\xi'}
\Big[S^-_{j}e^{B_-x_N}\widehat{h}_j(0)
\Big](x'),\nonumber\\
\mathcal{A}_{N3}^{3-}(\lambda) \mathbf{h}
=&0,\nonumber\\
\mathcal{B}_{+}^{3}(\lambda) \mathbf{h}
=&\mathcal{F}^{-1}_{\xi'}
\sum_{m=1}^{N-1}\Big[
P^+_{m,1} A M_{0+}(x_N) \widehat{h}_m(0)
\Big](x')+\mathcal{F}^{-1}_{\xi'}
\sum_{m=1}^{N-1}\Big[
P^+_{m,2} Ae^{-t_1x_N} \widehat{h}_m(0)
\Big](x'),\nonumber\\
\mathcal{P}_{-}^{3}(\lambda) \mathbf{h}
=&\sum_{m=1}^{N-1}\mathcal{F}^{-1}_{\xi'}
\Big[ P ^-_{m} e^{Ax_N}\widehat{h}_m(0)
\Big](x').\nonumber
\end{align*}	 
In order to prove Theorem \ref{Th05},
we introduce following lemma and corollary.
\begin{lemm}\label{lem.4.4}
Let $1<q<\infty$, $\lambda_0\geq0$, 
$\varepsilon_*<\varepsilon<\pi/2$.
Assume that $\rho_{*+}\neq\rho_{*-}$, $\eta_*\neq 0$,
and $\kappa_{*+}\neq \mu_{*+}\nu_{*+}$.
For $m(\lambda,\xi')\in \mathbb{M}_{0,2,\varepsilon,\lambda_0}$,
$i=0,1,2,3$, and $j=1,2$,
we define operators $J_i(\lambda)$, $K_i(\lambda)$, 
and $L_j(\lambda)$ by		 
\begin{align*}
[J_i(\lambda)f](x)&=\int_{0}^{\infty}\mathcal{F}^{-1}_{\xi'}
\Big[m(\lambda,\xi') \lambda^{1/2}AM_{i+}(x_N+y_N)
\widehat{f}(\xi',y_N) \Big](x')
\,dy_N,\\
[K_i(\lambda)f](x)&=\int_{0}^{\infty}\mathcal{F}^{-1}_{\xi'}
\Big[m(\lambda,\xi') A^2M_{i+}(x_N+y_N)
\widehat{f}(\xi',y_N) \Big](x')
\,dy_N,\\
[L_j(\lambda)f](x)&=\int_{0}^{\infty}\mathcal{F}^{-1}_{\xi'}
\Big[m(\lambda,\xi') Ae^{-t_j(x_N+y_N)}
\widehat{f}(\xi',y_N) \Big](x')
\,dy_N.
\end{align*}	 
Then, for $i=0,1,2$, $j=1,2$, and $s=0,1$,
the sets $\{(\tau\partial\tau)^s J_i(\lambda) \}$,
$\{(\tau\partial\tau)^s K_i (\lambda)\}$, and
$\{(\tau\partial\tau)^s L_j (\lambda)\}$ are 
$\mathcal{R}$-bounded families in 
$\mathcal{L}(L_q(\mathbb{R}^N_+))$,
whose $\mathcal{R}$-bounds do not exceed some constant
$C_{N,q,\lambda_0,\varepsilon,\mu_{*+},\nu_{*+},\kappa_{*+}}$
depending essentially only on 
$N$, $q$, $\lambda_0$, $\varepsilon$, $\mu_{*+}$,
$\nu_{*+}$, and $\kappa_{*+}$.
\end{lemm}
\begin{proof}
First we consider $K_0(\lambda)$.
Setting $$k_{0,\lambda}(x)=
\mathcal{F}^{-1}_{\xi'}[m(\lambda,\xi')A^2M_{0+}(x_N)](x'),$$
we have	 
\begin{align*}
[K_0(\lambda)f](x)=\int_{\mathbb{R}^N_+}
k_{0,\lambda}(x'-y',x_N+y_N)f(y)\,dy.
\end{align*}
Employing the same argumentation
due to Shibata and Shimizu \cite[Lemma 5.4]{SS2012},
it is sufficient to prove
\begin{align}\label{458}
|(\tau\partial_\tau)^sk_{0,\lambda}(x)|\leq
C_{N,\lambda_0,\varepsilon,\mu_{*+},\nu_{*+},\kappa_{*+}}
|x|^{-N}\quad(s=0,1).
\end{align}
According to Saito \cite[Lemma 4.8]{Saito}, we have
\begin{align*}
|\partial_{\xi'}^{\alpha'}\{(\tau\partial_\tau)^s
M_{0+}(x_N) \}|\leq Cx_N(|\lambda|^{1/2}+A)^{-|\alpha'|}
e^{-b(|\lambda|^{1/2}+A)x_N}
\quad (s=0,1).
\end{align*}
Here and in the sequel, $b$ is a positive constant 
depending on $\varepsilon$,
$\mu_{*+}$, $\nu_{*+}$, and $\kappa_{*+}$. 
On the other hand, $C$ is a positive constant depending on
$\alpha'$, $\lambda_0$, $\varepsilon$, $\mu_{*+}$, $\nu_{*+}$,
and $\kappa_{*+}$.	
Then, by the Leibniz rule and the assumption we have
\begin{align}\label{459}
&|\partial_{\xi'}^{\alpha'}\{
(\tau\partial_\tau)^s m(\lambda,\xi')A^2M_{0+}(x_N) \}|\\
&\quad\leq
C_{\alpha',\lambda_0,\varepsilon,\mu_{*+},\nu_{*+},\kappa_{*+}}
A^{1-|\alpha'|}e^{-(b/2)(|\lambda|^{1/2}+A) x_N}
\quad(s=0,1)\nonumber,
\end{align}
which, combined with Theorem 3.6 in Shibata and Shimizu
\cite{SS2012}, furnishes that
\begin{align*}
|(\tau\partial_\tau)^sk_{0,\lambda}(x)|\leq
C_{N,\lambda_0,\varepsilon,\mu_{*+},\nu_{*+},\kappa_{*+}}
|x'|^{-N}\quad(s=0,1).
\end{align*}	 
On the other hand, using (\ref{459}) with
$s=0$ and $\alpha'=0$, for $s=0,1$, we have
\begin{align*}
|(\tau\partial_\tau)^s k_{0,\lambda}(x)|
&\leq C_{\lambda_0,\varepsilon,\mu_{*+},\nu_{*+},\kappa_{*+}}
\Big(\frac{1}{2\pi}\Big)^{N-1}\int_{\mathbb{R}^{N-1}}
Ae^{-(b/2) Ax_N} \,d\xi'\\
&\leq |x_N|^{-N} 
C_{\lambda_0,\varepsilon,\mu_{*+},\nu_{*+},\kappa_{*+}}
\Big(\frac{1}{2\pi}\Big)^{N-1}\int_{\mathbb{R}^{N-1}}
|\eta'|e^{-(b/2)|\eta'|} \,d\eta',
\end{align*}
which, combined with (\ref{459}), implies (\ref{458}).
Thus this completes the case $K_0(\lambda)$.	 
	 
Next we consider $K_1(\lambda)$ and $K_2(\lambda)$.
By the identities:
\begin{align}\label{identity}
B_+^2=&A^2+\frac{\rho_{*+}\lambda}{\mu_{*+}}
=-\sum_{m=1}^{N-1}(i\xi_m)^2+\frac{\rho_{*+}\lambda}{\mu_{*+}},\\
t_j^2=&A^2+s_j\lambda=\sum_{m=1}^{N-1}(i\xi_m)^2+s_j\lambda
\quad (j=1,2),\nonumber
\end{align}
we have for $j=1,2$
\begin{align*}
M_{j+}(x_N)=\mathfrak{r}_j(\lambda,\xi')
\Big(\frac{e^{-t_jx_N} -e^{-B_+x_N}}{t_j-B_+} \Big),\quad
\mathfrak{r}_j(\lambda,\xi')
=\frac{(s_j-\rho_{*+}\mu_{*+}^{-1})(t_1+t_2)}
{(s_2-s_1)(t_j+B_+)}.
\end{align*}	 
Since $M_{j+}(x_N)=-\mathfrak{r}_j(\lambda,\xi')x_N
\int_{0}^{1}e^{-\{\theta t_j +(1-\theta)B_+\}x_N}\,d\theta$
and $\mathfrak{r}_{j}(\lambda,\xi')
\in \mathbb{M}_{0,1,\varepsilon,0}$ by (\ref{505}) below,
we can prove the required properties in the same manner 
as the case $K_{0}(\lambda)$.
In addition, we can prove the case 
$J_i(\lambda)$ ($i=0,1,2$) in the
same manner as the case $K_0(\lambda)$, so that we may omit
those proof.	 
	 
Then, we consider $L_1(\lambda)$.
If we set $$l_{1,\lambda}(x)=
\mathcal{F}^{-1}_{\xi'}[m(\lambda,\xi')Ae^{-t_1x_N}](x'),$$
the operator $L_1(\lambda)$ is given by the formula:
\begin{align*}
[L_1(\lambda)f](x)=\int_{\mathbb{R}^N_+}
l_{1,\lambda}(x'-y',x_N+y_N)f(y)\,dy,
\end{align*} 
so that to prove that $L_1$ has the required properties
it is sufficient to prove	 
\begin{align}\label{460}
|(\tau\partial_\tau)^sl_{1,\lambda}(x)|\leq
C_{N,\lambda_0,\varepsilon,\mu_{*+},\nu_{*+},\kappa_{*+}}
|x|^{-N}\quad(s=0,1).
\end{align}
According to Saito \cite[Lemma 4.5]{Saito} and Shibata and
Shimizu \cite[Lemma 5.3]{SS2012}, we have	 
\begin{align*}
|\partial_{\xi'}^{\alpha'}\{(\tau\partial_\tau)^s
e^{-t_1 x_N} \}|\leq C (|\lambda|^{1/2}+A)^{-|\alpha'|}
e^{-b(|\lambda|^{1/2}+A)x_N}
\quad (s=0,1).
\end{align*}	 
Then, by the Leibniz rule and the assumption we have
\begin{align}\label{461}
&|\partial_{\xi'}^{\alpha'}\{
(\tau\partial_\tau)^s m(\lambda,\xi')Ae^{-t_1x_N} \}|\\
&\quad\leq
C_{\alpha',\lambda_0,\varepsilon,\mu_{*+},\nu_{*+},\kappa_{*+}}
A^{1-|\alpha'|}e^{-b(|\lambda|^{1/2}+A) x_N},
\nonumber
\end{align}	 
which, combined with Theorem 3.6 in Shibata and Shimizu
\cite{SS2012}, furnishes that
\begin{align*}
|(\tau\partial_\tau)^s l_{1,\lambda}(x)|\leq
C_{N,\lambda_0,\varepsilon,\mu_{*+},\nu_{*+},\kappa_{*+}}
|x'|^{-N}\quad (s=0,1).
\end{align*}
On the other hand, using (\ref{461}) with
$s=0$ and $\alpha'=0$, for $s=0,1$, we have
\begin{align*}
|(\tau\partial_\tau)^s l_{1,\lambda}(x)|
&\leq C_{\lambda_0,\varepsilon,\mu_{*+},\nu_{*+},\kappa_{*+}}
\Big(\frac{1}{2\pi}\Big)^{N-1}\int_{\mathbb{R}^{N-1}}
Ae^{-bAx_N} \,d\xi'\\
&\leq |x_N|^{-N} 
C_{\lambda_0,\varepsilon,\mu_{*+},\nu_{*+},\kappa_{*+}}
\Big(\frac{1}{2\pi}\Big)^{N-1}\int_{\mathbb{R}^{N-1}}
|\eta'|e^{-b|\eta'|} \,d\eta',
\end{align*}
which, combined with (\ref{461}), implies (\ref{460}).
Hence, this completes the case $L_1(\lambda)$.	 
Furthermore, we can prove the case $L_2(\lambda)$ 
in the same manner as the case $L_1(\lambda)$, 
so that we omit its proof.
\end{proof}	 
	 
\begin{corr}\label{cor.operator}
Let $1<q<\infty$ and $\varepsilon_*<\varepsilon<\pi/2$.
Assume that $\rho_{*+}\neq\rho_{*-}$, $\eta_*\neq 0$,
and $\kappa_{*+}\neq \mu_{*+}\nu_{*+}$.
Set $N_1=N+1$ and $N_2=N^2+N+1$.

(1) Let $r=1,2$. For 
$k_{r,1}(\lambda,\xi')\in \mathbb{M}_{r-3,2,\varepsilon,0}$
and $l_r(\lambda,\xi')\in \mathbb{M}_{r-4,2,\varepsilon,0}$,
we define operators 
$K_0^r(\lambda)$, $L_{j}^r(\lambda)$ ($j=1,2$) by	
\begin{align*}
[K_0^r(\lambda)f](x)
&=\mathcal{F}^{-1}_{\xi'}\Big[k_{r,1}(\lambda,\xi')A
M_{0+}(x_N)\widehat{f}(\xi',0) \Big](x'),\\
[L_j^r(\lambda)f](x)
&=\mathcal{F}^{-1}_{\xi'}\Big[l_r(\lambda,\xi')A
e^{-t_j x_N}\widehat{f}(\xi',0) \Big](x')
\quad(j=1,2),
\end{align*}
for $\lambda\in \Sigma_{\varepsilon,\lambda_0}$, and
$f\in W^r_q(\mathbb{R}^N_+)$.
Then, there exist operators $\widetilde{K}^r_0(\lambda)$,
$\widetilde{L}_j^r(\lambda)$, with
\begin{align*}
\widetilde{K}^r_0(\lambda), \widetilde{L}_j^r(\lambda)
\in {\rm Anal}\,(\Sigma_{\varepsilon,\lambda_0},
\mathcal{L}(L_q(\mathbb{R}_+)^{N_r},W^3_q(\mathbb{R}^N_+))),
\end{align*}	
such that for any $g\in W^1_q(\mathbb{R}_+^N)$ and any
$h\in W^2_q(\mathbb{R}_+^N)$
\begin{alignat*}2
K^1_0(\lambda) g
&=\widetilde{K}^1_0(\lambda)(\lambda^{1/2} g,\nabla g),
&\quad
K^2_0(\lambda) h
&=\widetilde{K}^2_0(\lambda)(\lambda h,\lambda^{1/2}\nabla h,
\nabla^2 h),\\
L^1_j(\lambda) g
&=\widetilde{L}^1_j(\lambda)(\lambda^{1/2} g,\nabla g),
&\quad
L^2_j(\lambda) h
&=\widetilde{L}^2_j(\lambda)(\lambda h,\lambda^{1/2}\nabla h,
\nabla^2 h)
\quad(j=1,2).
\end{alignat*}	
Furthermore, for $s=0,1$, $j=1,2$
\begin{align*}
\mathcal{R}_{\mathcal{L}(L_q(\mathbb{R}^N_+)^{N_j},
L_q(\mathbb{R}^N_+)^{N^3+N^2}\times W^1_q(\mathbb{R}^N_+))}
(\{(\tau\partial_\tau)^s 
(G_\lambda^2 \widetilde{K}^r_0(\lambda))
\mid \lambda\in\Sigma_{\varepsilon,\lambda_0} \})&\leq C,\\
\mathcal{R}_{\mathcal{L}(L_q(\mathbb{R}^N_+)^{N_j},
L_q(\mathbb{R}^N_+)^{N^3+N^2}\times W^1_q(\mathbb{R}^N_+))}
(\{(\tau\partial_\tau)^s 
(G_\lambda^2 \widetilde{L}^r_j(\lambda))
\mid \lambda\in\Sigma_{\varepsilon,\lambda_0} \})&\leq C,
\end{align*}
with a positive constant 
$C=C_{N,q,\varepsilon,\lambda_0,
\mu_{*+},\nu_{*+},\kappa_{*+}}$.		 
	
(2) Let $i,r=1,2$. For
$k_{r,2}(\lambda,\xi')\in \mathbb{M}_{r-2,2,\varepsilon,0}$,
we define operators $K_i^r(\lambda)$ by
\begin{align*}
[K_i^r(\lambda)f](x)
&=\mathcal{F}^{-1}_{\xi'}\Big[k_{r,2}(\lambda,\xi')A
M_{i+}(x_N)\widehat{f}(\xi',0) \Big](x'),
\end{align*}
for $\lambda\in \Sigma_{\varepsilon,0}$, and
$f\in W^r_q(\mathbb{R}^N_+)$.
Then, there exist operators $\widetilde{K}^r_0(\lambda)$,
$\widetilde{L}_j^r(\lambda)$, with	
\begin{align*}
\widetilde{K}^r_i(\lambda)
\in {\rm Anal}\,(\Sigma_{\varepsilon,0},
\mathcal{L}(L_q(\mathbb{R}_+)^{N_r},W^2_q(\mathbb{R}^N_+))),
\end{align*}
such that for any $g\in W^1_q(\mathbb{R}_+^N)$ and any
$h\in W^2_q(\mathbb{R}_+^N)$
\begin{alignat*}2
K^1_i(\lambda) g
&=\widetilde{K}^1_i(\lambda)(\lambda^{1/2} g,\nabla g),
&\quad
K^2_i(\lambda) h
&=\widetilde{K}^2_i(\lambda)(\lambda h,\lambda^{1/2}\nabla h,
\nabla^2 h).
\end{alignat*}	
Furthermore, for $s=0,1$, 
\begin{align*}
\mathcal{R}_{\mathcal{L}(L_q(\mathbb{R}^N_+)^{N_j},
L_q(\mathbb{R}^N_+)^{N^3+N^2+N})}
(\{(\tau\partial_\tau)^s 
(G_\lambda^1 \widetilde{K}^r_0(\lambda))
\mid \lambda\in\Sigma_{\varepsilon,\lambda_0} \})&\leq C,
\end{align*}
with a positive constant 
$C=C_{N,q,\varepsilon,\mu_{*+},\nu_{*+},\kappa_{*+}}$.
\end{corr}
\begin{proof}
We only prove the case $K_0^1(\lambda)$.
The proof of other cases are similar to $K_0^1(\lambda)$,
so that we may omit the detailed proof 
(cf. Saito \cite[Corrolary 4.9]{Saito} and Shibata and Shimizu 
\cite[Lemma 5.6, Lemma 5.7]{SS2012}). 	
	
By the definition of $M_{0+}(x_N)$, we have
\begin{align}\label{M0}
\partial_N M_{0+}(x_N)=-t_2M_{0+}(x_N)-e^{-t_1x_N},
\end{align}
which, combined with (\ref{identity}), implies that
\begin{align*}
&K_0^1(\lambda) g\\
&=\int_{0}^{\infty}\mathcal{F}^{-1}_{\xi'}\Big[ 
\frac{\rho_{*+}\lambda^{1/2}t_2k_{1,1}(\lambda,\xi')}
{\mu_{*+}B_+^2}AM_{0+}(x_N+y_N)
\widehat{\lambda^{1/2}g}(\lambda,\xi')
\Big](x')\, dy_N\\
&+\sum_{m=1}^{N-1}\int_{0}^{\infty}\mathcal{F}^{-1}_{\xi'}\Big[ 
\frac{i\xi_mt_2k_{r,1}(\lambda,\xi')}{B_+^2}AM_{0+}(x_N+y_N)
\widehat{\partial_m g}(\xi',y_N)
\Big](x')\, dy_N\\
&+\int_{0}^{\infty}\mathcal{F}^{-1}_{\xi'}\Big[ 
\frac{\rho_{*+}\lambda^{1/2}k_{1,1}(\lambda,\xi')}
{\mu_{*+}B_+^2}Ae^{-t_1(x_N+y_N)}\widehat{\lambda^{1/2}g}
(\xi',y_N)\Big](x')\, dy_N\\
&-\sum_{m=1}^{N-1}\int_{0}^{\infty}\mathcal{F}^{-1}_{\xi'}\Big[ 
\frac{i\xi_mk_{1,1}(\lambda,\xi')}{B_+^2}Ae^{-t_1(x_N+y_N)}
\widehat{\partial_mg}(\xi',y_N)\Big](x')\, dy_N\\
&-\int_{0}^{\infty}\mathcal{F}^{-1}_{\xi'}\Big[ 
k_{1,1}(\lambda,\xi')AM_{0+}(x_N+y_N)\widehat{\partial_N g}
(\xi',y_N)\Big](x')\, dy_N\\
&=
\int_{0}^{\infty}\mathcal{F}^{-1}_{\xi'}\Big[ 
\frac{\rho_{*+}^2\lambda t_2k_{1,1}(\lambda,\xi')}
{\mu_{*+}^2 B_+^4}\lambda^{1/2}AM_{0+}(x_N+y_N)
\widehat{\lambda^{1/2}g}(\lambda,\xi')
\Big](x')\, dy_N\\
&+\int_{0}^{\infty}\mathcal{F}^{-1}_{\xi'}\Big[ 
\frac{\rho_{*+}\lambda^{1/2}t_2Ak_{1,1}(\lambda,\xi')}
{\mu_{*+}B_+^4}A^2M_{0+}(x_N+y_N)
\widehat{\lambda^{1/2}g}(\lambda,\xi')
\Big](x')\, dy_N\\
&+\sum_{m=1}^{N-1}\int_{0}^{\infty}\mathcal{F}^{-1}_{\xi'}\Big[ 
\frac{\rho_{*+} i\xi_m \lambda^{1/2}
t_2 k_{r,1}(\lambda,\xi')}{\mu_{*+} B_+^4}
\lambda^{1/2}AM_{0+}(x_N+y_N)\widehat{\partial_m g}(\xi',y_N)
\Big](x')\, dy_N\\ 
&+\sum_{m=1}^{N-1}\int_{0}^{\infty}\mathcal{F}^{-1}_{\xi'}\Big[ 
\frac{i\xi_m t_2A k_{r,1}(\lambda,\xi')}{B_+^4}
A^2M_{0+}(x_N+y_N)\widehat{\partial_m g}(\xi',y_N)
\Big](x')\, dy_N\\ 
&+\int_{0}^{\infty}\mathcal{F}^{-1}_{\xi'}\Big[ 
\frac{\rho_{*+}\lambda^{1/2}k_{1,1}(\lambda,\xi')}
{\mu_{*+}B_+^2}Ae^{-t_1(x_N+y_N)}\widehat{\lambda^{1/2}g}
(\xi',y_N)\Big](x')\, dy_N\\
&-\sum_{m=1}^{N-1}\int_{0}^{\infty}\mathcal{F}^{-1}_{\xi'}\Big[ 
\frac{i\xi_mk_{1,1}(\lambda,\xi')}{B_+^2}Ae^{-t_1(x_N+y_N)}
\widehat{\partial_mg}(\xi',y_N)\Big](x')\, dy_N\\
&-\int_{0}^{\infty}\mathcal{F}^{-1}_{\xi'}\Big[ 
\frac{\rho_{*+}\lambda^{1/2} k_{1,1}(\lambda,\xi')}
{\mu_{*+}B_+^2}
\lambda^{1/2}AM_{0+}(x_N+y_N)\widehat{\partial_N g}
(\xi',y_N)\Big](x')\, dy_N\\
&-\int_{0}^{\infty}\mathcal{F}^{-1}_{\xi'}\Big[ 
\frac{Ak_{1,1}(\lambda,\xi')}{B_+^2}
A^2M_{0+}(x_N+y_N)\widehat{\partial_N g}
(\xi',y_N)\Big](x')\, dy_N\\
&=:\sum_{l=1}^{8}\widetilde{K}^1_{0,l}(\lambda)
(\lambda^{1/2}g,\nabla g).
\end{align*}	
Here, we use Volevich's formula \cite{Volevich}:
\begin{align*}
&f(\xi',x_N)\widehat{g}(0)\\
&\quad=-\int_{0}^{\infty}\{(\partial_N f)(\xi',x_N+y_N)
\widehat{g}(y_N)+
f(\xi',x_N+y_N)\widehat{\partial_N g}(\xi',y_N) \}
\,dy_N
\end{align*}	
with $f(\xi',x_N+y_N)\widehat{g}(y_N)\to 0$ as $y_N\to \infty$.
	
First, we estimate $\widetilde{K}^1_{0,n_0}(\lambda)$
($n_0=1,2,3,4,7,8$).
Here, we consider $\widetilde{K}^1_{0,8}(\lambda)$
only, because we can treat other terms similarly.
Let $j,k,l=1,\dots,N-1$.
By (\ref{M0}), $\widetilde{K}^1_{0,8}(\lambda)$ can 
be written as	
\begin{align*}
&\lambda\widetilde{K}^1_{0,8}(\lambda)(\lambda^{1/2}g,\nabla g)\\
&=-\int_{0}^{\infty}\mathcal{F}^{-1}_{\xi'}\Big[
\frac{\lambda Ak_{1,1}(\lambda,\xi')}{B_+^2} A^2M_{0+}(x_N+y_N)
\widehat{\partial_N g}(\xi',y_N)
\Big](x')\,dy_N,\\
&(\lambda\partial_j,\lambda^{1/2}\partial_j\partial_k,
\partial_j\partial_k\partial_l)
\widetilde{K}^1_{0,8}(\lambda)(\lambda^{1/2}g,\nabla g)\\
&=-\int_{0}^{\infty}\mathcal{F}^{-1}_{\xi'}\Big[
\frac{(\lambda i\xi_j,-\lambda^{1/2}
\xi_j\xi_k,-i\xi_j\xi_k\xi_l)Ak_{1,1}(\lambda,\xi')}{B_+^2}\\
&\qquad\times A^2M_{0+}(x_N+y_N)\widehat{\partial_N g}
(\xi',y_N)\Big](x')\,dy_N,\\
&(\lambda,\lambda^{1/2}\partial_j,\partial_j\partial_k)
\partial_N\widetilde{K}^1_{0,8}
(\lambda)(\lambda^{1/2}g,\nabla g)\\
&=\int_{0}^{\infty}\mathcal{F}^{-1}_{\xi'}\Big[
\frac{(\lambda,\lambda^{1/2}i\xi_j,-\xi_j\xi_k)
t_2Ak_{1,1}(\lambda,\xi')}{B_+^2}
A^2 M_{0+}(x_N+y_N)\widehat{\partial_N g}(\xi',y_N)
\Big](x')\,dy_N\\
&+\int_{0}^{\infty}\mathcal{F}^{-1}_{\xi'}\Big[
\frac{(\lambda,\lambda^{1/2}i\xi_j,-\xi_j\xi_k)
A^2k_{1,1}(\lambda,\xi')}{B_+^2}
Ae^{-t_1(x_N+y_N)}\widehat{\partial_N g}(\xi',y_N)
\Big](x')\,dy_N,\\
&(\lambda^{1/2},\partial_j)\partial_N^2
\widetilde{K}^1_{0,8}(\lambda)(\lambda^{1/2}g,\nabla g)\\
&=-\int_{0}^{\infty}\mathcal{F}^{-1}_{\xi'}\Big[
\frac{(\lambda^{1/2},i\xi_j)t_2^2Ak_{1,1}(\lambda,\xi')}{B_+^2}
A^2M_{0+}(x_N+y_N)\widehat{\partial_N g}(\xi',y_N)
\Big](x')\,dy_N\\
&-\int_{0}^{\infty}\mathcal{F}^{-1}_{\xi'}\Big[
\frac{(\lambda^{1/2},i\xi_j)(t_1+t_2)A^2k_{1,1}(\lambda,\xi')}
{B_+^2} Ae^{-t_1(x_N+y_N)}\widehat{\partial_N g}(\xi',y_N)
\Big](x')\,dy_N,\\
&\partial^3_N\widetilde{K}^1_{0,8}
(\lambda)(\lambda^{1/2}g,\nabla g)\\
&=\int_{0}^{\infty}\mathcal{F}^{-1}_{\xi'}\Big[
\frac{t^3_2 Ak_{1,1}(\lambda,\xi')}{B_+^2}A^2M_{0+}(x_N+y_N)
\widehat{\partial_N g}(\xi',y_N)\Big](x')\,dy_N\\
&+\int_{0}^{\infty}\mathcal{F}^{-1}_{\xi'}\Big[
\frac{(t_1^2+t_1t_2+t_2^2)A^2k_{1,1}(\lambda,\xi')}{B_+^2}
Ae^{-t_1(x_N+y_N)}
\widehat{\partial_N g}(\xi',y_N)\Big](x')\,dy_N.
\end{align*}	
Then, by Lemma \ref{lem:multiplier}, the assumption for
$k_{1,1}(\lambda,\xi')$, and (\ref{505}) below,
we have	
\begin{align*}
\frac{\lambda Ak_{1,1}(\lambda,\xi')}{B_+^2}
\in \mathbb{M}_{-1,2,\varepsilon,0}
&\subset \mathbb{M}_{0,2,\varepsilon,\lambda_0},\\
\frac{(\lambda i\xi_j,-\lambda^{1/2}
\xi_j\xi_k,-i\xi_j\xi_k\xi_l)Ak_{1,1}(\lambda,\xi')}{B_+^2}
&\in \mathbb{M}_{0,2,\varepsilon,0},\\
\frac{(\lambda,\lambda^{1/2}i\xi_j,-\xi_j\xi_k)
t_2Ak_{1,1}(\lambda,\xi')}{B_+^2}
&\in \mathbb{M}_{0,2,\varepsilon,0},\\
\frac{(\lambda,\lambda^{1/2}i\xi_j,-\xi_j\xi_k)
A^2k_{1,1}(\lambda,\xi')}{B_+^2}
&\in \mathbb{M}_{0,2,\varepsilon,0},\\
\frac{(\lambda^{1/2},i\xi_j)t_2^2Ak_{1,1}(\lambda,\xi')}{B_+^2}
&\in \mathbb{M}_{0,2,\varepsilon,0},\\
\frac{(\lambda^{1/2},i\xi_j)(t_1+t_2)A^2k_{1,1}(\lambda,\xi')}
{B_+^2}
&\in \mathbb{M}_{0,2,\varepsilon,0},\\
\frac{t^3_2 Ak_{1,1}(\lambda,\xi')}{B_+^2}
&\in \mathbb{M}_{0,2,\varepsilon,0},\\
\frac{(t_1^2+t_1t_2+t_2^2)A^2k_{1,1}(\lambda,\xi')}{B_+^2}
&\in \mathbb{M}_{0,2,\varepsilon,0}.
\end{align*}	
Combining these properties with Lemma \ref{lem.4.4} 
furnishes for $s=0,1$
\begin{align}\label{462}
\mathcal{R}_{\mathcal{L}(L_q(\mathbb{R}^N_+)^{N_1},
L_q(\mathbb{R}^N_+)^{N^3+N^2}\times W^1_q(\mathbb{R}^N_+))}
(\{(\tau\partial_\tau)^s 
(G_\lambda^2 \widetilde{K}^1_{0,8}(\lambda))
\mid \lambda\in\Sigma_{\varepsilon,\lambda_0} \})&\leq C,
\end{align}
with some positive constant 
$C=C_{N,q,\varepsilon,\lambda_0,\mu_{*+},\nu_{*+},\kappa_{*+}}$.
Analogously, we have
\begin{align}\label{463}
\mathcal{R}_{\mathcal{L}(L_q(\mathbb{R}^N_+)^{N_1},
L_q(\mathbb{R}^N_+)^{N^3+N^2}\times W^1_q(\mathbb{R}^N_+))}
(\{(\tau\partial_\tau)^s 
(G_\lambda^2 \widetilde{K}^1_{0,n_1}(\lambda))
\mid \lambda\in\Sigma_{\varepsilon,\lambda_0} \})&\leq C
\end{align}
for $n_1=1,2,3,4,7$.

Next, we estimate $\widetilde{K}^1_{0,n_2}(\lambda)$ ($n_2=5,6$).
By Lemma \ref{lem:multiplier}, the assumption for
$k_{1,1}(\lambda,\xi')$, and (\ref{505}) below, we have for
$m=1,\dots,N-1$
\begin{align*}
\frac{\rho_{*+}\lambda^{1/2} k_{1,1}(\lambda,\xi')}
{\mu_{*+}B_+^2},\enskip
\frac{i\xi_mk_{1,1}(\lambda,\xi')}{B_+^2}
\in \mathbb{M}_{-3,2,\varepsilon,0}.
\end{align*}
Accordingly, if we employ the same argument as 
in proving (\ref{462}), we obtain	
\begin{align*}
\mathcal{R}_{\mathcal{L}(L_q(\mathbb{R}^N_+)^{N_1},
L_q(\mathbb{R}^N_+)^{N^3+N^2}\times W^1_q(\mathbb{R}^N_+))}
(\{(\tau\partial_\tau)^s 
(G_\lambda^2 \widetilde{K}^1_{0,n_2}(\lambda))
\mid \lambda\in\Sigma_{\varepsilon,\lambda_0} \})&\leq C
\end{align*}
for $s=0,1$ and $n_2=3,4$ with 
$C=C_{N,q,\varepsilon,\lambda_0,\mu_{*+},\nu_{*+},\kappa_{*+}}$,
which, combined with (\ref{462}) and (\ref{463}),
complete the proof. 
\end{proof}	
 
Employing the argument in Shibata
\cite[Sect. 4]{Shibata2016}, 
by (\ref{4.428}) and (\ref{4.429}), there exist
operator families $\widetilde{\mathcal{A}}^{3+}_{Ji}(\lambda)$
$(i=1,2,3,4)$,	 $\widetilde{\mathcal{A}}^{3-}_{Jl}(\lambda)$
$(l=1,2,3)$, and $\widetilde{\mathcal{P}}^3_-(\lambda)$  
such that	 
\begin{align*}
\mathcal{A}^{3+}_{Ji}(\lambda)\mathbf{h}=
\widetilde{\mathcal{A}}^{3+}_{Ji}(\lambda)
\widetilde{\mathbf{H}},
\quad
\mathcal{A}^{3-}_{Jl}(\lambda)\mathbf{h}=
\widetilde{\mathcal{A}}^{3-}_{Jl}(\lambda)
\widetilde{\mathbf{H}},
\quad
\mathcal{P}^{3}_-(\lambda)\mathbf{h}=
\widetilde{\mathcal{P}}^{3}_-(\lambda)
\widetilde{\mathbf{H}},
\end{align*}
respectively.
Using (\ref{4.428}), (\ref{4.429}), and 
Corollary \ref{cor.operator}, there exists operator family
$\widetilde{\mathcal{B}}^{3}_+(\lambda)$ such that	 
\begin{align*}
\mathcal{B}^3_{+}(\lambda)\mathbf{h}=
\widetilde{\mathcal{B}}^{3}_+(\lambda)
\widetilde{\mathbf{H}}.
\end{align*}	 
Consequently, if we define operators
$\mathcal{A}^2_\pm(\lambda)$, $\mathcal{B}_+^2(\lambda)$, 
and $\mathcal{P}_-^2(\lambda)$ by
\begin{align*}
\mathcal{A}^2_+(\lambda)\widetilde{\mathbf{H}}&=
\sum_{i=1}^{4}(\mathcal{A}^{3+}_{1i}(\lambda)
\widetilde{\mathbf{H}},
\dots,\mathcal{A}^{3+}_{Ni}(\lambda)\widetilde{\mathbf{H}} ),\\
\mathcal{A}^2_-(\lambda)\widetilde{\mathbf{H}}&=
\sum_{l=1}^{3}(\mathcal{A}^{3-}_{1l}(\lambda)
\widetilde{\mathbf{H}},
\dots,\mathcal{A}^{3-}_{Nl}(\lambda)\widetilde{\mathbf{H}} ),\\
\mathcal{B} ^{2}_+(\lambda)
\widetilde{\mathbf{H}}&=
\widetilde{\mathcal{B}}^{3}_-(\lambda)
\widetilde{\mathbf{H}},\\
\mathcal{P} ^{2}_-(\lambda)
\widetilde{\mathbf{H}}&=
\widetilde{\mathcal{P}}^{3}_-(\lambda)
\widetilde{\mathbf{H}},
\end{align*}
respectively, from (\ref{4.428}) we have
\begin{align*}
\mathbf{u}_\pm
&=\mathcal{A}_\pm^2(\lambda)(\lambda \mathbf{h},
\lambda^{1/2}\nabla\mathbf{h},\nabla^2\mathbf{h}),\\
\rho_+
&=\mathcal{B}_+^2(\lambda)(\lambda \mathbf{h},
\lambda^{1/2}\nabla\mathbf{h},\nabla^2\mathbf{h}),\\
\pi_-
&=\mathcal{P}_-^2(\lambda)(\lambda \mathbf{h},
\lambda^{1/2}\nabla\mathbf{h},\nabla^2\mathbf{h}).
\end{align*}	 
Combining Corollary \ref{cor.operator} and the argument 
in Shibata \cite[Sect. 4]{Shibata2016},
we have (\ref{operator.estimate1}), so that we have completed
the proof of Theorem \ref{Th05}.	 
	 



\section{Analysis of multipliers}\label{Sect.multipliers}

In this section, we estimate several multipliers.
To this end, we start with the following wildly known
estimate:
\begin{align}\label{501}
|\alpha \lambda+\beta|
\geq\Big(\sin\frac{\varepsilon}{2}\Big)(\alpha|\lambda|+\beta)
\end{align}	
for any $\lambda\in \Sigma_{\varepsilon}$ and positive numbers
$\alpha$ and $\beta$.	
	
First, we estimate $B_\pm^s$ and
$(\mu_{*+}B_++\mu_{*-}B_-)$.
For this purpose, we use the estimates:
\begin{align}\label{502}
c_1(|\lambda|^{1/2}+A)
\leq{\rm Re}\, B_\pm\leq |B_\pm|\leq c_2(|\lambda|^{1/2}+A)
\end{align}	
for any $(\lambda,\xi)\in \widetilde{\Sigma}_{\varepsilon,0}$
with some positive constant $c_1$ and $c_2$, 
which immediately follows from (\ref{501}).
Here and in the sequel, $c_1$ and $c_2$ denote some positive
constants essentially depending on
$\varepsilon$, $\mu_{*\pm}$, $\nu_{*+}$,
$\kappa_{*+}$, and $\rho_{*\pm}$.
In particular, by (\ref{502}) we have
\begin{align}\label{503}
c_1(|\lambda|^{1/2}+A)
&\leq{\rm Re}\, (\mu_{*+}B_++\mu_{*-}B_-)\\
&\leq |(\mu_{*+}B_++\mu_{*-}B_-)|\leq c_2(|\lambda|^{1/2}+A)
\nonumber
\end{align}	
for any $(\lambda,\xi)\in \widetilde{\Sigma}_{\varepsilon,0}$.
As shown in Enomoto and Shibata \cite[Lemma 4.3]{ES2013},
using (\ref{502}), (\ref{503}) and the Bell's formula:
\begin{align}\label{Bell} 
\partial_{\xi'}^{\alpha'}f(g(\xi'))
=\sum_{l=1}^{|\alpha'|}f^{(l)}(g(\xi'))
\sum_{\substack{\alpha_1+\cdots+\alpha_l=\alpha'\\
|\alpha_i|\geq1}}
\Gamma^{\alpha'}_{\alpha_1,\dots,\alpha_l}
(\partial_{\xi'}^{\alpha_1}g(\xi'))\cdots
(\partial_{\xi'}^{\alpha_l}g(\xi'))
\end{align}
with suitable coefficients 
$\Gamma^{\alpha'}_{\alpha_1,\dots,\alpha_l}$,
where $f^{(l)}(t)=d^lf(t)/dt^l$, we see that
\begin{align}\label{504}
(M_1)^s\in \mathbb{M}_{s,1,\varepsilon,0}\quad
(M_1=B_\pm, \mu_{*+}B_++\mu_{*-}B_-).
\end{align}	
	
Second, we estimate 
$(t_i)^s,\enskip t_i+B_+$, and $t_iB_++A^2 \enskip(i=1,2)$.
As seen in Saito \cite{Saito},
we have the following lemma.	
\begin{lemm}\label{lem:t}
Let $i=1,2$.
Then, the roots $s_i$ of (\ref{eq:P}) are given by	
\begin{align*}
s_1=
\begin{cases}
\cfrac{\mu_{*+}+\nu_{*+}}{2\kappa_{*+}}
+\sqrt{\eta_*}
& (\eta_*>0),\\
\cfrac{\mu_{*+}+\nu_{*+}}{2\kappa_{*+}}
+i\sqrt{|\eta_*|}
& (\eta_*<0),
\end{cases}\qquad
s_2=
\begin{cases}
\cfrac{\mu_{*+}+\nu_{*+}}{2\kappa_{*+}}
-\sqrt{\eta_*}
& (\eta_*>0),\\
\cfrac{\mu_{*+}+\nu_{*+}}{2\kappa_{*+}}
-i\sqrt{|\eta_*|}
& (\eta_*<0)
\end{cases}
\end{align*}
with
\begin{align*}
\eta_*
=\Big(\frac{\mu_{*+}+\nu_{*+}}{2\kappa_{*+}}\Big)^2
-\frac{1}{\kappa_{*+}}.
\end{align*}
In addition, there exist positive constants $c_1$ and $c_2$ 
such that
\begin{align*}
c_1(|\lambda|^{1/2}+A)
\leq{\rm Re}\, t_i\leq |t_i|\leq c_2(|\lambda|^{1/2}+A)
\quad(i=1,2)
\end{align*}
for any 
$(\lambda,\xi')\in \widetilde{\Sigma}_{\varepsilon,0}$.	
\end{lemm}	
\begin{rema}\label{eta.condition}
We have in general the following situations concerning
roots with positive real parts for the characteristic
equation of (\ref{413}):

(1) When $\eta_*<0$, it holds that $B_+\neq t_1$,
$B_+\neq t_2$, and $t_1\neq t_2$.

(2) When $\eta_*=0$, there are two cases:
$B_+\neq t_1$ and $t_1=t_2$; $B_+=t_1=t_2$.

(3) When $\eta_*>0$, there are three cases:
$B_+\neq t_1$, $B_+\neq t_2$, and $t_1\neq t_2$;
$B_+=t_1$ and $t_1\neq t_2$;
$B_+=t_2$ and $t_1\neq t_2$.

We assume $\eta_*\neq0$ and $\kappa_{*+}\neq\nu_{*+}\mu_{*+}$.
Under these assumptions, we have the three roots 
with positive real parts different from each other.
We consider, however, that our technique in this paper can 
be applied to the case of equal roots.	
\end{rema}	
From the Bell's formula and Lemma \ref{lem:t}, for $i=1,2$,
we have
\begin{align}
\label{505}	
(M_2)^s&\in\mathbb{M}_{s,1,\varepsilon,0}
\qquad(M_2=t_i,\enskip t_i+B_+), \quad
\mathfrak{n}(\lambda,\xi')\in
\mathbb{M}_{4,1,\varepsilon,0}.
\end{align}	
Then, by (\ref{442}), (\ref{det:L}), (\ref{446}), (\ref{448}),
and Lemma \ref{lem:t},
we have
\begin{align*}
Q_{Nm}^+
=&-\frac{2i\xi_mt_1t_2G\rho_{*+}(t_2^2-A^2)
(t_1^2-B_+^2)(t_1+t_2)}
{A\rho_{*+}\lambda(t_1^2-t_2^2)\mathfrak{l}(\lambda,\xi')
(t_1+B_+)}\\
=&-\frac{2(s_1-\rho_{*+}\mu_{*+}^{-1})s_2i\xi_mt_1t_2
(t_1+t_2)G}
{(s_1-s_2)A(t_1+B_+)\mathfrak{l}(\lambda,\xi')}
\end{align*}
with $s_1-\rho_{*+}\mu_{*+}^{-1}\neq 0$.
Hence, by Remark \ref{rema.multi}, (\ref{504}), 
Lemma \ref{lem:t}, (\ref{505}), and Lemma \ref{lem601} bellow,
we have $Q_{Nm}^+\in \mathbb{M}_{0,2,\varepsilon,0}$.
Analogously, we have 
$Q_{jm}^+\in \mathbb{M}_{0,2,\varepsilon,0}$
($j=1,\dots,N-1$), which yields
$Q_{Jm}^+\in \mathbb{M}_{0,2,\varepsilon,0}$
($J=1,\dots,N$).
Furthermore, employing the same argument as $Q_{Jm}^+$,
we have other assertions in (\ref{4.429}).
Then, we finish the estimate of multipliers.	




\section{Analysis of Lopatinski determinant}
\label{Sect.Lopatinski}
\begin{lemm}\label{lem601}
Let $\varepsilon_*<\varepsilon<\pi/2$ and 
$\mathfrak{l}(\lambda,\xi')$ be defined in (\ref{det:L}).
Assume that $\rho_{*+}\neq\rho_{*-}$, $\eta_*\neq 0$,
and $\kappa_{*+}\neq \mu_{*+}\nu_{*+}$.
Then, there exists a positive constant $C$ such that
\begin{align}\label{602}
|\mathfrak{l}(\lambda,\xi')|\geq C(|\lambda|^{1/2}+A)^6
\end{align}	
for any 
$(\lambda,\xi')\in \widetilde{\Sigma}_{\varepsilon,0}$.
Here, positive constant $C_{\alpha'}$ is depending on
$\alpha'$, $\varepsilon$, $\mu_{*\pm}$, $\nu_{*+}$,
$\kappa_{*+}$, and $\rho_{*\pm}$.

In addition, we have
\begin{align}\label{603}
|\partial_{\xi'}^{\alpha'}
\{(\tau\partial_\tau)^s(\mathfrak{l}(\lambda,\xi'))^{-1} \}|
\leq C(|\lambda|^{1/2}+A)^{-6}A^{-|\alpha'|}\qquad(s=0,1)
\end{align}
for any multi-index $\alpha'\in \mathbb{N}^{N-1}_0$ and 
$(\lambda,\xi')\in \widetilde{\Sigma}_{\varepsilon,0}$,
that is, $\mathfrak{l}(\lambda,\xi)
\in \mathbb{M}_{-6,2,\varepsilon,0}$.			
\end{lemm}
\begin{proof}
We can prove (\ref{603}) by using (\ref{602}) with the
Leibniz rule and the Bell's formula (\ref{Bell})
with $f(t)=t^{-1}$ and $g(\xi')=\mathfrak{l}(\lambda,\xi')$.

In order to prove (\ref{602}), we consider the three cases:
(1) $R_1|\lambda|^{1/2}\leq A$, 
(2) $R_2 A\leq |\lambda|^{1/2}$, 
(3) $R_2^{-1}|\lambda|^{1/2}\leq A\leq|\lambda|^{1/2}$
for large $R_1\geq 1$ and $R_2\geq 1$.

First, we consider the case: 
$R_1|\lambda|^{1/2}\leq A$ with large $R_1\geq1$.
In this case, we set $\delta_1=\lambda^{1/2}/A$ and see that
\begin{align*}
B_+=A(1+O(\delta_1)),\quad B_-=A(1+O(\delta_1)),\quad 
t_i=A(1+O(\delta_1))\enskip (i=1,2),
\end{align*}	
which imply that
\begin{align*}
\mathfrak{m}_i(\lambda,\xi')=&
8\rho_{*+}\mu_{*+}^{-1}
A^6(1+O(\delta_1))\quad(i=1,2),\\
\mathfrak{n}(\lambda,\xi')=&
4\rho_{*+}\mu_{*+}^{-1}A^4(1+O(\delta_1)).
\end{align*}
Here, by Lemma \ref{lem:t} 
we have $s_i-\rho_{*+}\mu_{*+}^{-1}\neq 0$.	
Then, we obtain
\begin{align*}
\mathfrak{l}(\lambda,\xi')
&=(24\rho_{*+}+8\rho_{*+}\mu_{*+}\mu_{*-})
A^6(1+O(\delta_1)),=:\omega_1 A^6(1+O(\delta_1)).
\end{align*}	
Summing up, there exists a positive constant 
$C_1:=\omega_1/2$ such that 
\begin{align}\label{608}
|\mathfrak{l}(\lambda,\xi')|\geq C_1(|\lambda|^{1/2}+A)^6.
\end{align}

Second, we consider the case $R_2 A\leq|\lambda|^{1/2}$ 
for large $R_2$.
In this case, we set $\delta_2=A/\lambda^{1/2}$ and see that
\begin{align*}
B_+&=\sqrt{\rho_{*+}\mu_{*+}^{-1}\lambda}(1+O(\delta_2)),\quad 
B_-=\sqrt{\rho_{*-}\mu_{*-}^{-1}\lambda}(1+O(\delta_2)),\\
t_i&=\sqrt{s_i\lambda}(1+O(\delta_2))\enskip (i=1,2),
\end{align*}		
which imply that
\begin{align*}
\mathfrak{m}_i
&=\rho_{*+}^2\mu_{*+}^{-2}\sqrt{s_i}
(\sqrt{s_i}+\rho_{*+}^{1/2}\mu_{*+}^{-1/2})
(s_1+\sqrt{s_1s_2}+s_2)
\lambda^{3}(1+O(\delta_2))\quad(i=1,2),\\
\mathfrak{n}
&=\rho_{*+}^2\mu_{*+}^{-2}
\Big(\sqrt{s_1}+\frac{\sqrt{s_2}}{\sqrt{s_1}}
(\sqrt{s_1}+\sqrt{s_2})\Big)
\lambda^{2}(1+O(\delta_2))\nonumber\\
&=\rho_{*+}^2\mu_{*+}^{-2}
\Big(\sqrt{s_2}+\frac{\sqrt{s_1}}{\sqrt{s_2}}
(\sqrt{s_1}+\sqrt{s_2})\Big)
\lambda^{2}(1+O(\delta_2)).\nonumber
\end{align*}	
Then, we obtain
\begin{align*}
\mathfrak{l}(\lambda,\xi')
&=\frac{(s_1+\sqrt{s_1s_2}+s_2)}{\sqrt{s_1}}
\Big(\rho_{*+}\mu_{*+}\mu_{*-}^{-1}
+\rho_{*+}^5\mu_{*+}^{-3} \mu_{*-}
\Big)\lambda^{3}(1+O(\delta_2))\\
&=\frac{(s_1+\sqrt{s_1s_2}+s_2)}{\sqrt{s_2}}
\Big(\rho_{*+}\mu_{*+}\mu_{*-}^{-1}
+\rho_{*+}^5\mu_{*+}^{-3} \mu_{*-}
\Big)\lambda^{3}(1+O(\delta_2))\\
&=:\omega_2\lambda^{3}(1+O(\delta_2)).
\end{align*}
From Lemma \ref{lem:t}, we obtain $\omega_2\neq0$.	
Summing up, there exists a positive constant $C_2:=|\omega_2|/2$ 
such that
\begin{align}\label{609}
|\mathfrak{l}(\lambda,\xi)|\geq C_2(|\lambda|^{1/2}+A)^6.
\end{align}	

Third, we consider the case 
$R_2^{-1}|\lambda|^{1/2}\leq A\leq R_1|\lambda|^{1/2}$.
Here, we set
\begin{align*}
\widetilde{\xi'}
=\frac{\xi'}{|\lambda|^{1/2}+A},\quad
&\widetilde{A}
=\frac{A}{|\lambda|^{1/2}+A},\quad
\widetilde{\lambda}
=\frac{\lambda}{(|\lambda|^{1/2}+A)^2},\\
\widetilde{B}_\pm
=\sqrt{\widetilde{A^2}+\rho_{*\pm}(\mu_{*\pm})^{-1}
\widetilde{\lambda}},\quad
&\widetilde{t_i}
=\sqrt{\widetilde{A^2}+s_i\widetilde{\lambda}},\nonumber \\
D_\varepsilon(R_1,R_2)
=\{(\widetilde{\lambda},\widetilde{A}) \mid 
&(1+R_1)^{-2}\leq |\widetilde{\lambda}|
\leq R_2^2(1+R_2)^2,\nonumber\\
&(1+R_2)^{-1}\leq \widetilde{A}\leq R_1(1+R_1)^{-1} ,\enskip
\widetilde{\lambda}\in\Sigma_\varepsilon\}. \nonumber
\end{align*}
If $(\lambda,\xi')$ satisfies the condition: 
$R_2^{-1}|\lambda|^{1/2}\leq A\leq R_1|\lambda|^{1/2}$ and 
$\lambda\in\Sigma_\varepsilon$, then 
$(\widetilde{\lambda},\widetilde{A})\in D_\varepsilon(R_1,R_2)$.
We also define 
$\mathfrak{l}(\widetilde{\lambda},\widetilde{\xi'})$ 
by replacing $A,B_\pm$ and $t_i\ (i=1,2)$ by 
$\widetilde{A},\widetilde{B}_\pm $ and $\widetilde{t_i}$, 
respectively.
Then, we have	
\begin{align}\label{604}
\mathfrak{l}(\lambda,\xi')
=(|\lambda|^{1/2}+A)^6
\mathfrak{l}(\widetilde{\lambda},\widetilde{\xi'}).
\end{align}
We prove that 
$\mathfrak{l}(\widetilde{\lambda},\widetilde{\xi'})\neq 0$ 
provided that 
$(\widetilde{\lambda},\widetilde{A})\in D_\varepsilon(R_1,R_2)$ 
by contradiction.
Suppose that 
$\mathfrak{l}(\widetilde{\lambda},\widetilde{\xi'})=0$, 
namely, $\det L=0$. 
In this case, in view of (\ref{443}) we assume that there exist
$\mathbf{u}_\pm(x_N)=(u_{1\pm}(x_N),\dots,u_{N\pm}(x_N))\neq0$, 
$\rho_+(x_N)\neq0$, and $\pi_-(x_N)\neq0$
satisfying (\ref{402}) - (\ref{412}) with 
$\widehat{d}(0)=0$, $\widehat{H}(0)=0$,
$\widehat{h}_m(0)=0$, and $\rho_+\neq0$,
that is, $\mathbf{u}_\pm(x_N)\neq0$, $\rho_+(x_N)\neq0$, and 
$\pi_-(x_N)\neq0$ satisfy the following homogeneous equations: 
for $x_N\neq0$ and $w_+=\sum_{j=1}^{N-1}i\xi_j u_{j+}(x_N)
+\partial_N u_{N+}(x_N)$ ($x_N>0$), 
\begin{alignat}2\label{605}
\lambda\rho_++\rho_{*+}w_+&=0 &\text{ for $x_N>0$},\\
\rho_{*+}\lambda u_{j+}
-\mu_{*+}\sum_{k=1}^{N-1}i\xi_k(i\xi_ku_{j+}+i\xi_j u_{k+})
-\mu_{*+}\partial_N(\partial_N u_{j+}+i\xi_j u_{N+})
\nonumber\\
-(\nu_{*+}-\mu_{*+})i\xi_j w_+
-i\xi_j\rho_{*+}\kappa_{*+}(\partial_N^2-|\xi'|^2) \rho_+&=0
&\text{ for $x_N>0$},\nonumber\\
\rho_{*+}\lambda u_{N+} -\mu_{*+}(\partial_N^2-|\xi'|^2)u_{N+} 
-\mu_{*+}\sum_{k=1}^{N-1}i\xi_k(i\xi_k u_{N+} 
+\partial_N u_{k+ })&\nonumber\\
-2\mu_{*+}\partial^2_N u_{N+} 
-(\nu_{*+}-\mu_{*+})\partial_N  w_+
-\rho_{*+}\kappa_{*+}\partial_N(\partial_N^2-|\xi'|^2) \rho_+ &=0
&\text{ for $x_N>0$},\nonumber\\
\sum_{j=1}^{N-1}i\xi_j u_{j-} +\partial_N u_{N-} &=0
&\text{ for $x_N<0$},\nonumber\\
\rho_{*-}\lambda u_{j- }-\mu_{*-}
\sum_{k=1}^{N-1}i\xi_k(i\xi_k u_{j- }+i\xi_j u_{k- })
-\mu_{*-}\partial_N(\partial_N u_{j-} +i\xi_j u_{N- })
-i\xi_j \pi_- &=0
&\text{ for $x_N<0$},\nonumber\\
\rho_{*-}\lambda u_{N-} -\mu_{*-}
\sum_{k=1}^{N-1}i\xi_k(i\xi_k u_{N-} 
+\partial_N u_{k-} )
-2\mu_{*-}\partial^2_N u_{N-}
-\partial_N \pi_- &=0&{\rm for}\ x_N<0,\nonumber\\
\mu_{*-}(\partial_Nu_{m-}+i\xi_mu_{N-})|_-
-\mu_{*+}(\partial_Nu_{m+}+i\xi_mu_{N+})|_+&=0,\nonumber\\
\{2\mu_{*-}\partial_Nu_{N-}-\pi_-\}|_-
-\{2\mu_{*+}\partial_Nu_{N+}
+(\nu_{*+}-\mu_{*+})
(i\xi'\cdot u'_++\partial_Nu_{N+})\nonumber\\
-\rho_{*+}\kappa_{*+}(\partial_N^2-|\xi'|^2)\rho_+ \}|_+&=0,
\nonumber\\
\partial_N \rho_+|_+&=0.\nonumber
\end{alignat}
Set $(f,g)_\pm=\pm\int_0^\infty f(x_N)\overline{g(x_N)}dx_N$ and 
$\| f\|_\pm=(f,f)_\pm^{1/2}$.
Multiplying the equations in (\ref{605}) by 
$\overline{u_{J\pm}}$ and using integration by parts 
and interface conditions in (\ref{605}), we have
\begin{align}\label{606}
0=
&\lambda\Big(\rho_{*+}\sum_{J=1}^{N}\| u_{J+} \|_+^2+\rho_{*-}
\sum_{J=1}^{N}\| u_{J-} \|_-^2\Big)\\
&+\mu_{*+}\Big(\sum_{j,k=1}^{N-1}\|i\xi_k u_{j+} \|_+^2
+\Big\|\sum_{j=1}^{N-1}i\xi_j u_{j+} \Big\|_+^2+\sum_{k=1}^{N-1}\|\partial_N u_{j+} 
+i\xi_j u_{N+} \|_+^2+2\|
\partial_N u_{N+} \|^2_+ \Big)\nonumber\\
&+(\nu_{*+}-\mu_{*+})\|w_+\|_+^2
+\frac{\rho_{*+}\kappa_{*+}}{\lambda}
(\|\partial_N w_+\|_+^2+|\xi'|^2\|w_+\|_+^2) \nonumber\\
&+\mu_{*-}\Big(\sum_{j,k=1}^{N-1}\|i\xi_k u_{j- }\|_-^2
+\|\sum_{j=1}^{N-1}i\xi_j u_{j-} \|_-^2+\sum_{k=1}^{N-1}\|\partial_N u_{j-} 
+i\xi_j u_{N -}\|_-^2
+2\|\partial_N u_{N- }\|^2_- \Big)\nonumber.
\end{align}
Here, we use the identity
\begin{align*}
&\sum_{j,k=1}^{N-1}(i\xi_k u_{j\pm }
+i\xi_j u_{k\pm },i\xi_k u_{j\pm} )_\pm
=\sum_{j,k=1}^{N-1}\|i\xi_k  u_{j\pm} \|_\pm^2
+\Big\|\sum_{j=1}^{N-1}i\xi_j u_{j\pm} \Big\|_\pm^2,\\
&\sum_{j=1}^{N-1}
(\partial_N u_{j\pm} +i\xi_j u_{N\pm} ,\partial_N u_{j\pm} )_\pm
+\sum_{k=1}^{N-1}
(i\xi_k u_{N\pm} +\partial_N u_{k\pm} ,i\xi_k u_{N\pm} )_\pm=\sum_{j=1}^{N-1}\|\partial_N u_{j\pm} 
+i\xi_j u_{N\pm} \|_\pm^2.
\end{align*}
Taking the real part of (\ref{606}) to obtain
\begin{align*}
0=
&({\rm Re}\,\lambda)\Big(\rho_{*+}
\sum_{J=1}^{N}\| u_{J+} \|_+^2
+\rho_{*-}\sum_{J=1}^{N}\| u_{J-} \|_-^2\Big)\\
&+\mu_{*+}\Big(\sum_{j,k=1}^{N-1}\|i\xi_k u_{j+} \|_+^2
+\Big\|\sum_{j=1}^{N-1}i\xi_j u_{j+} \Big\|_+^2+\sum_{k=1}^{N-1}
\|\partial_N u_{j+} +i\xi_j u_{N+} \|_+^2
+2\|\partial_N u_{N+ }\|^2_+ \Big)\\
&+(\nu_{*+}-\mu_{*+})\|w_+\|_+^2
+\frac{\rho_{*+}\kappa_{*+}({\rm Re}\,\lambda)}{|\lambda|^2}
(\|\partial_N w_+\|_+^2+|\xi'|^2\|w_+\|_+^2)  \\
&+\mu_{*-}\Big(\sum_{j,k=1}^{N-1}\|i\xi_k u_{j-} \|_-^2
+\Big\|\sum_{j=1}^{N-1}i\xi_ju_{j-}\Big\|_-^2+\sum_{k=1}^{N-1}
\|\partial_N u_{j- }+i\xi_j u_{N- }\|_-^2
+2\|\partial_Nu_{N-}\|^2_- \Big) ,
\end{align*}
which, combined with the inequality:
\begin{align*}
\|w_+\|_+^2
\leq&\sum_{j,k=1}^{N-1}\|i\xi_k  u_{j+} \|_+^2
+\Big\|\sum_{j=1}^{N-1}i\xi_j u_{j+ }\Big\|_+^2
+2\|\partial_N u_{N+ }\|^2_+,
\end{align*}
furnishes that	
\begin{align}\label{607}
0\geq
&({\rm Re}\,\lambda)\Big(\rho_{*+}\sum_{J=1}^{N}
\| u_{J+} \|_+^2
+\rho_{*-}\sum_{J=1}^{N}
\| u_{J-} \|_-^2\Big) +\nu_{*+}\|w_+\|_+^2 \\
&+\mu_{*+}\sum_{k=1}^{N-1}\|\partial_N u_{j+} 
+i\xi_j u_{N+} \|_+^2
+\mu_{*-}\sum_{k=1}^{N-1}
\|\partial_N u_{j- }+i\xi_j u_{N- }\|_-^2 \nonumber\\
&+\frac{\rho_{*+}\kappa_{*+}{\rm Re}\,\lambda}{|\lambda|^2}
(\|\partial_N w_+\|_+^2+|\xi'|^2\|w_+\|_+^2) \nonumber.
\end{align}
When ${\rm Im}\,\lambda=0$, we have $\lambda>0$ 
because $\lambda\in \Sigma_\varepsilon$.
However, this contradict to (\ref{607}).
Summing up, we have 
$\mathfrak{l}(\widetilde{\lambda},\widetilde{\xi'})\neq 0 $ 
for $(\widetilde{\lambda},\widetilde{A})\in
D(R_1,R_2)$,
where 
\begin{align*}
D(R_1,R_2)
=\{(\widetilde{\lambda},\widetilde{A}) \mid 
(1+R_1)^{-2}\leq |\widetilde{\lambda}|
\leq R_2^2(1+R_2)^2,\enskip
(1+R_2)^{-1}\leq \widetilde{A}\leq R_1(1+R_1)^{-1} ,\enskip
{\rm Re}\,\widetilde{\lambda}\geq 0 \}.
\end{align*}
Then, there exists a positive constant $C_3$ such that
\begin{align*}
\inf_{(\widetilde{\lambda},\widetilde{A})\in D(R_1,R_2)}|
\mathfrak{l}(\widetilde{\lambda},\widetilde{\xi'})|=2C_3>0
\end{align*}
because $D(R_1,R_2)$ is compact.
Since 
$\mathfrak{l}(\widetilde{\lambda},\widetilde{\xi'})$ is
continuous in 
$\Sigma_{\varepsilon,\lambda_0}\times\mathbb{R}^{N-1}$
for $0<\varepsilon<\pi/2$ and
$D_\varepsilon(R_1,R_2)$ is compact,
$\mathfrak{l}(\widetilde{\lambda},\widetilde{\xi'})$ is
uniformly continuous in 
$D_\varepsilon(R_1,R_2)\times\mathbb{R}^{N-1}$
for $0<\varepsilon<\pi/2$.
Then, there exist a constant 
$\varepsilon_*\in(0,\pi/2)$ such that
\begin{align*}
|\mathfrak{l}(\widetilde{\lambda},\widetilde{\xi'})|\geq C_3
\quad\text{for}\quad (\widetilde{\lambda},\widetilde{A})\in
D_\varepsilon(R_1,R_2), \enskip 
\varepsilon\in (\varepsilon_*,\pi/2),
\end{align*}
which, combined with (\ref{604}), furnishes that
\begin{align}\label{610}
|\mathfrak{l}(\lambda,\xi')|\geq C_3(|\lambda|^{1/2}+A)^6
\end{align}
provided that  
$R_2^{-1}|\lambda|^{1/2}\leq A\leq R_1|\lambda|^{1/2}$ and 
$\lambda\in\Sigma_{\varepsilon}$.

Summing up, setting $C=\min(C_1,C_2,C_3)$, 
by (\ref{608}), (\ref{609}), and (\ref{610}), 
we have (\ref{602}),
which completes the proof of Lemma \ref{lem601}.	
\end{proof}




\section{Problem with surface tension and height function}
\label{Sect.surface}

In this final section, we consider the following problem:
\begin{align}\label{701}
\lambda\rho_++\rho_{*+}{\rm div}\,\mathbf{u}_+&=0 
&\text{ in $\mathbb{R}^N_+$},\\
\rho_{*+}\lambda\mathbf{u_+}-\mu_{*+}\Delta\mathbf{u_+}
-\nu_{*+}\nabla{\rm div}\,\mathbf{u_+}
-\kappa_{*+}\Delta\nabla\rho_+&=0
&\text{ in $\mathbb{R}^N_+$}, \nonumber\\
{\rm div}\,\mathbf{u}_-=0,\quad 
\rho_{*-}\lambda\mathbf{u}_--\mu_{*-}\Delta\mathbf{u}_-
+\nabla \pi_-&=0 
&\text{ in $\mathbb{R}^N_-$}, \nonumber\\
\mu_{*-}D_{mN}(\mathbf{D(u_-)})|_-
-\mu_{*+}D_{mN}(\mathbf{u}_+)|+&=0,	\nonumber\\
\{\mu_{*-}D_{NN}(\mathbf{u}_-)-\pi_- \}|_-&
=\sigma_-\Delta'H, \nonumber\\
\{\mu_{*+}D_{NN}(\mathbf{u}_+)+(\nu_+{*+}-\mu_{*+})
{\rm div}\,\mathbf{u}_+
+\kappa_{*+}\Delta\nabla\rho_+ \}|_+&
=\sigma_+\Delta'H, \nonumber\\
u_{m-}|_--u_{m+}|_+=0,\quad \partial_N\rho_+|_+&=0,
\nonumber\\
\lambda H
-\Big(\frac{\rho_{*-}}{\rho_{*-}-\rho_{*+}}u|_{N-}|_-
- \frac{\rho_{*+}}{\rho_{*-}-\rho_{*+}}u_{N+}|_+ \Big) &=d,
\nonumber
\end{align}
where $\sigma_\pm=\rho_{*\pm}\sigma/(\rho_{*-}-\rho_{*+})$, and
prove the following theorem.
\begin{theo}\label{Th:sect.6}
Let $1<q<\infty$ and $\varepsilon_* <\varepsilon<\pi/2$.
Assume that $\rho_{*+}\neq\rho_{*-}$, $\eta_*\neq 0$,
and $\kappa_{*+}\neq \mu_{*+}\nu_{*+}$.
Then, there exist a positive constant $\lambda_0$ 
and operator families
$\mathcal{A}^4_\pm(\lambda)$, $\mathcal{B}^4_+(\lambda)$,
$\mathcal{P}^4_-(\lambda)$, and $\mathcal{H}(\lambda)$
with
\begin{align*}
\mathcal{A}^4_{\pm}(\lambda)\in&
{\rm Anal}\,(\Sigma_{\varepsilon,\lambda_0},
\mathcal{L}(W^2_q(\mathbb{R}^N),W^2_q(\mathbb{R}^N_\pm)^N),\\
\mathcal{B}^4_{+}(\lambda)\in&
{\rm Anal}\,(\Sigma_{\varepsilon,\lambda_0},
\mathcal{L}(W^2_q(\mathbb{R}^N),W^3_q(\mathbb{R}^N_+)),\\
\mathcal{P}^4_{-}(\lambda)\in &
{\rm Anal}\,(\Sigma_{\varepsilon,\lambda_0},\mathcal{L}
(W^2_q(\mathbb{R}^N),\hat{W}^1_q(\mathbb{R}^N_-))),\\
\mathcal{H}(\lambda)\in&
{\rm Anal}\,(\Sigma_{\varepsilon,\lambda_0},
\mathcal{L}(W^2_q(\mathbb{R}^N),W^3_q(\mathbb{R}^N))),	
\end{align*} 	
such that for any $\lambda\in \Sigma_{\varepsilon,\lambda_0}$ 
and 
$d\in W^2_q(\mathbb{R}^N)$,
$\mathbf{u}_\pm=\mathcal{A}^4_\pm(\lambda) d$,
$\rho_+=\mathcal{B}^4_+(\lambda) d$,
$\pi_-=\mathcal{P}^4_-(\lambda) d$,
and $H=\mathcal{H}(\lambda)d$ 
are solutions of problem (\ref{701}).
Furthermore, for $s=0,1$, we have
\begin{align*}	
\mathcal{R}_{\mathcal{L}(W^2_q(\mathbb{R}^N),
L_q(\mathbb{R}^N)^{N^3+N^2+N})}
(\{(\tau\partial_\tau)^s(G_\lambda^1 \mathcal{A}^3_\pm(\lambda))
\mid \lambda \in \Sigma_{\varepsilon,\lambda_0}\})
\leq& c_0,\\
\mathcal{R}_{\mathcal{L}(W^2_q(\mathbb{R}^N),
L_q(\mathbb{R}^N)^{N^3+N^2}\times W^1_q(\mathbb{R}^N))}
(\{(\tau\partial_\tau)^s(G_\lambda^2 \mathcal{B}^3_+(\lambda))
\mid \lambda \in \Sigma_{\varepsilon,\lambda_0}\})
\leq& c_0,\\
\mathcal{R}_{\mathcal{L}
(W^2_q(\mathbb{R}^N),L_q(\mathbb{R}^N)^{N})}
(\{(\tau\partial_\tau)^s(\nabla \mathcal{P}^3_-(\lambda))
\mid \lambda  \in \Sigma_{\varepsilon,\lambda_0}\})
\leq& c_0,\\
\mathcal{R}_{\mathcal{L}
(W^2_q(\mathbb{R}^N),W^2_q(\mathbb{R}^N)^{N+1})}
(\{(\tau\partial_\tau)^s(G_\lambda^3 \mathcal{H} (\lambda))
\mid \lambda  \in \Sigma_{\varepsilon,\lambda_0}\})
\leq& c_0,
\end{align*}
with some constant $c_0$.	
\end{theo}
\begin{rema}
Combining Theorem \ref{Th:sect.6} with Theorem \ref{Th03},
Theorem \ref{Th04}, and Theorem \ref{Th05},
we have Theorem \ref{Th02} immediately.
\end{rema}
As discussed in Sect. \ref{Sect.solution},
applying the partial Fourier transform to (\ref{701}),
we have	
\begin{align*}
\lambda\widehat{\rho}_++\rho_{*+}
\widehat{{\rm div}\,\mathbf{u}_+}&=0
&\text{ for $x_N>0$},\\
\rho_{*+}\lambda\widehat{u}_{j+}-\mu_{*+}
(\partial_N^2-|\xi'|^2)\widehat{u}_{j+}
-\nu_{*+}i\xi_j\widehat{{\rm div}\,\mathbf{u_+}}&\\
\qquad-i\xi_j\{\kappa_{*+}(\partial_N^2-|\xi'|^2) \}
\widehat{\rho}_+&=0&\text{ for $x_N>0$},\nonumber\\
\rho_{*+}\lambda\widehat{u}_{N+}
-\mu_{*+}(\partial_N^2-|\xi'|^2)\widehat{u}_{N+}
-\nu_{*+}\partial_N\widehat{{\rm div}\,\mathbf{u_+}}&\\
\qquad-\partial_N\{
\kappa_{*+}(\partial_N^2-|\xi'|^2) \}\widehat{\rho}_+&=0
&\text{ for $x_N>0$},\nonumber\\
\widehat{{\rm div}\,\mathbf{u}_-}&=0&\text{ for $x_N<0$},\\
\rho_{*-}\lambda\widehat{u}_{j-}-\mu_{*-}(\partial_N^2-|\xi'|^2)
\widehat{u}_{j-}-i\xi_j\widehat{\pi}_-&=0
&\text{ for $x_N<0$},\\
\rho_{*-}\lambda\widehat{u}_{N-}
-\mu_{*-}(\partial_N^2-|\xi'|^2)\widehat{u}_{N-}
-\partial_N\widehat{\pi}_-&=0&\text{ for $x_N<0$},
\end{align*} 	
with the interface condition:
\begin{align*}
\mu_{*-}(\partial_N\widehat{u}_{m-}+i\xi_m\widehat{u}_{N-})|_--
\mu_{*+}(\partial_N\widehat{u}_{m+}+i\xi_m\widehat{u}_{N+})|_+
&=0,\\	
\{2\mu_{*-}\partial_N\widehat{u}_{N-}-\widehat{\pi}_- \}|_-
&=\sigma_-|\xi'|^2\widehat{H}(0),\\
\{2\mu_{*+}\partial_N\widehat{u}_{N+}
+(\nu_{*+}-\mu_{*+})\widehat{{\rm div}\,\mathbf{u_+}}
\qquad&\\
+\kappa_{*+}(\partial_N^2-|\xi'|^2)
\widehat{\rho}_+ \}|_+
&=\sigma_+|\xi'|^2\widehat{H}(0),\nonumber\\
\widehat{u}_{m-}|_--\widehat{u}_{m+}|_+&=0,\\
\partial_N\widehat{\rho}_+&=0,
\end{align*}
and the resolvent equation for $H$:	
\begin{align}\label{resolvent:H}
\lambda \widehat{H}(0)
-\Big(\frac{\rho_{*-}}{\rho_{*-}-\rho_{*+}}\widehat{u}_{N-}|_-
-\frac{\rho_{*+}}{\rho_{*-}-\rho_{*+}}\widehat{u}_{N+}|_+\Big)
=\widehat{d}(0).
\end{align}	
Our task is to represent $\widehat{H}$ in terms of 
$\widehat{d}(0)$, so that we look for solutions 
$\widehat{u}_{J\pm}$ and $\widehat{\pi}_-$ of 
the form (\ref{4.428}) with $g_j=h_j=k=0$.
In view of (\ref{417}), (\ref{418}), and (\ref{419}),
when $g_j=h_j=k=0$, we have	
\begin{align*}
\widehat{u}^+_{N+}(0)=\alpha_{N+}=
&AR^+_{NN}\widehat{H}(0),\\
\widehat{u}^-_{N-}(0)=\alpha_{N-}=
&AR^-_{NN}\widehat{H}(0).\nonumber
\end{align*}
Inserting these formulas into (\ref{resolvent:H}),
we have
\begin{align}\label{eq:K}
(\lambda+K_H)\widehat{H}(0)=\widehat{d}(0)
\end{align}
with	
\begin{align}\label{Def:K}
K_H=\frac{\rho_{*-}}{\rho_{*-}-\rho_{*+}}
AR^-_{NN}
-\frac{\rho_{*+}}{\rho_{*-}-\rho_{*+}}
AR^+_{NN}.
\end{align}
We now prove the following lemma.
\begin{lemm}\label{lem:K}
Let $\varepsilon_*<\varepsilon<\pi/2$ and 
let $K_H$ be the function defined in (\ref{Def:K}).
Assume that $\rho_{*+}\neq\rho_{*-}$, $\eta_*\neq 0$,
and $\kappa_{*+}\neq \mu_{*+}\nu_{*+}$.
Then there exists a positive constant $\lambda_0$ depending on		
$\varepsilon$, $\mu_{*\pm}$, $\nu_{*+}$,
$\kappa_{*+}$, and $\rho_{*\pm}$ such that	
\begin{align}\label{702}
|\partial_{\xi'}^{\alpha'}
\{(\tau\partial_\tau)^s(\lambda+K_H)^{-1} \}|
\leq C_{\alpha'}(|\lambda|^{1/2}+A)^{-1}
A^{-|\alpha'|}\qquad(s=0,1)
\end{align}	
for any multi-index $\alpha'\in\mathbb{N}^{N-1}_0$ and 
$(\lambda,\xi')\in \widetilde{\Sigma}_{\varepsilon,\lambda_0}$
with some constant $C_{\alpha'}$ depending on
$\alpha'$, $\lambda_0$, $\varepsilon$, $\mu_{*\pm}$, $\nu_{*+}$,
$\kappa_{*+}$, and $\rho_{*\pm}$.
\end{lemm}
\begin{proof}	
To prove (\ref{702}) with $\alpha'=0$ and $s=0$,
first we consider the case where $R_1|\lambda|^{1/2}\leq A$
with large $R_1$.
In the following, $\delta_1$ is the same small number as in the
proof of Lemma. \ref{lem601}.
In this case, we see that
\begin{align*}
AR_{NN}^+=-\frac{\sigma_+}{2\mu_{*+}}A(1+O(\delta_1)),\quad
AR_{NN}^-=\frac{\sigma_-}{2\mu_{*-}}A(1+O(\delta_1)).
\end{align*}
Then, we have	
\begin{align}\label{703}
K_H&=\Big[
\Big(\frac{\rho_{*+}}{\rho_{*-}-\rho_{*+}} \Big)^2
\frac{\sigma}{2\mu_{*+}}
+\Big(\frac{\rho_{*-}}{\rho_{*-}-\rho_{*+}} \Big)^2
\frac{\sigma}{2\mu_{*-}}
\Big]A(1+O(\delta_1))\\
&=:\omega_3 A(1+O(\delta_1))\nonumber\quad (\omega_3>0).
\end{align}	
Since $\lambda\in \Sigma_\varepsilon$, 
by (\ref{501}) and (\ref{703})
we have
\begin{align*}
|\lambda+K_H|\geq\Big(\sin\frac{\varepsilon}{2}\Big)
\Big(|\lambda|^{1/2}+\omega_3A\Big)-\omega_3 A O(\delta_1).
\end{align*}
If we choose $\delta_1$ so small that $O(\delta_1)
\leq\sin(\varepsilon/2)/2$, we have
\begin{align}\label{704}
|\lambda+K_H|\geq
\Big(\frac{1}{2}\sin\frac{\varepsilon}{2}\Big)
\Big(|\lambda|^{1/2}+\omega_3A\Big)
\end{align}	
provided that $R_1|\lambda|^{1/2}\leq A$ with
large $R_1>0$ and $\lambda\in \Sigma_\varepsilon$.

Next we consider the case where $A\leq R_1|\lambda|^{1/2}$.
Here and in the sequel, $C$ denotes a generic constant
depending on $R_1$, $\varepsilon$, $\mu_{*\pm}$, $\nu_{*+}$,
$\kappa_{*+}$, and $\rho_{*\pm}$.
From (\ref{442}), (\ref{det:L}), (\ref{446}), and (\ref{448}),
we have	
\begin{align*}
|Q_{NN,1}^+|\leq C|\lambda|^{1/2},\quad
|Q_{NN,2}^+|\leq C|\lambda|^{1/2}.
\end{align*}
Then, by (\ref{452}) and (\ref{455}),
we easily see
\begin{align*}
|R_{NN}^+|\leq C,\quad |R_{NN}^-|\leq C,
\end{align*}
which, combined with (\ref{Def:K}), furnishes that
\begin{align*}
|K_H|\leq C|\lambda|^{1/2}
\end{align*}
for any $(\lambda,\xi')\in \widetilde{\Sigma}_{\varepsilon,0}$
provided that $A\leq R_1 |\lambda|$.
Thus we have 
$|\lambda+K_H|\geq |\lambda|^{1/2}(|\lambda|^{1/2}-C)$.
Consequently, when we take $\lambda_0>0$ so large that
$\lambda_0^{1/2}/2\geq C$,
we have
\begin{align*}
|\lambda+K_H|\geq \frac{1}{2}|\lambda|
\end{align*}	
for any $(\lambda,\xi')\in 
\widetilde{\Sigma}_{\varepsilon,\lambda_0}$ provided that
$A\leq R_1|\lambda|$.
Since $A\leq R_1|\lambda|^{1/2}$, we have
\begin{align*}
|\lambda+K_H|\geq
\frac{1}{4}|\lambda|+\frac{1}{4}|\lambda|
\geq \frac{1}{4}|\lambda|
+\frac{\lambda_0^{1/2}}{4}|\lambda|^{1/2}
\geq\frac{1}{4}\Big(|\lambda|+\frac{\lambda_0^{1/2}A}{R_1}\Big).
\end{align*}	
Choosing $R_1$ so large that 
$\lambda_0^{1/2}R_1^{-1}\leq \omega_3$, we obtain
\begin{align}\label{705}
|\lambda+K_H|\geq \frac{1}{4}
\Big(|\lambda|^{1/2}+\omega_3 A \Big)
\end{align}
for any $(\lambda,\xi')\in \widetilde{\Sigma}_{\varepsilon,0}$	
provided that $A\leq R_1|\lambda|$.
Accordingly, by 
$\Sigma_{\varepsilon,\lambda_0}\subset\Sigma_{\varepsilon}$,
combining (\ref{704}) and (\ref{705}), we obtain
\begin{align*}
|\lambda+K_H|\geq \omega_4 (|\lambda|^{1/2}+A)
\end{align*}
for any $(\lambda,\xi')\in 
\widetilde{\Sigma}_{\varepsilon,\lambda_0}$ with
\begin{align*}
\omega_4=\min \Big(\frac{1}{4},\frac{\omega_3}{4},
\frac{\sin(\varepsilon/2)}{2},
\frac{\omega_3\sin(\varepsilon/2)}{2}\Big).
\end{align*}	
		
Finally, we prove (\ref{702}) for any multi-index 
$\alpha'\in \mathbb{N}^{N-1}_0$.
By Lemma \ref{lem:multiplier}, (\ref{4.429}), 
and Lemma \ref{lem601},
we have $K_H\in \mathbb{M}_{1,2,\varepsilon,0}$, so that
by the Bell's formula (\ref{Bell}) with $f(t)=(\lambda+t)^{-1}$
and $g=K_H$, we have	
\begin{align*}
|\partial^{\alpha'}_{\xi'}(\lambda+K_H)^{-1}|
\leq C_{\alpha'}\sum_{l=1}^{|\alpha'|}
|\lambda+K_H|^{-(l+1)}(|\lambda|^{1/2}+A)^l A^{-|\alpha'|}
\leq C_{\alpha'}(|\lambda|^{1/2}+A)^{-1}A^{-|\alpha'|}.
\end{align*}
Analogously, we have
\begin{align*}
|\partial^{\alpha'}_{\xi'}
(\tau\partial_\tau(\lambda+K_H)^{-1})|
\leq C_{\alpha'}(|\lambda|^{1/2}+A)^{-1}A^{-|\alpha'|}.
\end{align*}
Summing up, we have (\ref{702}).		
\end{proof}	
From (\ref{Def:K}) and Lemma \ref{lem:K}, we have
\begin{align*}
\widehat{H}(\xi',0)
=(\lambda+K_H)^{-1}\widehat{d}(\xi',0),
\end{align*}
so that we define $\widehat{H}(\xi',x_N)$ by 
$\widehat{H}(\xi',x_N)
=e^{-(1+A^2)^{1/2}x_N}(\lambda+K_H)^{-1}\widehat{d}(\xi',0)$.
The following lemma was proved in Shibata \cite{Shibata2016}.
\begin{lemm}\label{Operator:H}
Let $1<q<\infty,\enskip \varepsilon_*<\varepsilon<\pi/2$ and 
let $\lambda_0$ be the same constant as in Lemma \ref{lem:K}.
Assume that $\rho_{*+}\neq\rho_{*-}$, $\eta_*\neq 0$,
and $\kappa_{*+}\neq \mu_{*+}\nu_{*+}$.
Given that the operator $\widetilde{\mathcal{H}}(\lambda)$
is defined by	
\begin{align*}
[\widetilde{\mathcal{H}}(\lambda)d](x)
=\mathcal{F}^{-1}_{\xi'}
[e^{-(1+A^2)^{1/2}x_N}(\lambda+K_H)^{-1}\widehat{d}(\xi',0)]
(x') \qquad\text{ for $x'\in\mathbb{R}^N_+$}
\end{align*}	
for any	$d\in W^2_q(\mathbb{R}^N)$, 
$\widetilde{\mathcal{H}}(\lambda)\in
{\rm Anal}\,(\Sigma_{\varepsilon,\lambda_0},
\mathcal{L}(W^2_q(\mathbb{R}^N),W^3_q(\mathbb{R}^N_+))$,
and		
\begin{align*}
\mathcal{R}_
{\mathcal{L}(W^2_q(\mathbb{R}^N),W^3_q(\mathbb{R}^N_+))}
(\{(\tau\partial_\tau)^s (\lambda,\nabla) 
\widetilde{\mathcal{H}}(\lambda)\mid
\lambda\in \Sigma_{\varepsilon,\lambda_0} \})\leq \gamma
\quad(s=0,1)
\end{align*}
with some constant $\gamma$ depending on 
$\lambda_0$, $\varepsilon$, $\mu_{*\pm}$, $\nu_{*+}$,
$\kappa_{*+}$, and $\rho_{*\pm}$.
\end{lemm}
We extend 
$\widetilde{\mathcal{H}}(\lambda) d$ to $x_N<0$,
namely, we define $\mathcal{H}(\lambda)$ by		
\begin{align*}
[\mathcal{H}(\lambda)d](x)=
\begin{cases}
[\widetilde{\mathcal{H}}(\lambda)d](x)
& (x_N>0),\\\displaystyle
\sum_{j=1}^{4}a_j[\widetilde{\mathcal{H}}(\lambda) d]
(x',-jx_N)
& (x_N<0),
\end{cases}
\end{align*}		
where $a_j$ are constants satisfying the equations:
$\sum_{j=1}^{4}a_j(-j)^k=1$ for $k=0,1,2,3$.
By Lemma \ref{Operator:H}, we have the following corollary
of Lemma \ref{Operator:H} immediately.
\begin{corr}
Let $1<q<\infty$, $\varepsilon_*<\varepsilon<\pi/2$ 
and let $\lambda_0$
be the same constant as in Lemma \ref{lem:K}.
Assume that $\rho_{*+}\neq\rho_{*-}$, $\eta_*\neq 0$,
and $\kappa_{*+}\neq \mu_{*+}\nu_{*+}$.
Then, there exists an operator family
\begin{align*}
\mathcal{H}(\lambda)\in
{\rm Anal}\,(\Sigma_{\varepsilon,\lambda_0},
\mathcal{L}(W^2_q(\mathbb{R}^N),W^3_q(\mathbb{R}^N))),
\end{align*}		
such that for any $\lambda\in \Sigma_{\varepsilon,\lambda_0}$
and $d\in W^2_q(\mathcal{R}^N)$,
$\mathcal{F}^{-1}_{\xi'}[\widehat{d}(\xi',0)/(\lambda+K_H)](x')
=\mathcal{H}(\lambda)d|_{x_N=0} $
and
\begin{align*}
\mathcal{R}_{\mathcal{L}(W^2_q(\mathbb{R}^N),
W^2_q(\mathbb{R}^N)^{N+1})}
(\{(\tau\partial_\tau)^s(\lambda,\nabla)
\mathcal{H}(\lambda)\mid\lambda\in 
\Sigma_{\varepsilon,\lambda_0} \})
\leq c_0 \qquad (s=0,1)
\end{align*}	
with some constant $c_0$ depending on
$\lambda_0$, $\varepsilon$, $\mu_{*\pm}$, $\nu_{*+}$,
$\kappa_{*+}$, and $\rho_{*\pm}$.
\end{corr}
Then, we construct solution operators of (\ref{701}).
Using (\ref{4.428}) and (\ref{eq:K}), we have	
\begin{align*}
\widehat{u}_{J+}
=&\frac{A Q_{JN,1}^+}{\lambda+K_H}AM_{1+}(x_N)\widehat{d}(0)
+\frac{A Q_{JN,2}^+}{\lambda+K_H}AM_{2+}(x_N)\widehat{d}(0)
+\frac{AR_{JN}^+}{\lambda+K_H}Ae^{-B_+x_N}\widehat{d}(0),\\
\widehat{u}_{J-}
=&\frac{AQ_{JN}^-}{\lambda+K_H} AM_{-}(x_N)\widehat{d}(0)
+\frac{AR_{JN}^-}{\lambda+K_H}Ae^{B_-x_N}\widehat{d}(0),\\
\widehat{\rho}_+
=&\frac{AP_{N,1}^+}{\lambda+K_H}AM_{0+}(x_N)\widehat{d}(0)
+\frac{AP_{N,2}^+}{\lambda+K_H}Ae^{-t_1x_N}\widehat{d}(0)\\
\widehat{\pi}_-
=&\frac{AP_N^-}{\lambda+K_H} e^{Ax_N}\widehat{d}(0),
\end{align*}
which yields
\begin{align*}	
u_{J+}
=&\mathcal{F}^{-1}_{\xi'}
\Big[\frac{A Q_{JN,1}^+}{\lambda+K_H}A M_{1+}(x_N)\widehat{d}(0)
\Big](x')
+\mathcal{F}^{-1}_{\xi'}
\Big[\frac{A Q_{JN,2}^+}{\lambda+K_H}A M_{2+}(x_N)\widehat{d}(0)
\Big](x')\\
&+\mathcal{F}^{-1}_{\xi'}
\Big[\frac{AR_{JN}^+}{\lambda+K_H}Ae^{-B_+x_N}\widehat{d}(0)
\Big](x')
=: \mathcal{A} ^{4+}_J(\lambda) d,\\
u_{J-}
=&\mathcal{F}^{-1}_{\xi'}
\Big[\frac{AQ_{JN}^-}{\lambda+K_H} AM_{-}(x_N)\widehat{d}(0)
\Big](x')
+\mathcal{F}^{-1}_{\xi'}
\Big[\frac{AR_{JN}^-}{\lambda+K_H}Ae^{B_-x_N}\widehat{d}(0)
\Big](x')
=: \mathcal{A} ^{4-}_J(\lambda) d,\\
\rho_+
=&\mathcal{F}^{-1}_{\xi'}
\Big[\frac{AP_{N,1}^+}{\lambda+K_H}A M_{0+}(x_N)\widehat{d}(0)
\Big](x')
+\mathcal{F}^{-1}_{\xi'}
\Big[\frac{AP_{N,2}^+}{\lambda+K_H}Ae^{-t_1x_N}\widehat{d}(0)
\Big](x')
=: \mathcal{B} ^{4}_+(\lambda) d,\\
\pi_-
=&\mathcal{F}^{-1}_{\xi'}
\Big[\frac{AP_N^-}{\lambda+K_H} e^{Ax_N}
\Big](x')
=: \mathcal{P} ^{4}_-(\lambda) d.
\end{align*}
Accordingly, if we the define operator
$\mathcal{A}^4_\pm(\lambda)$ by
\begin{align*}
\mathcal{A}^4_\pm d
=(\mathcal{A}^{4\pm}_1,\dots,\mathcal{A}^{4\pm}_N) d,
\end{align*}
by (\ref{4.429}), Corollary \ref{cor.operator},
Lemma \ref{lem:K}, and the argument in Shibata 
\cite[Sect. 6]{Shibata2016}, we have Theorem \ref{Th:sect.6},
which completes the proof of Theorem \ref{Th02}.
		
\section{Acknowledgment}
The author would like to thank Prof. Yoshihiro Shibata and
Dr. Hirokazu Saito for stimulating discussions
on the subject of this paper.


\begin{thebibliography}{99}
\bibitem{Abe2007}H. Abels,
{\it On generalized solutions of two-phase flows for viscous
incompressible fluids},
Interfaces Free Bound. {\bf 9} (2007), no.1, 31-65.
\bibitem{CH1958}J. W. Cahn and J. E. Hilliard, 
{\it Free energy of a nonuniform system. I. 
Interfacial free energy},
J. Chem. Phys. {\bf 28} (1958), 258-266.
\bibitem{Den1991}I. V. Denisova,
{\it A priori estimates for the solution of the linear nonstationary
problem connected with the motion of a drop in a liquid medium},
Trudy Mat. Inst. Steklov. {\bf 188} (1990), 3-21 (in Russian).;
English transl.: Proc. Steklov Inst. Math. {\bf 1991} (1991), no. 3, 1-24.
\bibitem{Den1994}I. V. Denisova,
{\it Problem of the motion of two viscous incompressible fluids
separated by a closed free interface},
Acta Appl. Math. {\bf 37} (1994), no. 1-2, 31-40.
\bibitem{Den2015}I. V. Denisova,
{\it On energy inequality for the problem on the evolution of two 
fluids of different types without surface tension},
J. Math. Fluid Mech. 17 (2015), no. 1, 183-198. 
\bibitem{DS1991}I. V. Denisova and V. A. Solonnikov,
{\it Solvability in H\"{o}lder spaces of a model initial-boundary
value problem generated by a problem on the motion of two fluids},
Zap. Nauchn. Sem. Leningrad. Otdel. Mat. Inst. Steklov. (LOMI) {\bf 188} (1991)
(in Russian).; English transl.: J. Math. Sci. {\bf 70} (1991), no. 3, 1717-1746. 
\bibitem{DS1995}I. V. Denisova and V. A. Solonnikov,
{\it Classical solvability of the problem of the motion of two 
viscous incompressible fluids},
Algebra i Analiz {\bf 7} (1995), no. 5, 101-142 (in Russian).;
English transl. St. Petersburg Math. J. {\bf 7} (1995), no. 5, 755-786.
\bibitem{DS2017}I. V. Denisova and V. A. Solonnikov,
{\it Local and global solvability of free Boundary problems for the 
compressible Navier?Stokes equations near equilibria}, in press.
\bibitem{DS1985}J. E. Dunn and J. Serrin, 
{\it On the thermomechanics of interstitial working}, 
Arch. Rational Mech. Anal. {\bf 88} (1985), no. 2, 95-133. 
\bibitem{ES2013}Y. Enomoto and Y. Shibata, 
{\it On the $\mathcal{R}$-sectoriality and the initial boundary 
value problem for the viscous compressible fluid flow},
Funkcialaj Ekvacioj \textbf{56} (2013), 441-505.
\bibitem{FK2017}H. Freist\"{u}hler and M. Kotschote,
{\it Phase-field and Korteweg-type models for the 
time-dependent flow of compressible two-phase fluids},
Arch. Rational Mech. Anal. {\bf 224} (2017), no. 1, 1-20.
\bibitem{GK2017}A. N. Gorban and I. V. Karlin, 
{\it Beyond Navier-Stokes equations: capillarity of ideal gas}, 
Comtemporary Physics {\bf 58} (2017), no. 1, 70-90.
\bibitem{GS2014}D. G\"{o}ts and Y. Shibata,
{\it On the R-boundedness of the solution operators in the study 
of the compressible viscous fluid flow with free boundary conditions},
Asymptot. Anal. {\bf 90} (2014), no. 3-4, 207-236.
\bibitem{GT1994}Y. Giga and S. Takahashi,
{\it On global weak solutions of the nonstationary two-phase Stokes flow},
SIAM J. Math. Anal. 25 (1994), no. 3, 876-893.
\bibitem{Korteweg}D. J. Korteweg, 
{\it Sur la forme que prennent les \'{e}quations du 
mouvement des fluides si l'on tient compte des forces 
capillaires caus\'{e}es par des variations de 
densit\'{e} consid\'{e}rables mais continues et sur la 
th\'{e}orie de la capillarit\'{e} dans l'hypothe\`{e}se 
d'une variation continue de la densit\'{e}}, Arch. N\'{e}erl. 
{\bf 6} (1901), no. 2, 1-24.
\bibitem{Kotschote2008}M. Kotshote, 
{\it Strong solutions for a compressible fluid model of 
Korteweg type}, Ann. Inst. H. Poincar\'{e} 
Anal. Non Lin\'{e}aire {\bf 25} (2008), no. 4, 679-696. 
\bibitem{Kotschote2010}M. Kotschote,
{\it Strong well-posedness for a Korteweg type model for the
dynamics of a compressible non-isothermal fluid},
J. math. fluid. mech. \textbf{12} (2010), 473-484.
\bibitem{KSS2014}T. Kubo, Y. Shibata, and K. Soga,
{\it On the $\mathcal{R}$-boundedness for the two phase problem: 
compressible-incompressible model problem},
Bound. Value Probl. {\bf 2014} (2014), 141, 33pp.
\bibitem{MS2017}S. Maryani and H. Saito,
{\it On the $\mathcal{R}$-boundedness of solution operator families 
for two-phase Stokes resolvent equations},
Differential Integral Equations {\bf 30} (2017), no. 1-2, 1-52. 
\bibitem{NP1995}A. Nouri and F. Poupaund,
{\it An existence theorem for the multifluid Navier-Stokes problem},
J. Differential Equations {\bf 122} (1995), no. 1, 71-88. 
\bibitem{PSSS2012}J. Pr{\"u}ss, S. Shimizu, Y. Shibata, 
and G. Simonett, 
{\it On well-posedness of incompressible two-phase flows with 
phase transitions: the case of equal densities}, 
Evol. Equ. Control Theory \textbf{1} (2012), no. 1, 171-194.
\bibitem{PS2010}J. Pr{\"u}ss and G. Simonett,
{\it On the two-phase Navier-Stokes equations with surface tension},
Interfaces Free Bound. {\bf 12} (2010), no. 3, 311-345.
\bibitem{PS2011}J. Pr{\"u}ss and G. Simonett,
{\it Analytic solutions for the two-phase Navier-Stokes equations 
with surface tension and gravity},
Progress in Nonlinear Differential Equations Appl. {\bf 80} (2011),
507-540.
\bibitem{PS2016}J. Pr{\"u}ss and G. Simonett, 
{\it Moving interfaces and quasilinear parabolic 
evolution equation}, Monographs in Mathematics, 
\textbf{105}, Birkh\"auser Verlag, 
Bassel$\cdot$Boston$\cdot$Berilin, 2016
\bibitem{Saito}H. Saito, 
{\it Compressible fluid model of Korteweg type with 
free boundary condition: model problem}, 
preprint 2017, arXiv:1705.00603.
\bibitem{Shibata2014a}Y. Shibata, 
{\it On the $\mathcal{R}$-boundedness of solution operators 
for the Stokes equations with free boundary condition}, 
Differential Integral Equations {\bf 27} (2014), 313-368.
\bibitem{Shibata2014}Y. Shibata, 
{\it On the 2 phase problem including the phase transition}, 
Abstract of the Sapporo symposium, 2013. 
http://www.math.sci.hokudai.ac.jp/sympo/sapporo/abst2014.html
\bibitem{Shibata2016}Y. Shibata, 
{\it On the ${\mathcal R}$-boundedness for the two phase 
problem with phase transition: compressible-incompressible 
model problem}, 
Funkcialaj Ekvacioj \textbf{59} (2016), 243-287.
\bibitem{Shibata2016a}Y. Shibata,
{\it On the {$\mathcal R$}-bounded solution operator and the
maximal	{$L_p$}-{$L_q$} regularity of the {S}tokes equations
with free boundary condition},
Mathematical fluid dynamics, present and future,
Springer Proc. Math. Stat. {\bf 183} (2016), 203-285.
\bibitem{SS2011}Y. Shibata and S. Shimizu, 
{\it Maximal $L_p$-$L_q$ regularity for the two-phase Stokes 
equations; model problems},
J. Differential Equations {\bf 251} (2011), no. 2, 373-419.
\bibitem{SS2012}Y. Shibata and S. Shimizu, 
{\it On the maximal $L_p\mathchar`-L_q$ regularity of 
the Stokes problem with first order boundary condition; 
model problems}, J. Math. Soc. Japan {\bf 64} (2012), 
no. 2, 561-626.
\bibitem{Tsuda2016}K. Tsuda,
{\it Existence and stability of time periodic solution to the compressible
Navier-Stokes-Korteweg system on $\BR^3$}, 
J. Math. Fluid Mech. {\bf 18} (2016), 157-185.
\bibitem{vanderWaals}J. D. van der Waals,
{\it The thermodynamic theory of capillary under the 
hypothesis of a continuous variation of density},
Verhandel. Konink. Akad. Weten. Amsterdam (Sect. 1)
{\bf 1} (1893), no. 8, 42-97 (in Dutch).;
English transl.: Journal of Statistical Physics
{\bf 20} (1979), no. 2, 200-244.
\bibitem{Volevich}L. R. Volevich, 
{\it Solubility of boundary value problems for general 
elliptic systems}, Mat. Sb. {\bf 68} (1965), 
373-416 (in Russian).; English transl.: Amer. 
Math. Soc. Transl., Ser. 2 {\bf 67} (1968), 182-225.
\bibitem{Weis2001}L. Weis, 
{\it Operator-valued Fourier multiplier theorems and maximal 
$L_p$-regularity}, Math. Ann. {\bf 319} (2001), no. 4, 735-758.	
\end{thebibliography}
\end{document}